\documentclass[reqno,dvipsnames]{amsart}
\usepackage[utf8]{inputenc}

\usepackage{hyperref}
\usepackage{amsmath, amssymb}
\usepackage{amsthm}
\usepackage[font=footnotesize]{caption}
\usepackage{bbm}
\usepackage{diagbox, hhline, booktabs}
\usepackage{multicol}
\usepackage[bottom=1in]{geometry}
\usepackage{youngtab}
\usepackage{arydshln}
\usepackage[normalem]{ulem} %option normalem means \emph -> italics rather than underline
\usepackage{enumitem}

\newtheorem{theorem}{Theorem}[section]
\newtheorem{proposition}[theorem]{Proposition}
\newtheorem{lemma}[theorem]{Lemma}
\newtheorem{corollary}[theorem]{Corollary}
\newtheorem{conjecture}[theorem]{Conjecture}
\newtheorem{definition}[theorem]{Definition}
\newtheorem{notation}[theorem]{Notation}

\newtheorem*{thmA}{Theorem A}
\newtheorem*{thmB}{Theorem B}
\newtheorem*{thmC}{Theorem C}

\theoremstyle{definition}
\newtheorem{example}[theorem]{Example}

\theoremstyle{remark}
\newtheorem{remark}[theorem]{Remark}

\numberwithin{equation}{section}

\usepackage{tikz}
\usetikzlibrary{patterns}

\newcommand{\triv}{\mathbbm{1}}
\newcommand{\Ch}{\operatorname{Char}}
\newcommand{\Infl}{\operatorname{Infl}}
\newcommand{\Irr}{\operatorname{Irr}}
\newcommand{\sgn}{\operatorname{sgn}}

\newcommand{\down}{\big\downarrow}
\newcommand{\up}{\big\uparrow}
\newcommand{\cB}{\mathcal{B}}
\newcommand{\cL}{\mathcal{L}}
\newcommand{\cP}{\mathcal{P}}
\newcommand{\cX}{\mathcal{X}}
\newcommand{\ff}{\mathsf{f}}
\newcommand{\fg}{\mathsf{g}}
\newcommand{\sk}{\mathsf{k}}
\newcommand{\sm}{\mathsf{m}}
\newcommand{\sn}{\mathsf{N}}
\newcommand{\C}{\mathbb{C}}
\newcommand{\N}{\mathbb{N}}
\newcommand{\Z}{\mathbb{Z}}

\allowdisplaybreaks
\title{On plethysms and Sylow branching coefficients}

\author{Stacey Law}
\address[S. Law]{Department of Pure Mathematics and Mathematical Statistics, University of Cambridge, Cambridge CB3 0WB, UK}
\email{swcl2@cam.ac.uk}

\author{Yuji Okitani}
\address[Y. Okitani]{Department of Mathematics, University of California, Berkeley, CA 94720, USA}
\email{yuji\_okitani@berkeley.edu}

\begin{document}
	
\begin{abstract}
	We prove a recursive formula for plethysm coefficients of the form $a^\mu_{\lambda,(m)}$, encompassing those which arise in a long-standing conjecture of Foulkes. This also generalises results on plethysms due to Bruns--Conca--Varbaro and de Boeck--Paget--Wildon. From this we deduce a stability result and resolve two conjectures of de Boeck concerning plethysms, as well as obtain new results on Sylow branching coefficients for symmetric groups for the prime 2. Further, letting $P_n$ denote a Sylow 2-subgroup of $S_n$, we show that almost all Sylow branching coefficients of $S_n$ corresponding to the trivial character of $P_n$ are positive.
\end{abstract}


\maketitle
\baselineskip=15pt

%-----------------------------------------------------------------------------------
\section{Introduction}

Symmetric groups lie at the intersection of a number of central topics of research in the representation theory of finite groups. In this article, we focus on two key themes: plethysms and Sylow branching coefficients.

Plethysm coefficients form an important family of numbers arising in the theory of symmetric functions as the multiplicities $a^\mu_{\lambda,\nu}$ appearing in the decompositions of plethystic products of Schur functions $s_\lambda\circ s_\nu$ into non-negative integral linear combinations of Schur functions $s_\mu$. The setting can be translated to the character theory of symmetric groups using the characteristic map: see Section~\ref{sec:prelims} below, for example. 

Finding a combinatorial rule for plethysm coefficients remains a major open problem in algebraic combinatorics \cite[Problem 9]{Sta-OP}, as does resolving the long-standing conjecture of Foulkes \cite{Foulkes} that the induced module $\triv_{S_n\wr S_m}\up^{S_{mn}}$ is a direct summand of $\triv_{S_m\wr S_n}\up^{S_{mn}}$ whenever $m\le n$. Here $\triv$ denotes the trivial representation, and Foulkes' Conjecture may equivalently be stated as $a^\mu_{(n),(m)}\ge a^\mu_{(m),(n)}$ for all partitions $\mu$.

Our first main result below is a recursive formula for plethysm coefficients of the form $a^\mu_{\lambda,(m)}$ for arbitrary partitions $\mu$ and $\lambda$. We extend the notation for plethysm coefficients from being indexed by partitions to being indexed by general skew shapes: see \eqref{eqn:skew-plethysm} for the full definition in terms of Littlewood--Richardson coefficients and plethysms indexed by bona fide partitions.
\begin{thmA}
	Fix $n\in\N$. Let $m\in\N$, $k\in\{0,1,\dotsc,n-1\}$ and $\lambda\vdash n$. Let $\mu\vdash mn$ with $l(\mu)=n-k$, and set $\hat{\mu}:=\mu-(1^{n-k})\vdash (m-1)n+k$. Then
	\[ a^\mu_{\lambda',(m)} = \sum_{i=0}^k (-1)^{k+i} \cdot \sum_{\substack{\alpha\vdash k+(m-1)i\\\beta\vdash i}} a^{\alpha/(k-i)}_{\beta',(m)}\cdot a^{\hat{\mu}/\alpha}_{\lambda/\beta,(m-1)}. \]
\end{thmA}
Here $\lambda'$ denotes the conjugate partition of $\lambda$, and we note that Theorem A only concerns partitions $\mu\vdash mn$ with $l(\mu)\le n$, since $a^\mu_{\lambda',(m)}=0$ whenever $l(\mu)>n$ (Lemma~\ref{lem:tall-pleth}).
We further remark that the case of $k=0$ in Theorem A coincides with \cite[Proposition 1.16]{BCV}, and also with the $\mu=(1^m)$ and $r=1$ special case of \cite[Theorem 1.1]{dBPW}, although our full Theorem A is a generalisation in a different direction.

We prove Theorem A as a consequence of a striking factorisation result concerning characters of symmetric groups (Theorem B below). In order to state this, we note that the irreducible characters of the symmetric group $S_n$ are naturally indexed by partitions of $n$, and for a partition $\lambda$ the corresponding character will be denoted by $\chi^\lambda$. This extends more generally to a useful class of characters $\chi^{\lambda/\mu}$ indexed by skew shapes $\lambda/\mu$, whose decompositions into irreducible constituents gives the Littlewood--Richardson coefficients (see \eqref{eqn:skew}). For any partition $\beta$ and any character $\phi$ of a symmetric group $S_n$, we let $\phi/\chi^\beta = \sum_{\alpha\vdash n} \langle \phi,\chi^\alpha\rangle \cdot \chi^{\alpha/\beta}$. For $\theta$ a character of $S_m$ we also write $\phi\boxtimes\theta:=(\phi\times\theta)\up_{S_n\times S_m}^{S_{n+m}}$. Finally, for any skew shape $\alpha/\beta$ of size $n$ we let $\rho^{\alpha/\beta}_m:=\chi^{\alpha/\beta}\up_{S_m\wr S_n}^{S_{mn}}$, where here $\chi^{\alpha/\beta}$ also denotes its inflation from $S_n$ to the imprimitive wreath product $S_m\wr S_n$. Then we may factorise such characters as follows.

\begin{thmB}
	Let $m,n,k\in\N$ with $m\ge 2$ and $k\in\{0,1,\dotsc,n-1\}$. Let $\lambda\vdash n$. Then
	\[ \rho^\lambda_m/\chi^{(1^{n-k})} = \sum_{\beta\vdash k} \rho^\beta_m \boxtimes \rho^{\lambda'/\beta'}_{m-1}. \]
\end{thmB}

We apply our Theorem A to deduce a new stability result for plethysm coefficients investigated in \cite{BBP}, and in the course of our work also resolve two conjectures of de Boeck \cite{dB-thesis}. 	
In addition to applications to plethysm coefficients, Theorem A allows us to deduce several new results on Sylow branching coefficients for symmetric groups. 
Motivated by connections to the McKay Conjecture \cite{GKNT, INOT} and the study of the relationship between characters of a finite group and those of its Sylow subgroups \cite{N,GN}, Sylow branching coefficients describe the decomposition of irreducible characters restricted from a finite group to their Sylow subgroups. Specifically, let $\Irr(G)$ denote the set of (ordinary) irreducible characters of a finite group $G$. Then for $\chi\in\Irr(G)$ and $\phi\in\Irr(P)$, where $P$ is a Sylow $p$-subgroup of $G$ for some given prime $p$, the Sylow branching coefficient $Z^\chi_\phi$ denotes the multiplicity $\langle \chi\down_P,\phi\rangle$.

Divisibility properties of Sylow branching coefficients were recently shown in \cite[Theorem A]{GLLV} to characterise whether a Sylow subgroup of a finite group is normal. Furthermore, the positivity of Sylow branching coefficients $Z^\chi_\phi$ for symmetric groups and linear characters $\phi$ was determined in the case of odd primes $p$ in \cite{GL1,GL2}. However, relatively little is known about these coefficients when $p=2$.

Using Theorem A, we are able to explicitly calculate several families of Sylow branching coefficients when $p=2$. In fact, we show that when $p=2$, there are very few Sylow branching coefficients of the symmetric group $S_n$ which take value zero as $n$ tends to infinity, countering a prediction made in \cite{GL1}.
\begin{thmC}
	For $n$ a natural number, let $P_n$ denote a Sylow 2-subgroup of the symmetric group $S_n$. Then almost all irreducible characters $\chi$ of $S_n$ have positive Sylow branching coefficient $Z^\chi_{\triv_{P_n}}$. That is,
	\[ \lim_{n\to\infty} \frac{|\{ \chi\in\Irr(S_n) \mid Z^\chi_{\triv_{P_n}} > 0\}|}{|\Irr(S_n)|} = 1. \]
\end{thmC}

\medskip

The structure of the article is as follows. In Section~\ref{sec:prelims}, we record the necessary background and notation. In particular, letting $P_n$ denote a Sylow 2-subgroup of $S_n$, we abbreviate $Z^{\chi^\lambda}_{\triv_{P_n}}$ to $Z^\lambda$.
In Section~\ref{sec:sbc} we collect together a number of elementary results on Sylow branching coefficients for symmetric groups. Specifically, in Section~\ref{sec:special-shapes} we compute $Z^\lambda$ for various `special' shapes of partitions $\lambda$, and the primary tools in this section will be the Littlewood--Richardson rule and Mackey's theorem for induction and restriction between subgroups. Theorem C is then proved in Section~\ref{sec:limit}.

In the second half of the article (Sections~\ref{sec:pre-recursive-formula}--\ref{sec:13}), we focus primarily on plethysm coefficients, motivated by the connections between plethysms and Sylow branching coefficients via wreath product groups. In Section~\ref{sec:pre-recursive-formula} we recall some useful results on plethysms in the literature. In particular, we give a proof of the $k=0$ case of Theorem A in the language of character deflations introduced by Evseev--Paget--Wildon in \cite{EPW}, and resolve two conjectures of de Boeck on plethysm coefficients involved in Foulkes' Conjecture in Section~\ref{sec:resolve}. We then illustrate some applications of Theorem A to plethysms in Section~\ref{sec:14} before proving Theorems A and B in full in Section~\ref{sec:16}. For ease of reference, Theorem A is numbered as Theorem~\ref{thm:14.6i} and Theorem B as Theorem~\ref{thm:16.7} below. Finally, in Section~\ref{sec:13} we apply Theorem A to deduce further results on Sylow branching coefficients for the prime 2.

\medskip

\subsection*{Acknowledgements}
The first author was supported by Emmanuel College, Cambridge. The second author was supported by a Summer Studentship from Trinity College, Cambridge.
We thank Mark Wildon and Eugenio Giannelli for helpful comments on an earlier version of this article. We are also grateful to the anonymous reviewers for helpful corrections and comments that have improved the exposition of this article.

\bigskip
%-----------------------------------------------------------------------------------
\section{Preliminaries}\label{sec:prelims}
Throughout, we use $\N$ to denote the set of natural numbers, and $\N_0$ for the set of non-negative integers. 
As stated in the introduction, for a finite group $G$ we use $\Irr(G)$ to denote a complete set of the ordinary irreducible characters of $G$. Further, we use $\Ch(G)$ to denote the set of all ordinary characters of $G$, and $\triv_G$ to mean the trivial character of $G$ (omitting the subscript when the meaning is clear from context).
For a subgroup $H\le G$ and $\phi\in\Irr(H)$, we let $\Irr(G\mid \phi)$ denote the set of $\chi\in\Irr(G)$ such that the restriction of $\chi$ to $H$ contains $\phi$ as a constituent. 

For $g\in G$ and $H\le G$, let $H^g:=gHg^{-1}$. Given $\phi\in\Ch(H)$, the character $\phi_g\in\Ch(H^g)$ is defined by $\phi_g(x)=\phi(g^{-1}xg)$. Mackey's Theorem allows us to describe restrictions and inductions between  subgroups of a finite group (see \cite[Chapter 5]{I}, for example).
\begin{theorem}[Mackey]\label{thm:Mackey}
	Let $G$ be a finite group and $H,K\le G$. Let $\phi\in\Ch(H)$. Then
	\[ \phi\up_H^G\down_K = \sum_{g\in K\setminus G/H} \phi_g \up_{H^g\cap K}^K, \]
	where the sum runs over a set of $(K,H)$-double coset representatives.
\end{theorem}

%------
\subsection{Representation theory of symmetric groups}\label{sec:prelims-sn}
Next, we recall some key facts concerning the representation theory of symmetric groups, and refer the reader to \cite{J,JK} for further detail.
It is well known that $\Irr(S_n)$ is naturally in bijection with the set $\cP(n)$ of all partitions of $n$. By convention, $\cP(0)=\{\emptyset\}$ where $\emptyset$ denotes the empty partition. 
The irreducible character of $S_n$ corresponding to the partition $\lambda\vdash n$ will be denoted by $\chi^\lambda$. In particular, $\chi^{(n)}=\triv_{S_n}$, the trivial character of $S_n$, and $\chi^{(1^n)}=\sgn_{S_n}$, the sign character of $S_n$. If $\alpha$ is a (finite) sequence of integers but is not a partition, then we interpret $\chi^\alpha$ to be the zero function.

The Young diagram of a partition $\lambda$ will be denoted by $[\lambda]$, and that of a skew shape $\lambda/\mu$ by $[\lambda/\mu]:=[\lambda]\setminus[\mu]$ for $\mu$ a subpartition of $\lambda$ (written $\mu\subseteq\lambda$). The boxes in a Young diagram will sometimes be referred to as nodes, and we refer to skew shapes and skew diagrams interchangeably when the meaning is clear.
We denote the length of the partition $\lambda$ by $l(\lambda)$, and the conjugate partition of $\lambda$ by $\lambda'$. Note
\begin{equation}\label{eqn:sign}
	\chi^{\lambda'}=\sgn_{S_n}\cdot\ \chi^\lambda
\end{equation}
for all $\lambda\vdash n$ (see \cite[2.1.8]{JK}). 

We record some operations on partitions. Let $\lambda=(\lambda_1,\lambda_2,\dotsc)$ and $\mu=(\mu_1,\mu_2,\dotsc)$ be two partitions. Then $+$ denotes component-wise addition, i.e.~$\lambda+\mu=(\lambda_1+\mu_1,\lambda_2+\mu_2,\dotsc)$, %$-$ denotes component-wise subtract similarly, 
and $\lambda\sqcup\mu$ denotes the partition obtained by taking the disjoint union of the parts of $\lambda$ and $\mu$ and reordering so that parts are in non-increasing order.
When clear from context, we abbreviate $(a^b):=(a,a,\dotsc,a)$ where there are $b$ parts of size $a$; in general we will specify $(a^b)\vdash ab$ to avoid confusion with the single part partition of $a^b$.

\medskip

We also make use of skew characters for symmetric groups, i.e.~those indexed by skew shapes. For partitions $\mu$ and $\lambda$ such that $|\mu|\le|\lambda|$, the skew character $\chi^{\lambda/\mu}$ of $S_{|\lambda|-|\mu|}$ satisfies
\begin{equation}\label{eqn:skew}
	\langle \chi^{\lambda/\mu},\chi^\nu\rangle = \langle \chi^\lambda, (\chi^\mu\times \chi^\nu)\up_{S_{|\mu|}\times S_{|\nu|}}^{S_{|\mu|+|\nu|}}\rangle\qquad\forall\ \nu\vdash |\lambda|-|\mu|.
\end{equation}
Note if $\mu\not\subseteq\lambda$ then $\chi^{\lambda/\mu}=0$. This can be seen from the Littlewood--Richardson rule, which gives an explicit combinatorial description of the decomposition into irreducibles of the induced character appearing in the above expression, with Littlewood--Richardson coefficients arising as the multiplicities. These appear in many contexts, so we shall now fix the notation which will be used throughout this article (see \cite{J}).

\begin{definition}
	Let $n\in\N$, $\lambda=(\lambda_1,\dotsc,\lambda_k)\vdash n$ and $\mathcal{C}=(c_1,c_2,\dotsc,c_n)$ be a sequence of positive integers. We say that $\mathcal{C}$ is of type $\lambda$ if $|\{i\in\{1,2,\dotsc,n\} \mid c_i=j\}|=\lambda_j$ for all $j\in\{1,2,\dotsc,k\}$. Moreover, we say that an element $c_j$ of $\mathcal{C}$ is good if $c_j=1$ or if
	\[ |\{i\in\{1,2,\dotsc,j-1\}\mid c_i=c_j-1\}| > |\{i\in\{1,2,\dotsc,j-1\}\mid c_i=c_j\}|. \]
	Finally, we say that $\mathcal{C}$ is good if $c_j$ is good for all $j\in\{1,\dotsc,n\}$.
\end{definition}

\begin{theorem}[Littlewood--Richardson rule]\label{thm:LR}
	Let $m,n\in\mathbb{N}_0$. Let $\mu\vdash m$ and $\nu\vdash n$. Then
	\[ (\chi^\mu\times\chi^\nu)\up^{S_{m+n}}_{S_m\times S_n} = \sum_{\lambda\vdash m+n} c^\lambda_{\mu,\nu}\cdot \chi^\lambda\]
	where the Littlewood--Richardson coefficient $c^\lambda_{\mu,\nu}$ equals the number of ways to replace the nodes of the skew Young diagram of $\lambda/\mu$ by natural numbers such that
	\begin{enumerate}[label=\textup{(\roman*)}]
		\item the sequence obtained by reading the numbers from right to left, top to bottom is good of type $\nu$;
		\item the numbers are non-decreasing (weakly increasing) left to right along rows; and
		\item the numbers are strictly increasing down columns.
	\end{enumerate}
\end{theorem}
We call the order in Theorem~\ref{thm:LR}(i) the reading order of a skew shape. Let $\nu$ be a partition and $\gamma$ be a skew shape of size $|\nu|$. We call a way of replacing the nodes of $\gamma$ by numbers satisfying conditions Theorem~\ref{thm:LR}(i)--(iii) a Littlewood--Richardson (LR) filling of $\gamma$ of type $\nu$. 
Clearly $c^\lambda_{\mu,\nu}=c^\lambda_{\nu,\mu}$. Using Littlewood--Richardson coefficients, we can also rephrase \eqref{eqn:skew} as $\chi^{\lambda/\mu}=\sum_\nu c^\lambda_{\mu,\nu}\cdot \chi^\nu$. Moreover, we can extend this notation to `generalised' Littlewood--Richardson coefficients $c^\lambda_{\mu^1,\dotsc,\mu^r}$ describing the constituents of
\[ (\chi^{\mu^1}\times\cdots\times\chi^{\mu^r})\up_{S_{n_1}\times\cdots\times S_{n_r}}^{S_{n_1+\cdots+n_r}}, \]
for any $r\in\N$ and $n_i\in\N_0$, and partitions $\mu^i\vdash n_i$ and $\lambda\vdash n_1+\cdots+n_r$. Similarly, $c^\lambda_{\mu^1,\dotsc,\mu^r}=c^\lambda_{\mu^{\sigma(1)},\dotsc,\mu^{\sigma(r)}}$ for any $\sigma\in S_r$. Furthermore, for $A\subseteq\cP(n)$ and $B\subseteq\cP(m)$ we define the operation $\star$ as follows:
\[ A\star B:= \{\lambda\vdash m+n \mid \exists\ \mu\in A,\ \nu\in B\text{ s.t. }c^\lambda_{\mu,\nu}>0 \}. \]
We note that $\star$ is both commutative and associative.

Finally, for a partition $\lambda=(\lambda_1,\lambda_2,\dotsc,\lambda_k)\vdash n$, we let $S_\lambda\cong S_{\lambda_1}\times\cdots\times S_{\lambda_k}$ denote the corresponding Young subgroup of $S_n$. The permutation module $\triv_{S_\lambda}\up^{S_n}$ induced by the action of $S_n$ on the cosets of $S_\lambda$ will be denoted by $M^\lambda$. % (see \cite[Chapter 4]{J}).
Young's Rule (see \cite[2.8.5]{JK}) tells us the decomposition of these permutation modules into irreducibles. Denoting the character of $M^\lambda$ by $\xi^\lambda$, we have that $\langle \xi^\lambda,\chi^\alpha\rangle$ equals the number of semistandard Young tableaux of shape $\alpha$ and content $\lambda$, for any $\alpha\vdash n$. Moreover, this multiplicity is positive if and only if $\alpha\trianglerighteq\lambda$, where $\trianglerighteq$ denotes the dominance partial order on partitions.

\medskip

\subsection{Wreath products and Sylow subgroups of symmetric groups}
In order to describe the Sylow subgroups of symmetric groups, we briefly introduce some notation for wreath products. Let $G$ be a finite group and let $H\le S_n$ for some $n\in\N$. The natural action of $S_n$ on the factors of the direct product $G^{\times n}$ induces an action of $S_n$ (and therefore of $H$) via automorphisms of $G^{\times n}$, giving the wreath product $G\wr H:= G^{\times n}\rtimes H$. %We sometimes refer to $G^{\times n}$ as the base group of the wreath product $G\wr H$. 
As in \cite[Chapter 4]{JK}, we denote the elements of $G\wr H$ by $(g_1,\dotsc,g_n;h)$ for $g_i\in G$ and $h\in H$. Let $V$ be a $\C G$--module and suppose it affords the character $\phi$. Let $V^{\otimes n}$ be the corresponding $\C G^{\times n}$--module. The left action of $G\wr H$ on $V^{\otimes n}$ defined by linearly extending
\[ (g_1,\dotsc,g_n;h)\ :\quad v_1\otimes \cdots\otimes v_n \longmapsto g_1v_{h^{-1}(1)}\otimes\cdots\otimes g_nv_{h^{-1}(n)} \]
turns $V^{\otimes n}$ into a $\C(G\wr H)$--module, which we denote by $\widetilde{V^{\otimes n}}$ (see \cite[(4.3.7)]{JK}), and we denote its character by $\tilde{\phi}$. For any $\psi\in\Ch(H)$, we define $\cX(\phi;\psi)$ as follows:
\[ \cX(\phi;\psi):=\tilde{\phi}\cdot\Infl_H^{G\wr H}(\psi)\ \in\Ch(G\wr H). \]
The inflation $\Infl_H^{G\wr H}(\psi)$ of $\psi$ from $H$ to $G\wr H$ (identifying $H$ with the quotient $(G\wr H)/G^{\times n}$) is sometimes abbreviated to simply $\psi$, for convenience.

\begin{lemma}\label{lem:infl-ind}
	Let $G$ ,$H$ and $\phi$ be as above. Let $L\le H$ and $\tau\in\Ch(L)$. Then $\cX(\phi;\tau)\up_{G\wr L}^{G\wr H} = \cX(\phi;\tau\up_L^H)$.
\end{lemma}
\begin{proof}
	For any $\alpha\in\Ch(H)$ and $\beta\in\Ch(L)$, it is easy to check that $\alpha\cdot(\beta\up^H)=(\alpha\down_L\cdot \beta)\up^H$. Hence
	\[ \cX(\phi;\tau)\up_{G\wr L}^{G\wr H} := \big( \tilde{\phi}\down_{G\wr L}^{G\wr H}\cdot\Infl_L^{G\wr L}\tau\big)\up_{G\wr L}^{G\wr H} = \tilde{\phi} \cdot \big(\Infl_L^{G\wr L}(\tau)\big)\up_{G\wr L}^{G\wr H}. \]
	% notice that \tilde{\phi} can in fact be defined as a character of G\wr S_n, so that the various \tilde{\phi} that we use in the definition of \cX(\phi;...) are simply restrictions of this from G\wr S_n to G\wr H for various H.
	But induction and inflation of characters commute,
	% see e.g. Exercise 4.1.11 of https://arxiv.org/abs/1409.8356v5 (https://arxiv.org/src/1409.8356v5/anc/HopfComb-v73-with-solutions.pdf)
	so $\big(\Infl_L^{G\wr L}(\tau)\big)\up_{G\wr L}^{G\wr H} = \Infl_H^{G\wr H}(\tau\up_L^H)$. Therefore $\cX(\phi;\tau)\up_{G\wr L}^{G\wr H} = \tilde{\phi}\cdot \Infl_H^{G\wr H}(\tau\up_L^H) = \cX(\phi;\tau\up_L^H)$, as claimed.
\end{proof}

Furthermore, if $\phi\in\Irr(G)$ then Gallagher's Theorem \cite[Corollary 6.17]{I} gives $\Irr(G\wr H\mid \phi^{\times n}) = \{\cX(\phi;\psi)\mid \psi\in\Irr(H)\}$, where $\Irr(G\wr H\mid \phi^{\times n}):=\{\chi\in\Irr(G\wr H)\mid \langle \chi\down_{G^{\times n}}, \phi^{\times n}\rangle \ne 0 \}$.
For a full description of $\Irr(G\wr H)$, we refer the reader to \cite[Chapter 4]{JK}; in the case $H=S_2$, we use the notation below.

\begin{notation}\label{not:GwrS2}
	Let $G$ be a finite group and suppose $\Irr(G)=\{\chi^i \mid i\in I\}$. Then 
	\[ \Irr(G\wr S_2) = \{ \psi^{i,j} \mid i\ne j\in I \} \sqcup \{ \psi^{i,i}_+, \psi^{i,i}_- \mid i\in I \} \]
	where 
	\[ \psi^{i,j} := (\chi^i\times\chi^j)\up_{G\times G}^{G\wr S_2}\ = \psi^{j,i},\quad \psi^{i,i}_{+} := \cX(\chi^i;\triv_{S_2})\quad\text{and}\quad \psi^{i,i}_{-} := \cX(\chi^i;\sgn_{S_2}). \]
	%In particular, we remark that $\psi^{i,j}=\psi^{j,i}$ whenever $i\ne j$.
\end{notation}

It will be useful to describe the decomposition of the permutation character $\triv_{H\wr S_2}\up^{G\wr S_2}$, for finite groups $H\le G$.

\begin{lemma}\label{lem:Prop1.2}
	Let $G$ be a finite group and let $\Irr(G)$ and $\Irr(G\wr S_2)$ be as in Notation~\ref{not:GwrS2}.
	Let $H\leq G$ and let $\pi := \triv_{H}\up^{G}$. Then $\tilde{\pi} := \triv_{H\wr S_2}\up^{G\wr S_2}$ decomposes into irreducible constituents with multiplicities given by
	\[ \langle\tilde{\pi},\psi^{i,j}\rangle = \langle\pi,\chi^i\rangle \cdot \langle\pi,\chi^j\rangle \quad\text{and}\quad \langle\tilde{\pi},\psi^{i,i}_{\pm}\rangle = \tfrac{1}{2}\cdot\big( \langle\pi,\chi^i\rangle^2 \pm \langle\pi,\chi^i\rangle \big). \]
\end{lemma}

\begin{proof}
	The first part follows from Mackey's theorem applied to the subgroups $G\times G$ and $H\wr S_2$ of $G\wr S_2$: since $(G\times G)\cdot (H\wr S_2)=G\wr S_2$ and $(G\times G)\cap (H\wr S_2)=H\times H$, we have
	\[ \langle\tilde{\pi},\psi^{i,j}\rangle	= \langle \triv_{H\wr S_2}\up^{G\wr S_2}\down_{G\times G}, \chi^i\times\chi^j\rangle = \langle \triv_{H\times H}\up^{G\times G}, \chi^i\times \chi^j\rangle = \langle\pi,\chi^i\rangle \cdot \langle\pi,\chi^j\rangle. \]
	Next, we use the wreath product character formula \cite[Lemma 4.3.9]{JK} to obtain
	\[ \psi^{i,i}_{\pm}(g_1,g_2;1) = \chi^i(g_1)\cdot\chi^i(g_2)\quad\text{and}\quad\psi^{i,i}_{\pm}\big(g_1,g_2;(1,2)\big) = \pm\chi^i(g_1 g_2) \qquad\forall\ g_1, g_2\in G. \]
	Hence
	\[ \langle\tilde{\pi},\psi^{i,i}_{\pm}\rangle = \langle\triv_{H\wr S_2},\psi^{i,i}_{\pm}\down_{H\wr S_2}\rangle = \frac{1}{|H\wr S_2|}\sum_{h_1,h_2\in H}\big(\chi^i(h_1)\cdot\chi^i(h_2)\pm\chi^i(h_1 h_2)\big) = \tfrac{1}{2}\big(\langle\pi,\chi^i\rangle^2\pm \langle\pi,\chi^i\rangle\big) \]
	as claimed.
\end{proof}

To describe Sylow subgroups of symmetric groups, %(see \cite[Chapter 4]{JK} for further detail). 
fix a prime $p$ and let $n\in\N$. Let $P_n$ denote a Sylow $p$-subgroup of $S_n$. Clearly $P_1$ is trivial while $P_p$ is cyclic of order $p$. More generally, $P_{p^k}= (P_{p^{k-1}})^{\times p}\rtimes P_p=P_{p^{k-1}}\wr P_p\cong P_p\wr \cdots \wr P_p$ ($k$-fold wreath product) for all $k\in\N$. Now suppose $n\in\N$ and let $n=\sum_{i=1}^k a_ip^{n_i}$ be its $p$-adic expansion, i.e.~$n_1>\cdots>n_k\ge 0$ and $a_i\in\{1,2,\dotsc,p-1\}$ for all $i$. Then $P_n\cong (P_{p^{n_1}})^{\times a_1}\times\cdots\times (P_{p^{n_k}})^{\times a_k}$. 

To fix a convention for denoting such wreath products involving Sylow subgroups of symmetric groups more generally, we have the following.

\begin{notation}\label{not:convention}
	Let $G$ be a finite group $G$ and $p$ a prime. We use the convention that $P_n$ will always be viewed as a subgroup of $S_n$ in the notation $G\wr P_n$, that is, $G\wr P_n$ is a semi-direct product $G^{\times n}\rtimes P_n$. 
\end{notation}

\begin{remark}\label{rem:convention}
	Suppose $p=2$ and $n=2^{n_1}+\cdots+2^{n_k}$ for some $n_1>\cdots>n_k\ge 0$. With the convention of Notation~\ref{not:convention}, we observe that 
	\[ P_{2n} \cong P_{2^{n_1+1}}\times\cdots\times P_{2^{n_k+1}} \cong (P_2\wr P_{2^{n_1}})\times\cdots\times (P_2\wr P_{2^{n_k}}) \cong P_2 \wr (P_{2^{n_1}}\times \cdots \times P_{2^{n_k}})\cong P_2 \wr P_n, \]
	viewing $P_{2^{n_1}}\times \cdots \times P_{2^{n_k}}\cong P_n$ naturally as a subgroup of $S_n$. Inductively, we also have $P_{2^tn}\cong P_{2^t}\wr P_n$ for all $t\in\N$.
	On the other hand, we clarify for example that $P_2\wr P_3 \not\cong P_2\wr P_2$ in this notation, even though $P_3\cong P_2$.\hfill$\lozenge$
	%If $p=2$ then necessarily $a_i=1$ for all $i$, and we obtain $P_2\wr P_n \cong P_{2^{n_1+1}}\times\cdots\times P_{2^{n_k+1}} \cong P_{2n}$. Inductively, $P_{2^tn}\cong P_{2^t}\wr P_n$ for all $t\in\N$.
\end{remark}

We now return to an arbitrary prime $p$. Following the notation introduced in \cite{GL2}, given $\chi\in\Irr(S_n)$ and $\phi\in\Irr(P_n)$, the \textit{Sylow branching coefficient} $Z^\chi_\phi$ denotes the non-negative integer
\[ Z^\chi_\phi:=\langle \chi\down_{P_n}, \phi\rangle. \]
In this article, we will be particularly interested in the case where $\phi=\triv_{P_n}$, and abbreviate $Z^\chi_{\triv_{P_n}}$ to $Z^\chi$. Moreover, if $\chi=\chi^\lambda$ for a partition $\lambda$, then we shorten $Z^{\chi^\lambda}$ to $Z^\lambda$. 

We record one more lemma which will be useful later.
\begin{lemma}\label{lem:X-induced}
	Let $A$ and $B$ be finite groups, and let $n\in\N$. Then
	\[ \triv\up_{(A\times B)\wr S_n}^{A\wr S_n \times B\wr S_n} = \sum_{\phi\in\Irr(S_n)} \cX(\triv_A;\phi)\cdot\cX(\triv_B;\phi). \]
\end{lemma}
\begin{proof}
	Let $\phi\in\Irr(S_n)$. By Frobenius reciprocity, 
	\begin{align*}
		\langle \triv\up_{(A\times B)\wr S_n}^{A\wr S_n \times B\wr S_n}, \cX(\triv_A;\phi)\cdot\cX(\triv_B;\phi)\rangle &= \langle \triv, \big( \cX(\triv_A;\phi)\cdot\cX(\triv_B;\phi) \big) \down_{(A\times B)\wr S_n}\rangle\\
		&= \tfrac{1}{|(A\times B)\wr S_n|}\sum_{x\in (A\times B)\wr S_n} \cX(\triv_A;\phi)(x)\cdot\cX(\triv_B;\phi)(x).
	\end{align*}
	But this equals $\tfrac{|A|^n\cdot|B|^n}{|(A\times B)\wr S_n|}\sum_{g\in S_n} \phi(g)^2$ by \cite[Lemma 4.3.9]{JK}. As symmetric group characters are real-valued (in fact, integer-valued), this simplifies to $\tfrac{1}{|S_n|}\sum_{g\in S_n}\overline{\phi(g)}\cdot\phi(g)=\langle \phi,\phi\rangle=1$.
\end{proof}

\medskip

\subsection{Plethysms and deflations}\label{sec:plethysm-deflation}
When $\phi$ and $\psi$ are characters of symmetric groups, the characters $\cX(\phi;\psi)$ introduced above are closely related to plethysms of Schur functions: we give a brief description here and refer the reader to \cite{dBPW, Mac, Stanley} for further detail. Let $s_\lambda$ denote the Schur function corresponding to the partition $\lambda$, and $\circ$ the plethystic product of symmetric functions. Using the characteristic map (see e.g.~\cite[Chapter 7]{Stanley}) between class functions of finite symmetric groups and the ring of symmetric functions, we have the correspondence
\[ \cX(\chi^\nu;\chi^\lambda)\up_{S_{|\nu|}\wr S_{|\lambda|}}^{S_{|\nu|\cdot|\lambda|}}\ \longleftrightarrow\ s_\lambda\circ s_\nu \]
for all partitions $\lambda$ and $\nu$. Therefore the plethysm coefficient $a^\mu_{\lambda,\nu}$ satisfies
\begin{equation}\label{eqn:a}
	a^\mu_{\lambda,\nu} = \langle s_\lambda\circ s_\nu, s_\mu\rangle = \langle \cX(\chi^\nu;\chi^\lambda)\up_{S_{|\nu|}\wr S_{|\lambda|}}^{S_{|\nu|\cdot|\lambda|}}, \chi^\mu \rangle	
\end{equation}
for all partitions $\mu$ (and note that this is zero if $|\mu|\ne|\nu|\cdot|\lambda|$). 
We also introduce plethysm coefficients indexed by skew shapes: for partitions $\beta\subseteq\alpha$, $\delta\subseteq\gamma$ and $\nu$,
\begin{equation}\label{eqn:skew-plethysm}
	a^{\alpha/\beta}_{\gamma/\delta,\nu}:= \sum_{\substack{\eta\vdash|\gamma|-|\delta|\\\zeta\vdash|\alpha|-|\beta|}} c^\gamma_{\eta,\delta}\cdot c^\alpha_{\zeta,\beta}\cdot a^\zeta_{\eta,\nu}.
\end{equation}
In other words, if $\phi$ and $\theta$ are skew shapes and $\nu$ is any partition, then $a^\phi_{\theta,\nu} = \langle \cX(\chi^\nu; \chi^\theta)\up_{S_{|\nu|}\wr S_{|\theta|}}^{S_{|\nu|\cdot|\theta|}}, \chi^\phi\rangle$, extending the equality in \eqref{eqn:a}. We also remark that if $\zeta$ and $\eta$ are partitions, then $a^\zeta_{\eta,\emptyset}=1$ if $\zeta=\emptyset$ and $\eta=(n)$ for some $n$, and $a^\zeta_{\eta,\emptyset}=0$ otherwise. 

A well known symmetry property of plethysm coefficients involving the conjugation involution is the following (see e.g.~\cite[Ex.~1, Ch. I.8]{Mac}):
\begin{lemma}\label{lem:conj-plethysm}
	Let $\lambda$, $\mu$ and $\nu$ be partitions. Then
	\[ a^\nu_{\lambda,\mu} = a^{\nu'}_{\lambda^\ast,\mu'},\quad\text{where}\quad \lambda^\ast:=\begin{cases} \lambda & \text{if $|\mu|$ is even},\\ \lambda' & \text{if $|\mu|$ is odd}.\end{cases} \]
\end{lemma}

Character deflations were introduced in \cite[Definition 1.1]{EPW}, and used to prove results generalising the Murnaghan--Nakayama rule for computing symmetric group character values and (special cases of) the Littlewood--Richardson rule, as well as to verify new cases of the long-standing Foulkes' Conjecture. We observe that they give another language in which to describe certain plethysm coefficients. We first record the definition of these deflations in notation which we have introduced thus far.
\begin{definition}\label{def:deflation}
	Let $m,n\in\N$ and $\theta\in\Irr(S_m)$. Let $\xi\in\Irr(S_m\wr S_n)$. Then 
	\[ \operatorname{Def}^\theta_{S_n}(\xi) := \begin{cases}
		\chi^\nu & \text{ if }\xi = \cX(\theta;\chi^\nu)\text{ for some }\nu\vdash n,\\
		0 & \text{ otherwise},
	\end{cases} \]
	which then extends linearly to all class functions of $S_m\wr S_n$. If $\chi$ is a class function of $S_{mn}$ then set
	\[ \operatorname{Defres}^\theta_{S_n}(\chi) := \operatorname{Def}^\theta_{S_n}(\chi\down_{S_m\wr S_n}). \]
	When $n\in\N$ is fixed and $\gamma\vdash mn$ for some $m\in\N$, we use the notation
	\[ \delta^\gamma:=\operatorname{Defres}^{\triv_{S_m}}_{S_n}(\chi^\gamma) = \operatorname{Def}^{\triv_{S_m}}_{S_n}(\chi^\gamma\down^{S_{mn}}_{S_m\wr S_n}). \]
	In this article, we will sometimes refer to $\delta^\gamma$ as the deflation of $\gamma$ with respect to $S_n$ (where the $m$ is understood from $|\gamma|/n$ and we suppress $m$ from the notation). % note EPW refers to deflation wrt to the superscript rather than the subscript
\end{definition}
In particular, if $\lambda\vdash n$ then we have that $a^\gamma_{\lambda,(m)} = \langle \delta^\gamma,\chi^\lambda \rangle$. This relation between plethysm coefficients and deflations can immediately be extended to skew shapes using Littlewood--Richardson coefficients. Namely, if $\alpha,\beta,\nu,\lambda$ are partitions such that $|\alpha|-|\beta|=n$ and $|\nu|-|\lambda|=mn$, then we may define $\delta^{\nu/\lambda}:=\operatorname{Defres}^{\triv_{S_m}}_{S_n}(\chi^{\nu/\lambda}) = \operatorname{Def}^{\triv_{S_m}}_{S_n}(\chi^{\nu/\lambda}\down^{S_{mn}}_{S_m\wr S_n})$. Then $\chi^{\alpha/\beta}=\sum_\sigma c^\alpha_{\beta,\sigma}\cdot \chi^\sigma$ and $\delta^{\nu/\lambda}=\sum_\tau c^\nu_{\lambda,\tau}\cdot \delta^\tau$ give
\begin{equation}\label{eqn:plethysm-deflation-skew}
	a^{\nu/\lambda}_{\alpha/\beta,(m)} = \langle \delta^{\nu/\lambda}, \chi^{\alpha/\beta}\rangle.
\end{equation}

The following is a straightforward result on plethysm coefficients involving `tall' partitions. We include a proof for convenience.
\begin{lemma}\label{lem:tall-pleth}
	Let $m,n\in\N$, $\lambda\vdash mn$ and $\nu\vdash n$. If $l(\lambda)>n$, then $a^\lambda_{\nu,(m)}=0$.
\end{lemma}

\begin{proof}
	Let $a:=a^\lambda_{\nu,(m)}= \langle \chi^\lambda\down^{S_{mn}}_{S_m\wr S_n}, \cX(\chi^{(m)};\chi^\nu)\rangle$. Let $Y:=(S_m)^{\times n}\le S_m\wr S_n$. 
	Let $\cX:=\cX(\chi^{(m)};\chi^\nu)$ and write $\chi^\lambda\down_{S_m\wr S_n} = a\cX + \Delta$ for some $\Delta\in\Ch(S_m\wr S_n)$. Since $\cX\down_Y = \chi^\nu(1)\cdot(\chi^{(m)})^n=\chi^\nu(1)\cdot\triv_Y$, it follows that $\langle \chi^\lambda\down_Y,\triv_Y\rangle \ge \langle a\cdot \cX\down_Y, \triv_Y\rangle = a\cdot\chi^\nu(1).$
	%Since $\cX(\chi^{(m)};\chi^\nu)\down_Y = \chi^\nu(1)\cdot(\chi^{(m)})^n=\chi^\nu(1)\cdot\triv_Y$, it follows that $\langle \chi^\lambda\down_Y,\triv_Y\rangle \ge a\cdot\chi^\nu(1)$. 
	But by the Littlewood--Richardson rule, $l(\lambda)>n$ implies $\langle \chi^\lambda\down_Y,\triv_Y\rangle=0$ since $\triv_Y = (\chi^{(m)})^n$, which gives $a=0$.
\end{proof}

Finally, we record the following result of Thrall \cite{T}. We note that a partition is even if all of its parts are of even size.
\begin{proposition}\label{prop:thrall}
	Let $n\in\N$. Then 
	\begin{enumerate}[label=\textup{(\roman*)}]
		\item $s_{(n)}\circ s_{(2)} = \sum_\lambda s_\lambda$ where the sum runs over all even partitions $\lambda\vdash 2n$, and
		\item $s_{(2)}\circ s_{(n)} = \sum_\mu s_\mu$ where the sum runs over all even partitions $\mu\vdash 2n$ with at most two parts.
	\end{enumerate}
\end{proposition}
In other words, $a^\lambda_{(n),(2)}=1$ if $\lambda\vdash 2n$ is even and $a^\lambda_{(n),(2)}=0$ otherwise, while $a^\mu_{(2),(n)}=1$ if $\mu\vdash 2n$ is even and $l(\mu)\le 2$, and $a^\mu_{(2),(n)}=0$ otherwise.

\bigskip
%-----------------------------------------------------------------------------------
\section{Sylow branching coefficients for the prime 2}\label{sec:sbc}
Throughout Section~\ref{sec:sbc}, we fix $p=2$ and consider Sylow branching coefficients $Z^\lambda=Z^{\chi^\lambda}_{\triv_{P_n}}$ for symmetric groups for the prime 2.

\subsection{Special shapes}\label{sec:special-shapes}
In this section we provide a survey of facts involving $Z^\lambda$ for partitions $\lambda$ of various `special' shapes: namely when $\lambda$ is an even partition; when $l(\lambda)$ is large; when $\lambda$ has at most 2 columns; when $\lambda$ is a hook; and when $\lambda$ is of the form $(a,2,1^b)$. 

In general, the strategy is to first consider the case when $\lambda\vdash n=2^k$ and to induct on $k$, before considering the case of general $n\in\N$. The results follow from a combination of elementary applications of the Littlewood--Richardson rule, Mackey's theorem and known results on character restrictions for symmetric groups. We include full proofs for the convenience of the reader.

\begin{lemma}
	Let $\lambda$ be any partition. Then $Z^{2\lambda}>0$.
\end{lemma}

\begin{proof}
	Suppose $\lambda\vdash n$. Then $Z^{2\lambda} %= \langle \chi^{2\lambda}\down_{P_{2n}},\triv_{P_{2n}}\rangle 
	\ge \langle \chi^{2\lambda}\down_{S_2\wr S_n},\triv_{S_2\wr S_n}\rangle=1$ since $P_{2n}\cong P_2\wr P_n\le S_2\wr S_n$, where the equality follows from $\triv\up_{S_2\wr S_n}^{S_{2n}} = \sum_{\mu\vdash n} \chi^{2\mu}$ by Proposition~\ref{prop:thrall}.
\end{proof}

\begin{lemma}\label{lem:tall}
	Let $n\in\N$ and $\lambda\vdash n$. If $n$ is even and $l(\lambda)>\frac{n}{2}$, then $Z^\lambda=0$. If $n$ is odd and $l(\lambda)>\frac{n+1}{2}$, then $Z^\lambda=0$.
\end{lemma}

\begin{proof}
	\textbf{(i)} First we consider the case $n=2^k$ and proceed by induction on $k$, noting that the claim is clear for small $k\in\N$. Suppose $\mu\vdash 2^k$ with $l(\mu)>2^{k-1}$ and $Z^\mu>0$. Let $S:=S_{2^k}$, $P:=P_{2^k}=P_{2^{k-1}}\wr P_2\le S$ and $Q=P_{2^{k-1}}\times P_{2^{k-1}}\le P$. Let $Y$ be the subgroup of $S$ isomorphic to $S_{2^{k-1}}\times S_{2^{k-1}}$ containing $Q$.
	
	Since $Z^\mu>0$, then considering $\chi^\mu\down_Q = (\chi^\mu\down_P)\down_Q$ gives $\langle \chi^\mu\down_Q,\triv_Q\rangle >0$. On the other hand, by considering $\chi^\mu\down_Q = (\chi^\mu\down_Y)\down_Q$ we obtain
	\[ \langle \chi^\mu\down_Q,\triv_Q\rangle = \sum_{\alpha,\beta\vdash 2^{k-1}} c^\mu_{\alpha,\beta}\cdot Z^\alpha\cdot Z^\beta. \]
	If $c^\mu_{\alpha,\beta}>0$, then by the Littlewood--Richardson rule we must have $l(\mu)\le l(\alpha)+l(\beta)$, and so either $l(\alpha)>2^{k-2}$ or $l(\beta)>2^{k-2}$. But then by the inductive hypothesis $Z^\alpha=0$ or $Z^\beta=0$ for each such $c^\mu_{\alpha,\beta}>0$, giving $\langle \chi^\mu\down_Q,\triv_Q\rangle=0$, a contradiction.
	
	\noindent\textbf{(ii)} Now consider general $n\in\N$. Suppose $n=2^{n_1}+\cdots+2^{n_k}$ with $n_1>\cdots>n_k\ge 0$ and $\lambda\vdash n$. Then 
	\[  Z^\lambda = \sum_{\mu^i\vdash 2^{n_i}} c^\lambda_{\mu^1,\dotsc,\mu^k}\cdot Z^{\mu^1}\cdots Z^{\mu^k}. \]
	But $c^\lambda_{\mu^1,\dotsc,\mu^k}>0$ implies $l(\lambda)\le l(\mu^1)+\cdots+l(\mu^k)$. If $n$ is even and $l(\lambda)>\frac{n}{2}$, then there exists $1\le i\le k$ such that $l(\mu^i)>2^{n_i-1}$, and so $Z^\lambda=0$ follows from case (i) as $Z^{\mu^i}=0$. If $n$ is odd and $l(\lambda)>\frac{n+1}{2}$, then $l(\mu^1)+\cdots+l(\mu^{k-1})+1\ge \frac{n+1}{2}+1$ and so there exists $1\le i\le k-1$ such that $l(\mu^i)>2^{n_i-1}$. That $Z^\lambda=0$ follows from case (i) similarly.
\end{proof}

\begin{remark}
	The bounds on the number of parts of $\lambda$ cannot be improved. For instance, from Lemma~\ref{lem:2col} below we see that $\lambda=(2,2,\dotsc,2,\varepsilon)\vdash n$ where $\varepsilon\in\{0,1\}$ satisfies $Z^\lambda=1$ and $l(\lambda)=\frac{n}{2}$ if $n$ is even, respectively $l(\lambda)=\frac{n+1}{2}$ if $n$ is odd.\hfill$\lozenge$
\end{remark}

\begin{lemma}\label{lem:2col}
	Let $\lambda$ be a partition with at most two columns. Then $Z^\lambda=0$ unless $\lambda=(2,2,\dotsc,2,\varepsilon)$ where $\varepsilon\in\{0,1\}$, in which case $Z^\lambda=1$.
\end{lemma}

\begin{proof}
	First suppose $|\lambda|=2^k$ where $k\in\N_0$ and proceed by induction on $k$. Clearly the claim holds for small $k$, so now suppose $\lambda\vdash 2^{k+1}$ and $\lambda_1\le 2$. We immediately deduce from Lemma~\ref{lem:tall} that if $\lambda\ne(2^{2^k})$ then $Z^\lambda=0$, so we may suppose $\lambda=(2^{2^k})$. We have that
	\[ Z^\lambda \le \langle \chi^\lambda\down_{P_{2^k}\times P_{2^k}}, \triv_{P_{2^k}\times P_{2^k}}\rangle = \sum_{\mu,\nu\vdash 2^k} c^\lambda_{\mu,\nu} \cdot Z^\mu\cdot Z^\nu. \]
	Let $\alpha:=(2^{2^{k-1}})\vdash 2^k$. If $c^\lambda_{\mu,\nu}>0$, then $\mu_1,\nu_1\le\lambda_1\le 2$. By the inductive hypothesis, $Z^\mu=\delta_{\mu,\alpha}$ and $Z^\nu=\delta_{\nu,\alpha}$ (here $\delta$ denotes the Kronecker delta). Also $c^\lambda_{\alpha,\alpha}=1$, %if $\lambda=(2^{2^k})$ and $c^\lambda_{\alpha,\alpha}=0$ otherwise. Hence if $\lambda\ne(2^{2^k})$, then $Z^\lambda=0$.
	%Now suppose $\lambda=(2^{2^k})$, 
	so we have shown that $Z^\lambda\le 1$. 
	On the other hand, $Z^\lambda>0$ by Proposition~\ref{prop:thrall}, so we conclude that $Z^\lambda=1$.
%	\[ Z^\lambda\ge \langle \chi^\lambda\down^{S_{2^{k+1}}}_{S_{2^k}\wr S_2}, \cX(\chi^\alpha;\chi^{(2)})\rangle \cdot \langle \cX(\chi^\alpha;\chi^{(2)})\down_{P_{2^{k+1}}}, \triv \rangle = \langle \chi^\lambda\down^{S_{2^{k+1}}}_{S_{2^k}\wr S_2}, \cX(\chi^\alpha;\chi^{(2)})\rangle\]
%	since 
%	\[ \langle \cX(\chi^\alpha;\chi^{(2)})\down_{P_{2^{k+1}}},\triv_{P_{2^{k+1}}}\rangle = \langle \cX(\chi^\alpha\down_{P_{2^k}}; \chi^{(2)}), \cX(\triv_{P_{2^k}}; \chi^{(2)})\rangle =1\]
%	by \cite[Lemma 2.19]{SLthesis}. Moreover, $\cX(\chi^{(2^{k-1},2^{k-1})}; \chi^{(2)})$ is a constituent of $\chi^{(2^k,2^k)}\down^{S_{2^{k+1}}}_{S_{2^k}\wr S_2}$ by \cite[Corollary 9.1]{PW}. Therefore
%	\[ \cX(\chi^\alpha;\chi^{(2)}) = \cX(\chi^{\alpha'};\chi^{(2)})\cdot \cX(\sgn_{S_{2^k}}; \chi^{(2)}) = \cX(\chi^{(2^{k-1},2^{k-1})}; \chi^{(2)}) \cdot \Big(\sgn\down^{S_{2^{k+1}}}_{S_{2^k}\wr S_2}\Big) \]
%	is a constituent of
%	\[ \chi^{(2^k,2^k)}\down^{S_{2^{k+1}}}_{S_{2^k}\wr S_2} \cdot \sgn\down^{S_{2^{k+1}}}_{S_{2^k}\wr S_2} = \chi^\lambda\down^{S_{2^{k+1}}}_{S_{2^k}\wr S_2}, \]
%	and so $Z^\lambda>0$. Thus $Z^\lambda=1$.
	
	Next, we consider the general case, i.e.~suppose $\lambda\vdash n=2^{n_1}+\cdots+2^{n_k}$ with $n_1>\cdots>n_k\ge 0$ and $\lambda_1\le 2$. Then
	\[ Z^\lambda = \sum_{\mu^i\vdash 2^{n_i}} c^\lambda_{\mu^1,\dotsc,\mu^k}\cdot Z^{\mu^1}\cdots Z^{\mu^k}. \]
	If $c^\lambda_{\mu^1,\dotsc,\mu^k}>0$ then each $\mu^i$ also has at most two columns, in which case $Z^{\mu^i}=1$ if $\mu^i=(2^{2^{n_i-1}})$, or $\mu^i=(1)$ if $n_i=0$, or $Z^{\mu^i}=0$ otherwise. Hence $Z^\lambda=1$ if $\lambda=(2,\dotsc,2,\varepsilon)$ where $\varepsilon=\delta_{n_k,0}$, and $Z^\lambda=0$ otherwise, as claimed.
\end{proof}

We deduce the values of $Z^\lambda$ for hooks $\lambda$ from \cite{G}.
\begin{proposition}\label{prop:hook}
	%Let $k\in\N_0$ and let $\lambda\vdash 2^k$ be a hook. Then $Z^\lambda=1$ if $\lambda=(2^k)$, and $Z^\lambda=0$ otherwise.
	Let $n\in\N$ and let $\lambda=(n-t,1^t)$ for some $0\le t\le n-1$. Suppose $n=2^{n_1}+\cdots+2^{n_k}$ where $n_1>\cdots>n_k\ge 0$. Then $Z^\lambda=\binom{k-1}{t}$.
\end{proposition}
\begin{proof}
	The case of $n=2^j$ for $j\in\N_0$ follows immediately from \cite[Theorem 1.1]{G}, since $\Irr_{2'}(S_{2^j}):=\{\chi^\lambda\in\Irr(S_{2^j}) \mid 2\nmid\chi^\lambda(1) \}$ consists precisely of those $\chi^\lambda$ where $\lambda\vdash 2^j$ is a hook. In particular, $Z^{(n)}=1$ and $Z^\lambda=0$ for all hooks $\lambda\ne(n)$ when $n=2^j$. Thus for arbitrary $n\in\N$ and $\lambda=(n-t,1^t)$ we have that
	\[ Z^\lambda = \sum_{\mu^i\vdash 2^{n_i}} c^\lambda_{\mu^1,\dotsc,\mu^k}\cdot Z^{\mu^1}\cdots Z^{\mu^k} = c^\lambda_{(2^{n_1}),\dotsc,(2^{n_k})} = \tbinom{k-1}{t}. \]
	where the final equality follows from Theorem~\ref{thm:LR}.
\end{proof}

\begin{lemma}\label{lem:a-2-1b}
	Let $k\in\N_{\ge 2}$. If $\lambda=(2^k-i,2,1^{i-2})\vdash 2^k$ with $2\le i\le 2^k-2$, then $Z^\lambda = \binom{k-1}{k-i}$.
\end{lemma}

\begin{proof}
	We proceed by induction on $k$, and observe that the assertion holds for small $k$ by direct computation. 
	Now suppose $k>2$ and let $\lambda=(2^{k+1}-i,2,1^{i-2})\vdash 2^{k+1}$ for some $2\le i\le 2^{k+1}-2$. Call $S:=S_{2^{k+1}}$ and $P:=P_{2^{k+1}}\le S$. Let $W:=S_{2^k}\wr S_2$ be such that $P\le W\le S$ and set $Y:=S_{2^k}\times S_{2^k}$ and $Q:=P_{2^k}\times P_{2^k}$ such that $Q\le Y\le W$ and $Q\le P$. Then
	\[ Z^\lambda %= \langle (\chi^\lambda\down^S_W)\down^W_P, \triv_P\rangle 
	=\langle \chi^\lambda\down^S_W, \triv_P\up^W\rangle = \sum_{\psi\in\Irr(W)} \langle \chi^\lambda\down^S_W,\psi\rangle \cdot \langle \triv_P\up^W,\psi\rangle.
	\]
	Using Notation~\ref{not:GwrS2} with $G=S_{2^k}$ and $I=\cP(2^k)$, we have that 
	\[ \Irr(W) = \{\psi^{\mu,\nu}=\psi^{\nu,\mu} \mid \mu\ne\nu\in\cP(2^k) \} \sqcup \{ \psi^{\mu,\mu}_{\pm} \mid \mu\in\cP(2^k) \}. \]
	Suppose $\psi\in\Irr(W)$ is such that $\langle \chi^\lambda\down^S_W,\psi\rangle \cdot \langle \triv_P\up^W,\psi\rangle\ne 0$.
	\begin{itemize}
		\item If $\psi=\psi^{\mu,\nu}$ for some $\mu\ne\nu$, then $\langle \chi^\lambda\down^S_W,\psi\rangle %= \langle \chi^\lambda\down^S_W, \chi^\mu\times \chi^\nu\up_Y^W\rangle 
		= \langle \chi^\lambda\down^S_Y, \chi^\mu\times\chi^\nu\rangle = c^\lambda_{\mu,\nu}$. Then $c^\lambda_{\mu,\nu}\ne 0$ implies that $\mu,\nu\subseteq\lambda$, and hence each of $\mu$ and $\nu$ is either a hook or of the form $(a,2,1^b)\vdash 2^k$ for some $a\ge 2$ and $b\ge 0$. Moreover, at least one of $\mu$ and $\nu$ must be a hook, so without loss of generality we may assume that $\mu$ is a hook.
		
		If $\nu=(2^k-j,2,1^{j-2})$ for some $2\le j\le 2^k-2$, then by Lemma~\ref{lem:Prop1.2} we have $\langle \triv_P\up^W,\psi\rangle = Z^\mu\cdot Z^\nu$. By assumption, $\langle \triv_P\up^W,\psi\rangle\ne 0$ and hence $Z^\mu\ne 0$, from which we deduce that $\mu$ is the trivial hook $(2^k)\vdash 2^k$ and $Z^\mu=1$ using Proposition~\ref{prop:hook}.		
		If $\nu$ and $\mu$ are both hooks then similarly we deduce from $Z^\mu\cdot Z^\nu\ne 0$ that $\mu=\nu=(2^k)\vdash 2^k$, a contradiction.
		
		\item If $\psi=\psi^{\mu,\mu}_{\pm}$ for some $\mu$ and choice of sign $\pm$, then $\langle \chi^\lambda\down^S_W,\psi\rangle = \langle \chi^\lambda\down^S_W, \cX(\chi^\mu;\theta)\rangle$ for some $\theta\in\{\chi^{(2)},\chi^{(1^2)}\}$. In either case $\langle \chi^\lambda\down^S_W,\cX(\chi^\mu;\theta)\rangle\ne 0$ implies that $c^\lambda_{\mu,\mu}\ne 0$, hence $\mu$ must be a hook. By Lemma~\ref{lem:Prop1.2}, $\langle \triv_P\up^W,\psi\rangle = \tfrac{1}{2}\cdot((Z^\mu)^2 \pm Z^\mu)$. By assumption, $\langle \triv_P\up^W,\psi\rangle \ne 0$ and hence $Z^\mu\ne 0$, from which we deduce that $\mu=(2^k)\vdash 2^k$ using Proposition~\ref{prop:hook}. Then from $c^\lambda_{(2^k),(2^k)}\ne 0$ we deduce that $i=2$, and also since $Z^{(2^k)}=1$ we must have $\theta=\chi^{(2)}$.
	\end{itemize}
	As a result, for $i\ge 3$ we observe from the inductive hypothesis that
	\[ Z^\lambda = \sum_{j=2}^{2^k-2} c^\lambda_{(2^k), (2^k-j,2,1^{j-2})}\cdot Z^{(2^k-j,2,1^{j-2})} = \sum_{j\in\{i-1,i\}} Z^{(2^k-j,2,1^{j-2})} = \tbinom{k-1}{k-(i-1)} + \tbinom{k-1}{i-1} = \tbinom{(k+1)-1}{(k+1)-i}. \]
	Finally, suppose $i=2$ so $\lambda=(2^{k+1}-2,2)$. From above, we observe that
	\[ Z^\lambda = \sum_{j=2}^{2^k-2} c^\lambda_{(2^k), (2^k-j,2,1^{j-2})}\cdot Z^{(2^k-j,2,1^{j-2})} + \langle \chi^\lambda\down^S_W, \cX(\chi^{(2^k)};\chi^{(2)})\rangle = Z^{(2^k-2,2)} + \langle \chi^\lambda\down^S_W, \cX(\chi^{(2^k)};\chi^{(2)})\rangle. \]
	By the inductive hypothesis, $Z^{(2^k-2,2)}=\binom{k-1}{k-2}=k-1$, while by \cite[Theorem 1.2]{dBPW},
	\[ \langle \chi^{(2^{k+1}-2,2)}\down^S_W, \cX(\chi^{(2^k)};\chi^{(2)}) \rangle = \langle \chi^{(2,2)}\down^{S_4}_{S_2\wr S_2}, \cX(\chi^{(2)};\chi^{(2)}) \rangle = 1, \]
	and so $Z^\lambda=k=\binom{(k+1)-1}{(k+1)-2}$ as required.
\end{proof}

\begin{remark}
	In Section~\ref{sec:13}, we give an alternative proof of Lemma~\ref{lem:a-2-1b} using plethysms (see Corollary~\ref{cor:14.9}), and discuss the case of general $|\lambda|\in\N$ in Remark~\ref{rem:general-a-2-1b}.\hfill$\lozenge$
\end{remark}

\bigskip
%-----------------------------------------------------------------------------------
\subsection{Positivity of Sylow branching coefficients}\label{sec:limit}
The main aim of this section is to prove Theorem C. %First, we describe the decomposition of the permutation character $\triv_{H\wr S_2}\up^{G\wr S_2}$, for finite groups $H\le G$.
%Recall for $A\subseteq\cP(n)$ and $B\subseteq\cP(m)$ that $A\star B:=\{\lambda\vdash m+n \mid \exists\ \mu\in A,\ \nu\in B\ \text{s.t.}\ c^\lambda_{\mu,\nu}>0 \}$.
We recall the definition of the operation $\star$ was given in Section~\ref{sec:prelims-sn}.
\begin{lemma}\label{lem:SBC>=3}
	Let $n\in\N$. Suppose $A_1, A_2\subseteq\{\lambda\vdash n\mid Z^\lambda\ge 3 \}$. Then $A_1\star A_2\subseteq\{\mu\vdash 2n\mid Z^\mu\ge 3\}$.
\end{lemma}

\begin{proof}
	We follow the notation of Lemma~\ref{lem:Prop1.2}, letting $G=S_n$, $H=P_n$, $\pi:=\triv_{P_n}\up^{S_n}$ and $\tilde{\pi}=\triv_{P_{2n}}\up^{S_n\wr S_2}$. The irreducible characters of $G$ are indexed by $I=\cP(n)$.
	Let $\mu\in A_1\star A_2$ and take $\lambda_1\in A_1$, $\lambda_2\in A_2$ such that $c^\mu_{\lambda_1,\lambda_2 }>0$. By Frobenius reciprocity, 
	\[ Z^\mu = \langle\tilde{\pi},\chi^\mu\down^{S_{2n}}_{S_n\wr S_2}\rangle. \]
	Suppose first that $\lambda_1\neq\lambda_2$. In this case, $\psi^{\lambda_1,\lambda_2}=(\chi^{\lambda_1}\times\chi^{\lambda_2})\up^{S_n\wr S_2}$ is an irreducible constituent of $\chi^\mu\down^{S_{2n}}_{S_n\wr S_2}$, because $\langle\psi^{\lambda_1,\lambda_2},\chi^\mu\down^{S_{2n}}_{S_n\wr S_2}\rangle=c^\mu_{\lambda_1,\lambda_2}$. Thus by Lemma~\ref{lem:Prop1.2},
	\[ Z^{\mu} = \langle\tilde{\pi},\chi^\mu\down^{S_{2n}}_{S_n\wr S_2}\rangle \geq \langle\tilde{\pi},\psi^{\lambda_1,\lambda_2}\rangle = \langle\pi,\chi^{\lambda_1}\rangle \cdot \langle\pi,\chi^{\lambda_2}\rangle = Z^{\lambda_1}\cdot Z^{\lambda_2} \geq 9>3. \]
	Now suppose that $\lambda_1=\lambda_2 =:\lambda$. Then at least one of $\psi^{\lambda,\lambda}_{+}$ and $\psi^{\lambda,\lambda}_{-}$ is a constituent of $\chi^\mu\down^{S_{2n}}_{S_n\wr S_2}$, since $\langle\psi^{\lambda,\lambda}_{+}+\psi^{\lambda,\lambda}_{-},\chi^\mu\down^{S_{2n}}_{S_n\wr S_2}\rangle = c^\mu_{\lambda,\lambda}$. Thus by Lemma~\ref{lem:Prop1.2},
	\[ Z^\mu \ge \langle \tilde{\pi}, \psi^{\lambda,\lambda}_+\rangle = \tfrac{1}{2}\big((Z^\lambda)^2+Z^\lambda\big) \quad\text{or}\quad Z^\mu \ge \langle \tilde{\pi}, \psi^{\lambda,\lambda}_-\rangle = \tfrac{1}{2}\big((Z^\lambda)^2-Z^\lambda\big) \]
	and so $Z^\mu\ge 3$ in either case as $Z^\lambda\ge 3$.
\end{proof}

\begin{definition}
	Let $n,w,h\in\N$. We define $\cB_{w,h}(n):=\{\lambda\vdash n\mid \lambda_1\le w,\ l(\lambda)\le h\}.$
\end{definition}

In other words, $\cB_{w,h}(n)$ consists of those partitions of $n$ whose Young diagram is contained inside a $w\times h$ rectangle, i.e.~the Young diagram of $(w^h)$, a rectangle of width $w$ and height $h$. Below, we let ${}^{\star k}$ denote a $k$-fold $\star$-product. That is, $A^{\star k} = A\star A\star \cdots \star A$ ($k$ times).

\begin{proposition}\label{prop:rectangle}
	Let $n\ge 4$ and $k\ge 2$ be natural numbers. Then
	\[ \{(2n-2,2)\}^{\star k} \supseteq \cB_{(k+1)n,k}(2kn). \]
\end{proposition}

The following observation will be used throughout the proof of Proposition~\ref{prop:rectangle}.

\begin{lemma}\label{lem:skew}
	Let $n\ge 2$ be a natural number and suppose that a skew shape $\lambda/\mu$ of size $2n$ is such that
	\begin{itemize}
		\item no two nodes of $[\lambda/\mu]$ lie in the same column, and
		\item $[\lambda/\mu]$ is not (a translation of) $[(2n)]$, nor a disjoint union of $[(2n-1)]$ and $[(1)]$.
	\end{itemize}
	Then, $[\lambda/\mu]$ has a Littlewood--Richardson filling of type $(2n-2,2)$.
\end{lemma}

\begin{proof}[Proof of Proposition~\ref{prop:rectangle}]
	We proceed by induction on $k$, with base cases $k=2$ and $k=3$ holding by direct application of the Littlewood--Richardson (LR) rule.
	For the inductive step, suppose $k\geq 4$ and let $\mu=(\mu_1,\mu_2,\dotsc,\mu_{l(\mu)})\in\cB_{(k+1)n,k}(2kn)$. It suffices to show that there exist either
	\begin{enumerate}[label=\textup{(\alph*)}]
		\item $\tilde{\mu}\in\cB_{kn,k-1}(2(k-1)n)$ such that $\tilde{\mu}\subseteq\mu$ and an LR filling of $[\mu/\tilde{\mu}]$ of type $(2n-2,2)$, or 
		\item $\tilde{\mu}\in\cB_{(k-1)n,k-2}(2(k-2)n)$ such that $\tilde{\mu}\subseteq\mu$ and an LR filling of $[\mu / \tilde{\mu}]$ of some type $\nu\in\{(2n-2,2)\}\star\{(2n-2,2)\}$.
	\end{enumerate}
	Observe first that $\mu$ cannot be a hook as $|\mu|=2kn$ and $\mu_1\le (k+1)n$, $l(\mu)\le k$. Moreover, we must have $\mu_2\le kn$ and so the number of nodes of $[\mu]$ lying outside of the rectangle $[(kn)^{k-1}]$ equals $\max\{\mu_1-kn,0\}+\mu_k$ and is at most $2n$. Indeed, this is immediate if $\mu_1\le kn$, and if $\mu_1>kn$ then we use the inequalities $\mu_1\leq(k+1)n$ and $\mu_1+(k-1)\mu_k\leq 2kn$ to obtain
	\[ \mu_1-kn+\mu_k=\tfrac{k-2}{k-1}\cdot\mu_1 + \tfrac{1}{k-1}(\mu_1+(k-1)\mu_k)-kn \leq \tfrac{k-2}{k-1}\cdot(k+1)n + \tfrac{2kn}{k-1} -kn = 2n. \]
	
	In order to show that there exist an appropriate partition $\tilde{\mu}$ and an LR filling as in either (a) or (b),  we consider the following four cases (i)--(iv), depending on the manner in which $[\mu]$ lies outside of the rectangle $[(kn)^{k-1}]$ (if at all). Examples of each of the four cases are illustrated in Figures~\ref{fig:1} to~\ref{fig:4}, for $n=k=4$. In each figure, the shaded nodes indicate $[\mu/\tilde{\mu}]$, the dotted lines outline the rectangle $[((k+1)n)^k]$, and the dashed lines outline the rectangle $[(kn)^{k-1}]$.
	
	\noindent\textbf{Case (i): $\mu_1\leq kn$ and $l(\mu)\leq k-1$.} The assumptions imply that $(2n-2,2)\subseteq \mu$, and so $[\mu/(2n-2,2)]$ has an LR filling of some type $\tilde{\mu}\in\cB_{kn,k-1}(2(k-1)n)$. Since $c^{\mu}_{\tilde{\mu},(2n-2,2)}=c^\mu_{(2n-2,2),\tilde{\mu}} > 0$, we conclude that (a) holds. An example of case (i) is illustrated in Figure~\ref{fig:1}.
	
	\begin{figure}
		\centering
		\begin{tikzpicture}[scale = 0.8]
			%(14,11,7) \vdash 32
			\draw [step=0.5] (0,0) grid (7,-0.5);
			\draw [step=0.5] (0,-0.5) grid (5.5,-1);
			\draw [step=0.5] (0,-1) grid (3.5,-1.5);
			
			%rectangles
			\draw [dashed] (0,0) rectangle (8,-1.5);
			\draw [dotted] (0,0) rectangle (10,-2);
			
			%skew part
			\draw [pattern=north east lines] (4.5,0) rectangle (7,-0.5);
			\draw [pattern=north east lines] (4.5,-0.5) rectangle (5.5,-1);
			\draw [pattern=north east lines] (3,-1) rectangle (3.5,-1.5);
		\end{tikzpicture}
		\caption{Example of case (i) with $n=k=4$ and $\mu=(14,11,7)\vdash 32$. The dotted lines outline the rectangle $[((k+1)n)^k]=[(20)^4]$, the dashed lines outline the rectangle $[(kn)^{k-1}]=[(16)^3]$, and the shaded nodes indicate $[\mu / \tilde{\mu}]$ where we have chosen $\tilde{\mu}=(9,9,6)$.}
		\label{fig:1}
	\end{figure}
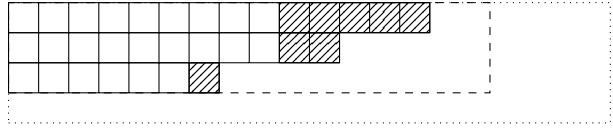
	
	\noindent\textbf{Case (ii): $\mu_1>kn$ and $l(\mu)\le k-1$.} Since it must be that $\mu_2<kn$, the nodes of $[\mu]$ lying outside of the rectangle $[(kn)^{k-1}]$ all lie in the first row. We choose $\tilde{\mu}\vdash |\mu|-2n$ with $\tilde{\mu}\subseteq\mu$ so that the skew shape $[\mu/\tilde{\mu}]$ contains these $\mu_1-kn$ nodes. It is clear that $\tilde{\mu}$ can be chosen such that the conditions of Lemma~\ref{lem:skew} are also satisfied, since $\mu$ is not a hook. Hence we obtain $\tilde{\mu}\in\cB_{kn,k-1}(2(k-1)n)$ as required in (a); see Figure~\ref{fig:2}, for example.
	
	\begin{figure}
		\centering
		\begin{tikzpicture}[scale = 0.8]
			%(18,9,5) \vdash 32
			\draw [step=0.5] grid (0,0) grid (9,-0.5);
			\draw [step=0.5] grid (0,-0.5) grid (4.5,-1);
			\draw [step=0.5] grid (0,-1) grid (2.5,-1.5);
			
			%rectangles
			\draw [dashed] (0,0) rectangle (8,-1.5);
			\draw [dotted] (0,0) rectangle (10,-2);
			
			%skew part
			\draw [pattern=north east lines] (8,0) rectangle (9,-0.5);
			\draw [pattern=north east lines] (4,-0.5) rectangle (4.5,-1);
			\draw [pattern=north east lines] (0,-1) rectangle (2.5,-1.5);
		\end{tikzpicture}
		\caption{Example of case (ii) with $n=k=4$ and $\mu=(18,9,5)\vdash 32$, where we have chosen $\tilde{\mu}=(16,8)$.}
		\label{fig:2}
	\end{figure}
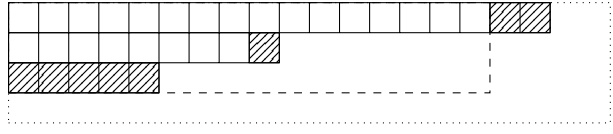
	
	\noindent\textbf{Case (iii): $\mu_1\leq kn$ and $l(\mu)> k-1$.} In this case, the nodes of $[\mu]$ outside of the rectangle $[(kn)^{k-1}]$ are precisely the $\mu_k$ nodes in the $k$th row. Letting $m=\mu_k$, we have that $1\le m\le 2n$ as $|\mu|=2kn$.
	
	We first assume that $m\leq 2n-2$. It is straightforward to choose $\tilde{\mu}\vdash |\mu|-2n$ with $\tilde{\mu}\subseteq\mu$ so that $[\mu/\tilde{\mu}]$ contains the $m$ nodes in the $k$th row and satisfies the conditions in Lemma~\ref{lem:skew}. We are then done as in case (ii), as this shows that (a) holds.
	
	Now suppose instead that $m=2n-1$. Then the nodes of $[\mu]$ outside of the rectangle $[((k-1)n)^{k-2}]$ are precisely the $x:=\mu_{k-1}+\mu_k$ nodes lying in the $k-1$ and $k$th rows. In particular, $\mu_{k-1}\in\{2n-1,2n\}$ since if $\mu_{k-1}\ge 2n+1$ then $|\mu|\ge (k-1)(2n+1)+2n-1>2kn$, a contradiction.
	If $\mu_{k-1}=2n-1$ then we can choose $\tilde{\mu}\vdash |\mu|-4n$ with $\tilde{\mu}\subseteq\mu$ so that $[\mu/\tilde{\mu}]$ consists of the $x=4n-2$ `outside' nodes, and two more nodes which can be chosen to lie in different columns to the right of column $2n-1$ as $\mu_1>2n$.
	If $\mu_{k-1}=2n$ then $\mu=(2n+1,(2n)^{k-2},2n-1)$, so we can take $\tilde{\mu} = ((2n)^{k-2})$. 
	In both instances we observe that $[\mu / \tilde{\mu}]$ has an LR filling of type $(2n+1,2n-1)\in\{(2n-2,2)\}\star\{(2n-2,2)\}$, and so (b) holds. An example is given in Figure~\ref{fig:3}.
	
	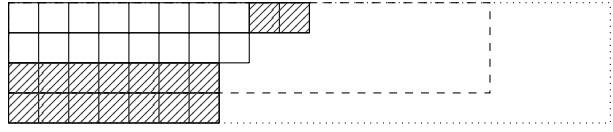
\begin{figure}
		\centering
		\begin{tikzpicture}[scale = 0.8]
			%(10,8,7,7) \vdash 32
			\draw [step=0.5] grid (0,0) grid (5,-0.5);
			\draw [step=0.5] grid (0,-0.5) grid (4,-1);
			\draw [step=0.5] grid (0,-1) grid (3.5,-1.5);
			\draw [step=0.5] grid (0,-1.5) grid (3.5,-2);
			
			%rectangles
			\draw [dashed] (0,0) rectangle (8,-1.5);
			\draw [dotted] (0,0) rectangle (10,-2);
			
			%skew part
			\draw [pattern=north east lines, pattern color=darkgray] (4,0) rectangle (5,-0.5);
			\draw [pattern=north east lines, pattern color=darkgray] (0,-1) rectangle (3.5,-2);
		\end{tikzpicture}
		\caption{Example of case (iii) with $n=k=4$, $\mu=(10,8,7,7)\vdash 32$ and $m=\mu_k=2n-1$, where we have chosen $\tilde{\mu}=(8,8)$.}
		\label{fig:3}
	\end{figure}
	
	Finally, suppose $m=2n$. Then $\mu=((2n)^{k})$, for which we can take $\tilde{\mu}=((2n)^{k-2})$ and $[\mu / \tilde{\mu}]$ has a (unique) LR filling of type $(2n,2n)\in\{(2n-2,2)\}\star\{(2n-2,2)\}$, and so again (b) holds.
	
	\noindent\textbf{Case (iv): $\mu_1 > kn$ and $l(\mu)> k-1$.} The nodes of $[\mu]$ outside of the rectangle $[(kn)^{k-1}]$ lie in the first and $k$th rows. As in case (ii), we can choose $\tilde{\mu}\vdash|\mu|-2n$ with $\tilde{\mu}\subseteq\mu$ so that $[\mu / \tilde{\mu}]$ contains all $\mu_1-kn+\mu_k$ of these nodes and satisfies the conditions of Lemma~\ref{lem:skew}. Thus (a) holds, as desired (see Figure~\ref{fig:4}, for example).	
	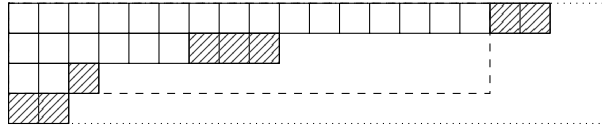
\begin{figure}
		\centering
		\begin{tikzpicture}[scale = 0.8]
			%(19,9,3,2) \vdash 32
			\draw [step=0.5] grid (0,0) grid (9,-0.5);
			\draw [step=0.5] grid (0,-0.5) grid (4.5,-1);
			\draw [step=0.5] grid (0,-1) grid (1.5,-1.5);
			\draw [step=0.5] grid (0,-1.5) grid (1,-2);
			
			%rectangles
			\draw [dashed] (0,0) rectangle (8,-1.5);
			\draw [dotted] (0,0) rectangle (10,-2);
			
			%skew part
			\draw [pattern=north east lines, pattern color=darkgray] (8,0) rectangle (9,-0.5);
			\draw [pattern=north east lines, pattern color=darkgray] (3,-0.5) rectangle (4.5,-1);
			\draw [pattern=north east lines, pattern color=darkgray] (1,-1) rectangle (1.5,-1.5);
			\draw [pattern=north east lines, pattern color=darkgray] (0,-1.5) rectangle (1,-2);
		\end{tikzpicture}
		\caption{Example of case (iv) with $n=k=4$ and $\mu=(19,9,3,2)\vdash 32$, where we have chosen $\tilde{\mu}=(16,6,2)$.}
		\label{fig:4}
	\end{figure}
\end{proof}

\begin{corollary}\label{cor:(14,2)}
	For each $k\in\N_{\ge 3}$, we have $\cB_{8(k+1),k}(16k) \subseteq \{\mu\vdash 16k \mid Z^\mu\ge 3 \}$.
\end{corollary}

\begin{proof}
	We note that $Z^{(14,2)}=3$ and so iterated application of Lemma~\ref{lem:SBC>=3} shows that for each $r\in\N_0$, ${(14,2)}^{\star 2^r} \subseteq \{\mu\vdash 16\cdot 2^r \mid Z^\mu\ge 3 \}$.
	Now, suppose $k=2^{n_1}+\cdots+2^{n_t}$ where $n_1>\cdots>n_t\ge 0$. Since $P_{k} \cong P_{2^{n_1}}\times\cdots\times P_{2^{n_t}}$, we obtain
	\[ {(14,2)}^{\star k} \subseteq \{\mu\vdash 16\cdot2^{n_1} \mid Z^\mu\ge 3 \}\star \cdots \star\{\mu\vdash 16\cdot2^{n_t} \mid Z^\mu\ge 3 \} \subseteq \{\mu\vdash 16k \mid Z^\mu\ge 3 \}. \]
	The assertion then follows from Proposition~\ref{prop:rectangle}.
\end{proof}

\begin{proof}[Proof of Theorem C]
	By \cite[(1.4)]{EL}, we have
	\[ \lim_{k\to\infty}\frac{|\cB_{8(k+1),k}(16k)|}{|\cP(16k)|} = 1. \]
	Now, for each $n\in\N$, write $n=16k+r$ where $k:=\lfloor\tfrac{n}{16}\rfloor$ and $0\le r<16$. By the Littlewood--Richardson rule, we have $Z^{\mu+(r)}\geq Z^{\mu}$ for each $\mu\vdash 16k$. Hence $|\{\lambda\vdash n\mid Z^\lambda>0\}|\geq|\{\mu\vdash16k\mid Z^\mu\geq3\}|$.\\
	Moreover, the partition function $|\cP(n)|$ is subexponential% \cite{HR}
	, so 
	\[ \lim_{n\to\infty}\frac{|\cP(16k)|}{|\cP(n)|} = 1. \]
	Putting these together with Corollary~\ref{cor:(14,2)},
	\[ \frac{|\{\lambda\vdash n\mid Z^\lambda>0\}|}{|\cP(n)|} \geq \frac{|\cP(16k)|}{|\cP(n)|}\cdot \frac{|\{\mu\vdash16k\mid Z^\mu\geq3\}|}{|\cP(16k)|} \geq \frac{|\cP(16k)|}{|\cP(n)|}\cdot \frac{|\cB_{8(k+1),k}(16k)|}{|\cP(16k)|}\to 1 \]
	as $n\to\infty$.
\end{proof}

\begin{example}\label{ex:32-1}
	We consider those partitions $\mu\vdash 32$ such that $Z^\mu=0$. We note that $|\cP(32)|=8349$, and $|\{\mu\vdash 32\mid Z^\mu=0\}|=879$. Of these 879 partitions, we can identify the following using the results in Section~\ref{sec:special-shapes} (note the different properties are not mutually exclusive):
	
	\medskip
	
	\begin{tabular}{p{3.5cm}|p{1cm}p{5.5cm}|p{3.5cm}}
		\textit{Property} & \multicolumn{2}{l|}{\textit{\# of such $\mu\vdash 32$}} & \textit{$Z^\mu=0$} \\
		\hline
		$l(\mu)>16$ & 684 & & from Lemma~\ref{lem:tall}\\
		$\mu$ has $\le 2$ columns & 16 & i.e.~$\mu=(2^a,1^{32-2a})$ for $0\le a\le 15$ & from Lemma~\ref{lem:2col}\\
		$\mu$ is a non-trivial hook & 31 & i.e.~$\mu=(32-l,1^l)$ for $1\le l\le 31$ & from Proposition~\ref{prop:hook}\\
		$\mu$ is of form $(a,2,1^b)$ & 25 & for $2\le a\le 26$ & from Lemma~\ref{lem:a-2-1b}\\
	\end{tabular}
	
	\medskip
	
	\noindent In total, the above partitions cover 710 out of 879 of those $\mu\vdash 32$ such that $Z^\mu=0$. 
	
	In the next section, we move on to the second half of this article centering on plethysm coefficients. Using applications of plethysms to Sylow branching coefficients in Section~\ref{sec:13}, we will in fact be able to identify almost all of the remaining 169 partitions; see Example~\ref{ex:32-2}.\hfill$\lozenge$
\end{example}

\bigskip
%-----------------------------------------------------------------------------------
\section{Plethysms and character deflations}\label{sec:pre-recursive-formula}
We record a result of Briand--Orellana--Rosas on plethysm coefficients from which we can deduce the $k=0$ case of Theorem A (Theorem~\ref{thm:4.4.3}), as well as resolve two conjectures of de Boeck in Section~\ref{sec:resolve} (Theorems~\ref{thm:6.5.1} and~\ref{thm:6.5.2}).

%We will use the following notation from \cite[Figure 1]{BOR}.% to refer to the `complement' of a partition.
\begin{definition}
	Let $\lambda$ be a partition and $w,h\in\N_0$ such that $\lambda\subseteq (w^h)$. Then $\square_{w,h}(\lambda)$ denotes the partition of size $wh-|\lambda|$ whose Young diagram is the $180^\circ$ rotation of the complement of the Young diagram of $\lambda$ in a rectangle of width $w$ and height $h$.
\end{definition}

\begin{theorem}[{\cite[Theorem 12]{BOR}}]\label{thm:BOR}
	Let $m,n\in\N_0$ and let $\lambda$, $\mu$ and $\nu$ be partitions such that $\mu\subseteq(w^h)$ and $l(\nu)\le h$. If $\nu\subseteq \big((w|\lambda|)^h\big)$ then
	\[ a^\nu_{\lambda,\mu} = a^{\square_{w|\lambda|,h}(\nu)}_{\lambda,\square_{w,h}(\mu)}. \]
	Otherwise $a^\nu_{\lambda,\mu}=0$.
\end{theorem}

Let $m,n\in\N$ and $\mu\vdash mn$. We recall the definition of $\delta^\mu=\operatorname{Def}^{\triv_{S_m}}_{S_n}(\chi^\mu\down^{S_{mn}}_{S_m\wr S_n})$, the deflation of $\mu$ with respect to $S_n$, from Definition~\ref{def:deflation}.
\begin{proposition}\label{prop:12.4}
	Fix $n\in\N$ and $m_1,m_2\in\N$. Let $\mu_1\vdash m_1n$ and suppose $\mu_1\subseteq \big((m_1+m_2)^n\big)$. Set $\mu_2=\square_{m_1+m_2,n}(\mu_1)\vdash m_2n$. Then
	\[ \delta^{\mu_1} = \begin{cases}
		\delta^{\mu_2} & \text{if }m_1+m_2\text{ is even},\\
		\sgn_{S_n}\cdot\ \delta^{\mu_2} & \text{if }m_1+m_2\text{ is odd},
	\end{cases}
	\]
	where $\delta^{\mu_i}$ refers to the deflation of $\mu_i$ with respect to $S_n$, for $i\in\{1,2\}$.
	%\[ \operatorname{Defres}_{S_n}(\chi^{\mu_1}) = \begin{cases}
	%	\operatorname{Defres}_{S_n}(\chi^{\mu_2}) & \text{if }m_1+m_2\text{ is even},\\
	%	\sgn_{S_n}\cdot\operatorname{Defres}_{S_n}(\chi^{\mu_2}) & \text{if }m_1+m_2\text{ is odd}.
	%\end{cases} \]
\end{proposition}

\begin{proof}
	Since $(1^{m_1})\subseteq(1^{m_1+m_2})$ and $l(\mu'_1)\le m_1+m_2$, as well as $\mu'_1\subseteq \big((1\cdot|\lambda|)^{m_1+m_2}\big)$, we see from Theorem~\ref{thm:BOR} that
	\[ a^{\mu_1'}_{\lambda,(1^{m_1})} = a^{\square_{n,m_1+m_2}(\mu_1')}_{\lambda,\square_{1,m_1+m_2}\big((1^{m_1})\big)} = a^{\mu'_2}_{\lambda,(1^{m_2})}\qquad\forall\ \lambda\vdash n. \]
	Applying Lemma~\ref{lem:conj-plethysm} and observing that $\lambda\mapsto\lambda'$ is an involution on the set of partitions of $n$, we obtain
	\[ a^{\mu_1}_{\lambda,(m_1)} = a^{\mu_2}_{\lambda^\ast,(m_2)} \]
	for all $\lambda\vdash n$, where $\lambda^\ast$ denotes $\lambda$ (resp.~$\lambda'$) if $m_1+m_2$ is even (resp.~odd). 
	Since $\langle \delta^\mu, \chi^\lambda\rangle = a^\mu_{\lambda,(m)}$, we have shown that
	\[ \langle \delta^{\mu_1},\chi^\lambda\rangle = \begin{cases}
		\langle \delta^{\mu_2},\chi^\lambda\rangle & \text{if $m_1+m_2$ is even},\\ \langle \delta^{\mu_2}, \chi^{\lambda'}\rangle & \text{if $m_1+m_2$ is odd}.\end{cases} \]
	Since %$\Irr(S_n)=\{\chi^\lambda\mid\lambda\vdash n\}$ gives a basis of class functions and 
	$\chi^{\lambda'}=\sgn_{S_n}\cdot\ \chi^\lambda$, we have $\delta^{\mu_1}=\delta^{\mu_2}$ if $m_1+m_2$ is even (resp.~$\delta^{\mu_1}=\sgn_{S_n}\cdot\ \delta^{\mu_2}$ if $m_1+m_2$ is odd), as required.
\end{proof}

The following result was first proven in \cite[Proposition 1.16]{BCV}, and is a special case of \cite[Theorem 1.1]{dBPW}. It also forms the $k=0$ case of Theorem A, since $\sum_{\alpha,\beta\vdash 0} \big(a^{\alpha/\emptyset}_{\beta',(m)}\cdot a^{\hat{\mu}/\alpha}_{\lambda/\beta,(m-1)}\big) = a^{\hat{\mu}}_{\lambda,(m-1)}$. We give a short proof in the language of deflations.

\begin{theorem}\label{thm:4.4.3}
	Let $m,n\in\N$, $\lambda\vdash mn$ and $\nu\vdash n$. Then $a^\lambda_{\nu,(m)} = a^{\lambda+(1^n)}_{\nu',(m+1)}$.
\end{theorem}

\begin{proof}
	From Lemma~\ref{lem:tall-pleth}, $a^\lambda_{\nu,(m)} = a^{\lambda+(1^n)}_{\nu',(m+1)}=0$ if $l(\lambda)>n$, %First suppose $l(\lambda)>n$ and let $a:=a^\lambda_{\nu,(m)}= \langle \chi^\lambda\down^{S_{mn}}_{S_m\wr S_n}, \cX(\chi^{(m)};\chi^\nu)\rangle$. Let $Y:=(S_m)^{\times n}\le S_m\wr S_n$. Since $\cX(\chi^{(m)};\chi^\nu)\down_Y = \chi^\nu(1)\cdot(\chi^{(m)})^n=\chi^\nu(1)\cdot\triv_Y$, it follows that $\langle \chi^\lambda\down_Y,\triv_Y\rangle \ge a\cdot\chi^\nu(1)$. But by the Littlewood--Richardson rule, $l(\lambda)>n$ implies $\langle \chi^\lambda\down_Y,\triv_Y\rangle=0$, which gives $a=0$. Clearly $l\big(\lambda+(1^n)\big)>n$ also, and thus by the same argument $a^{\lambda+(1^n)}_{\nu',(m+1)}=0$ as $\chi^{(m+1)}=\triv_{S_{m+1}}$.
	so we may now assume $l(\lambda)\le n$. 
	%By Lemma~\ref{lem:conj-plethysm}, $a^\lambda_{\nu,(m)} = a^{\lambda'}_{\nu^\ast, (1^m)}$ where $\nu^\ast=\nu$ if $m$ is even, and $\nu^\ast=\nu'$ if $m$ is odd. Set $w:=\lambda_1\ge m$. Since $(1^m)\subseteq (1^w)$ and $\lambda'\subseteq (n^w)$, it follows from Theorem~\ref{thm:BOR} that
	%\[ a^{\lambda'}_{\nu^\ast,(m)} = a^{\square_{n,w}(\lambda')}_{\nu^\ast,\square_{1,w}(1^m)} = a^{\delta'}_{\nu^\ast,(1^{w-m})}  \]
	%where $\delta:=\square_{w,n}(\lambda)$. Applying Theorem~\ref{thm:BOR} again, observing that $l(\delta')<w+1$, $(1^{w-m})\subseteq (1^{w+1})$ and $\delta'\subset (n^{w+1})$, we have
	%\[ a^{\delta'}_{\nu^\ast,(1^{w-m})} = a^{\square_{n,w+1}(\delta')}_{\nu^\ast,(1^{m+1})} = a^{\tilde{\lambda}'}_{\nu^\ast,(1^{m+1})} \]
	%where $\tilde{\lambda}:=\lambda+(1^n)$. But then $a^{\tilde{\lambda}'}_{\nu^\ast,(1^{m+1})} = a^{\tilde{\lambda}}_{\nu',(m+1)}$ by Lemma~\ref{lem:conj-plethysm}, since exactly one of $m$ and $m+1$ is even.
	Since $\chi^{\nu'}=\sgn_{S_n}\cdot\ \chi^\nu$, the assertion is equivalent to proving that $\delta^\lambda=\sgn_{S_n}\cdot\ \delta^{\lambda+(1^n)}$. %This follows from a similar argument to that in the proof of Corollary~\ref{cor:10.1,2,4}(i). Specifically, c
	Choose $k\in\N$ such that $\lambda\subseteq(k^n)$, say $k=mn$. Then by Proposition~\ref{prop:12.4},
	\[ \delta^\lambda =s_1\cdot \ \delta^{\square_{k,n}(\lambda)} = s_1\cdot s_2\cdot\ \delta^{\square_{1+k,n}(\square_{k,n}(\lambda))} = s_1\cdot s_2\cdot\ \delta^{\lambda+(1^n)}, \]
	where $\{s_1,s_2\}=\{\triv_{S_n},\sgn_{S_n}\}$.
\end{proof}

%\begin{remark}
%	Theorem~\ref{thm:4.4.3} was also conjectured in \cite[Conjecture 7.2]{dB-ch4}, \cite[Conjecture 4.4.3]{dB-thesis}. 
%	We note that \cite[Conjecture 7.3]{dB-ch4}, \cite[Conjecture 4.4.1]{dB-thesis} follows as an immediate corollary.\hfill$\lozenge$
%\end{remark}

\bigskip
%-----------------------------------------------------------------------------------
\subsection{Resolving conjectures on plethysm coefficients}\label{sec:resolve}
In \cite{W}, Weintraub conjectured that if $m,n\in\N$ with $m$ even, and $\lambda\vdash mn$ is an even partition with $l(\lambda)\le n$, then $a^\lambda_{(n),(m)}>0$. An asymptotic version of the conjecture was proven in \cite{Man}, and the conjecture was first proven in full in \cite{BCI} using techniques from quantum information theory, and reproven in \cite{MM} by considering highest weight vectors. The following sharpening of Weintraub's conjecture was posed in \cite[Conjecture 6.5.1]{dB-thesis}.

\begin{theorem}\label{thm:6.5.1}
	Let $m,n\in\N$ and let $\lambda\vdash mn$ with $l(\lambda)\le n$ and $\lambda_1=m+2$.
	\begin{enumerate}[label=\textup{(\alph*)}]
		\item Suppose $m$ is even. If $\lambda$ has all parts even, then $a^\lambda_{(n),(m)}=1$. Otherwise, $a^\lambda_{(n),(m)}=0$.
		\item Suppose $m$ is odd. If $\lambda$ has all parts odd, then $a^\lambda_{(1^n),(m)}=1$. Otherwise, $a^\lambda_{(1^n),(m)}=0$.
	\end{enumerate}
\end{theorem}

\begin{proof}
	Let $\nu:=\square_{m+2,n}(\lambda)$. 
	
	\noindent\textbf{(a)} By Lemma~\ref{lem:conj-plethysm} and Theorem~\ref{thm:BOR}, noting that $(1^m)\subseteq(1^{m+2})$, $l(\lambda')=m+2$ and $\lambda'\subseteq (n^{m+2})$,
	\[ a^\lambda_{(n),(m)} = a^{\lambda'}_{(n),(1^m)} = a^{\square_{n,m+2}(\lambda')}_{(n),\square_{1,m+2}((1^m))} = a^{\nu'}_{(n),(1^2)} = a^\nu_{(n),(2)}. \]
	Since $m$ is even then $\lambda$ is an even partition if and only if $\nu$ is an even partition. Moreover, $a^\gamma_{(n),(2)}=1$ when $\gamma\vdash 2n$ is an even partition by Proposition~\ref{prop:thrall} and $a^\gamma_{(n),(2)}=0$ otherwise. The assertion follows.
	
	\noindent\textbf{(b)} Similarly to case (a), we obtain from Lemma~\ref{lem:conj-plethysm} and Theorem~\ref{thm:BOR}
	\[ a^\lambda_{(1^n),(m)} = a^{\lambda'}_{(n),(1^m)} = a^{\square_{n,m+2}(\lambda')}_{(n),\square_{1,m+2}((1^m))} = a^{\nu'}_{(n),(1^2)} = a^\nu_{(n),(2)}. \]
	Since $m$ is odd, $\lambda$ has all parts odd if and only if $\nu$ has all parts even, whence the assertion follows again from Proposition~\ref{prop:thrall}.
\end{proof}

The maximal and minimal partitions $\lambda$ with respect to dominance labelling a Schur function $s_\lambda$ in a plethysm of two arbitrary Schur functions were determined combinatorially using certain collections of tableaux in \cite{PW} and \cite{dBPW}. Below, we prove a conjecture of de Boeck \cite[Conjecture 6.5.2]{dB-thesis} describing certain minimal constituents satisfying a parity restriction on the parts of the partition.

\begin{theorem}\label{thm:6.5.2}
	Let $m,n\in\N_{\ge 3}$.
	\begin{enumerate}[label=\textup{(\alph*)}]
		\item Suppose $m$ is even. Then the lexicographically smallest partition $\lambda\vdash mn$ such that $a^\lambda_{(n),(m)}>0$ and $\lambda$ has an odd part is $\lambda=(m+3,m^{n-2},m-3)$.
		\item Suppose $m$ is odd. Then the lexicographically smallest partition $\lambda\vdash mn$ such that $a^\lambda_{(1^n),(m)}>0$ and $\lambda$ has an even part is $\lambda=(m+3,m^{n-2},m-3)$.
	\end{enumerate}
\end{theorem}

\begin{proof}
	We note by Lemma~\ref{lem:tall-pleth} that $a^\lambda_{(n),(m)}=a^\lambda_{(1^n),(m)}=0$ whenever $l(\lambda)>n$.
	
	\noindent\textbf{(a)} First, let $\lambda:=(m+3,m^{n-2},m-3)$. We show that $a^\lambda_{(n),(m)}>0$. Applying Lemma~\ref{lem:conj-plethysm} and Theorem~\ref{thm:BOR}, noting that $(1^m)\subseteq(1^{m+3})$, $l(\lambda')=m+3$ and $\lambda'\subseteq(n^{m+3})$, we obtain
	\[ a^\lambda_{(n),(m)} = a^{\lambda'}_{(n),(1^m)} = a^{\square_{n,m+3}(\lambda')}_{(n),\square_{1,m+3}((1^m))} = a^{((n-1)^3,1^3)}_{(n),(1^3)} = a^{(6,3^{n-2})}_{(1^n),(3)}. \]
	But by \cite[Theorem 5.1.1]{dB-thesis},
	\[ a^{(6,3^{n-2})}_{(1^n),(3)} \ge a^{(6,3^{n-3})}_{(1^{n-1}),(3)} \ge \cdots \ge a^{(6,3)}_{(1^3),(3)}=1. \]
	Next, since $(m^n)$ itself has no odd parts, it remains to show that $a^\lambda_{(n),(m)}=0$ for every $\lambda\vdash mn$ lying strictly between $(m^n)$ and $(m+3,m^{n-2},m-3)$ in lexicographical order and containing an odd part. We may further assume that $l(\lambda)\le n$, using Lemma~\ref{lem:tall-pleth} as noted above. Such $\lambda$ must satisfy one of:
	\[ \lambda_1=m+1;\quad \lambda_1=m+2;\quad \lambda=(m+3,m^{n-4},(m-1)^3)\ \ \text{where $n\ge 4$};\quad \lambda=(m+3,m^{n-3},m-1,m-2). \]
	We show that $a^\lambda_{(n),(m)}=0$ in each of these cases when $\lambda$ has an odd part:
	\begin{itemize}
		\item If $\lambda_1=m+1$, then similarly to above we have
		\[ a^\lambda_{(n),(m)} = a^{\lambda'}_{(n),(1^m)} = a^{\square_{n,m+1}(\lambda')}_{(n),\square_{1,m+1}((1^m))} = a^{\gamma'}_{(n),(1)}, \]
		where $\gamma:=\square_{m+1,n}(\lambda)$. But $\lambda_1=m+1$ so $\gamma\ne(1^n)$ and hence $\gamma'\ne(n)$, giving $a^{\gamma'}_{(n),(1)}=0$.
		
		\item If $\lambda_1=m+2$, then $a^\lambda_{(n),(m)}=0$ by Theorem~\ref{thm:6.5.1} since $\lambda$ has an odd part.
		%	If $\lambda_1=m+2$, then
		%	\[ a^\lambda_{(n),(m)} = a^{\lambda'}_{(n),(1^m)} = a^{\square_{n,m+2}(\lambda')}_{(n),\square_{1,m+2}((1^m))} = a^{\gamma'}_{(n),(1^2)} = a^\gamma_{(n),(2)}, \]
		%	where $\gamma:=\square_{m+2,n}(\lambda)$. Since $\lambda$ has an odd part and $m$ is even, $\gamma$ has an odd part, and so $a^\gamma_{(n),(2)}=0$ by Proposition~\ref{prop:thrall}.
		
		\item If $n\ge 4$ and $\lambda=(m+3,m^{n-4},(m-1)^3)$, then
		\[ a^\lambda_{(n),(m)} = a^{\lambda'}_{(n),(1^m)} = a^{\square_{n,m+3}(\lambda')}_{(n),\square_{1,m+3}((1^m))} = a^{\nu}_{(n),(1^3)} \]
		where $\nu:=((n-1)^3,3)$. But then applying Theorem~\ref{thm:BOR} again gives
		\[ a^\nu_{(n),(1^3)} = a^{\square_{n,4}(\nu)}_{(n),\square_{1,4}((1^3))} = a^{(n-3,1^3)}_{(n),(1)} = 0. \]
		
		\item Finally, if $\lambda=(m+3,m^{n-3},m-1,m-2)$ then
		\[ a^\lambda_{(n),(m)} = a^{\lambda'}_{(n),(1^m)} = a^{\square_{n,m+3}(\lambda')}_{(n),\square_{1,m+3}((1^m))} = a^{\nu}_{(n),(1^3)} \]
		where $\nu:=((n-1)^3,2,1)$. But then
		\[ a^\nu_{(n),(1^3)} = a^{\square_{n,5}(\nu)}_{(n),\square_{1,5}((1^3))} = a^{(n-1,n-2,1^3)}_{(n),(1^2)} = a^{(5,2^{n-3},1)}_{(n),(2)}=0 \]
		where the final equality holds by Proposition~\ref{prop:thrall}.
	\end{itemize}
	
	\noindent\textbf{(b)} %Suppose $a^\lambda_{(1^n),(m)}>0$, i.e.~$\langle\cX(\chi^{(m)};\chi^{(1^n)}), \chi^\lambda\down^{S_{mn}}_{(S_m)^{\times n}}\rangle>0$. Since $\cX(\chi^{(m)};\chi^{(1^n)})\down_{(S_m)^{\times n}}= (\chi^{(m)})^n$, we have that $\langle \chi^\lambda\down_{(S_m)^{\times n}},\triv_{(S_m)^{\times n}}\rangle >0$. By the Littlewood--Richardson rule, this implies that $l(\lambda)\le n$.
	Let $\lambda:=(m+3,m^{n-2},m-3)$. Similarly to case (a), we have
	\[ a^\lambda_{(1^n),(m)} = a^{\lambda'}_{(n),(1^m)} = a^{\square_{n,m+3}(\lambda')}_{(n),\square_{1,m+3}((1^m))} = a^{((n-1)^3,1^3)}_{(n),(1^3)} = a^{(6,3^{n-2})}_{(1^n),(3)} >0 \]
	where we note the first equality holds since $m$ is now odd.	
	Since $(m^n)$ itself does not contain an even part, it remains to consider all $\lambda\vdash mn$ strictly between $(m^n)$ and $(m+3,m^{n-2},m-3)$ in lexicographical order. By the same argument as in (a), we obtain $a^\lambda_{(1^n),(m)}=a^{\lambda'}_{(n),(1^m)}=0$ for such $\lambda$ when $\lambda$ has an even part, noting that $\lambda$ has an even part if and only if $\gamma:=\square_{m+2,n}(\lambda)$ has an odd part as $m$ is now odd.
\end{proof}

\begin{remark}
	In fact, we need not have used \cite[Theorem 5.1.1]{dB-thesis} in the proof of Theorem~\ref{thm:6.5.2} to deduce that $a^{(6,3^{n-2})}_{(1^n),(3)}>0$: we show that $a^{(6,3^{n-2})}_{(1^n),(3)}=1$ for all $n\ge 3$ in Example~\ref{ex:14.7} below.\hfill$\lozenge$
\end{remark}

\bigskip
%-----------------------------------------------------------------------------------
\section{A recursive formula for plethysm coefficients}\label{sec:14}
The main result of this article is Theorem A, a recursive formula for plethysm coefficients of the form $a^\mu_{\lambda,(m)}$ for arbitrary $m\in\N$ and partitions $\mu$ and $\lambda$. Together with Lemma~\ref{lem:tall-pleth}, it describes the deflations $\delta^\mu$ of $\mu\vdash mn$ with respect to $S_n$, noting that $a^\mu_{\lambda',(m)}=\langle \sgn_{S_n}\cdot\ \delta^\mu,\chi^\lambda\rangle$ for $\lambda\vdash n$. We restate Theorem A here as Theorem~\ref{thm:14.6i} for ease of reference for the reader, and recall that plethysm coefficients indexed by skew shapes were defined in \eqref{eqn:skew-plethysm}:

\begin{theorem}[Theorem A]\label{thm:14.6i}
	Fix $n\in\N$. Let $m\in\N$, $k\in\{0,1,\dotsc,n-1\}$ and $\lambda\vdash n$. Let $\mu\vdash mn$ with $l(\mu)=n-k$, and set $\hat{\mu}:=\mu-(1^{n-k})\vdash (m-1)n+k$. Then
	\begin{equation}\label{eqn:A}
		a^\mu_{\lambda',(m)} = \sum_{i=0}^k (-1)^{k+i} \cdot \sum_{\substack{\alpha\vdash k+(m-1)i\\\beta\vdash i}} a^{\alpha/(k-i)}_{\beta',(m)}\cdot a^{\hat{\mu}/\alpha}_{\lambda/\beta,(m-1)}.
	\end{equation}
\end{theorem}

We first illustrate some of the uses of our main theorem in Example~\ref{ex:14.7} and in proving a stability result (Proposition~\ref{prop:17}), before proving Theorem~\ref{thm:14.6i} in full in Section~\ref{sec:16}. We present further applications to Sylow branching coefficients in Section~\ref{sec:13}.

\begin{example}\label{ex:14.7}
	We illustrate how to compute $a^{(6,3^{n-2})}_{\lambda,(3)}$ for all $n\ge 6$ and $\lambda\vdash n$ using Theorem~\ref{thm:14.6i}. % using Theorem~\ref{thm:14.6i}.
	Let $\mu=(6,3^{n-2})$. Since $l(\mu)=n-1$, Theorem~\ref{thm:14.6i} gives
	\[ a^\mu_{\lambda,(3)} = - a^{\hat{\mu}/(1)}_{\lambda',(2)} + a^{\hat{\mu}/(3)}_{\lambda'/(1),(2)} = -a^{(4,2^{n-2})}_{\lambda',(2)} - a^{(5,2^{n-3},1)}_{\lambda',(2)} + a^{(2^{n-1})}_{\lambda'/(1),(2)} + a^{(3,2^{n-3},1)}_{\lambda'/(1),(2)} + a^{(4,2^{n-3})}_{\lambda'/(1),(2)}, \]
	since $a^\emptyset_{\emptyset,(3)}=1$ and $a^\alpha_{(1),(3)}=\delta_{\alpha,(3)}$. Applying Theorems~\ref{thm:14.6i} and~\ref{thm:LR}, we obtain
	\begin{align*}
		a^{(4,2^{n-2})}_{\lambda',(2)} &= \begin{cases} 1 & \text{if }\lambda'\in\{(n),(n-1,1),(n-2,2)\},\\ 0&\text{otherwise},\end{cases}\\
		a^{(5,2^{n-3},1)}_{\lambda',(2)} &= \begin{cases} 1 & \text{if }\lambda'\in\{(n-1,1),(n-2,2),(n-2,1^2),(n-3,2,1)\},\\ 0&\text{otherwise}.\end{cases}
	\end{align*}
	Similarly, since $a^\theta_{\lambda'/(1),(2)} = \sum_{\varepsilon\vdash n-1}c^{\lambda'}_{\varepsilon,(1)}\cdot a^\theta_{\varepsilon,(2)}$, we have that
	\begin{align*}
		a^{(2^{n-1})}_{\lambda'/(1),(2)} &= c^{\lambda'}_{(n-1),(1)} = \begin{cases}
			1 & \text{if }\lambda'\in\{(n),(n-1,1)\},\\ 0&\text{otherwise},\end{cases}\\
		a^{(3,2^{n-3},1)}_{\lambda'/(1),(2)} &= c^{\lambda'}_{(n-2,1),(1)} = \begin{cases}
			1 & \text{if }\lambda'\in\{(n-1,1),(n-2,2),(n-2,1^2)\},\\ 0&\text{otherwise},\end{cases}\\
		a^{(4,2^{n-3})}_{\lambda'/(1),(2)} &= c^{\lambda'}_{(n-2,1),(1)} + c^{\lambda'}_{(n-1),(1)} + c^{\lambda'}_{(n-3,2),(1)}\\
		&= \begin{cases}
			2 & \text{if }\lambda'\in\{(n-1,1),(n-2,2)\},\\
			1 & \text{if }\lambda'\in\{(n),(n-2,1^2),(n-3,3),(n-3,2,1)\},\\
			0 & \text{otherwise}.
		\end{cases}
	\end{align*}
	Putting this together, we obtain
	\[ a^\mu_{\lambda,(3)} = \begin{cases}
		2 & \text{if }\lambda'=(n-1,1),\\
		1 & \text{if }\lambda'\in\{(n),(n-2,2),(n-2,1^2),(n-3,3)\},\\
		0 & \text{otherwise}.
	\end{cases} \]
	In particular, this gives an alternative method for one of the steps in the proof of Theorem~\ref{thm:6.5.2} by showing that $a^{(6,3^{n-2})}_{(1^n),(3)}=1$ for all $n\ge 3$ (note when $n\le 5$ this follows by direct computation).\hfill$\lozenge$ %also 1 when $n=2$ in fact
\end{example}

A corollary of Theorem~\ref{thm:14.6i} is the stability of the following sequence of plethysm coefficients, whose monotonicity was predicted in \cite[Conjecture 1.2]{BBP}.

\begin{proposition}\label{prop:17}
	Let $\lambda$ and $\mu$ be partitions. For all $j\in\N_0$, define partitions $\lambda^j:=\lambda\sqcup(1^j)$ and $\mu^j:=(\mu+(j))\sqcup(1^j)$. Then the sequence $\big(a^{\mu^j}_{\lambda^j,(2)}\big)_{j\in\N_0}$ is eventually constant.
\end{proposition}

To prove Proposition~\ref{prop:17}, we first record a stability property of Littlewood--Richardson coefficients. For convenience we include a proof in our present notation.
\begin{lemma}\label{lem:lr-stab}
	Let $\alpha,\beta,\gamma$ and $\delta$ be partitions. Define $\gamma(j):=\gamma+(j)$ and $\delta(j):=\delta+(j)$ for all $j\in\N_0$. Then the sequence $\big( \langle\chi^{\gamma(j)/\alpha},\chi^{\delta(j)/\beta}\rangle \big)_{j\in\N_0}$ is non-decreasing and eventually constant.
\end{lemma}

\begin{proof}
	For a skew shape $\rho$, let $\cL(\rho)$ denote the set of all Littlewood--Richardson fillings of $\rho$. Let $\ff_j:\cL(\gamma(j)/\alpha)\to \cL(\gamma(j+1)/\alpha)$ be the map given by filling the additional box in $\gamma(j+1)/\alpha - \gamma(j)/\alpha$ with the number 1. Clearly $\ff_j$ is well-defined and injective for all $j$, and furthermore, bijective for all sufficiently large $j$ (for example, $j\ge \alpha_1+\gamma_2$ will suffice).
	Similarly define $\fg_j:\cL(\delta(j)/\beta)\to\cL(\delta(j+1)/\beta)$. For $\mathsf{s}\in\cL(\gamma(j)/\alpha)$ and $\mathsf{t}\in\cL(\delta(j)/\beta)$, note that $\mathsf{s}$ and $\mathsf{t}$ have the same type if and only if $\ff_j(\mathsf{s})$ and $\fg_j(\mathsf{t})$ have the same type. The assertion of the lemma follows since $\langle\chi^{\gamma(j)/\alpha},\chi^{\delta(j)/\beta}\rangle = |\{(\mathsf{s},\mathsf{t})\in\cL(\gamma(j)/\alpha)\times\cL(\delta(j)/\beta) \mid \mathsf{s}, \mathsf{t}\text{ have the same type} \}|$.
\end{proof}

\begin{proof}[Proof of Proposition~\ref{prop:17}]
	We may assume that $\lambda\vdash n$, $\mu\vdash 2n$ and $l(\mu)=n-k$ for some $n$ and $k\in\N_0$, else $a^{\mu^j}_{\lambda^j,(2)}=0$ for all $j$ by Lemma~\ref{lem:tall-pleth}. Since $\mu^j\vdash 2n+2j$ and $l(\mu^j)=(n+j)-k$, we have from Theorem~\ref{thm:14.6i} that
	\[ a^{\mu^j}_{\lambda^j,(2)} = \sum_{i=0}^k (-1)^{k+i}\cdot \sum_{\substack{\alpha\vdash k+i\\\beta\vdash i}} a^{\alpha/(k-i)}_{\beta',(2)}\cdot a^{\hat{\mu^j}/\alpha}_{{\lambda^j}'/\beta,(1)} \]
	for all $j$, where $\hat{\mu^j}:=\mu^j-(1^{n+j-k})$. The proof is concluded by observing that $a^{\hat{\mu^j}/\alpha}_{{\lambda^j}'/\beta,(1)}=\langle \chi^{{\lambda^j}'/\beta},\chi^{\hat{\mu^j}/\alpha}\rangle$ and using Lemma~\ref{lem:lr-stab} with $\gamma:=\hat{\mu}$ and $\delta:=\lambda'$.
\end{proof}

\begin{remark}
	A similar argument can be used to give a new proof that the sequence $\big(a^{\mu^j}_{\lambda^j,(2)}\big)_j$ also stabilises where $\lambda^j:=\lambda+(j)$ and $\mu^j=\mu\sqcup(2^j)$%$\mu^j:= (\mu'+(j,j))'$
	; this sequence is already known to be both non-decreasing and eventually constant by \cite[\textsection 2.6 Corollary 1]{Brion}.\hfill$\lozenge$
\end{remark}

\bigskip
%----------------------------------------------
\subsection{Proof of Theorem~\ref{thm:14.6i}}\label{sec:16}
We first introduce some notation in preparation for the proofs to come.
\begin{notation}\label{not:rho}
	\begin{enumerate}[label=\textup{(\roman*)}]
		\item Let $\lambda$ be a partition, $n\in\N_0$ and $\phi$ be a virtual character of $S_n$.
		We define
		\[ \phi/\chi^\lambda:= \sum_{\mu\vdash n}\langle \phi,\chi^\mu\rangle\cdot\chi^{\mu/\lambda} \]
		where $\chi^{\mu/\lambda}=0$ if $\lambda\not\subseteq\mu$.
		
		\item For $m,n\in\N$ and $\alpha/\beta$ a skew shape of size $n$, define $\rho^{\alpha/\beta}_m := \cX(\triv_{S_m};\chi^{\alpha/\beta})\up_{S_m\wr S_n}^{S_{mn}}$. 
		
		\item For $\phi\in\Ch(S_{n_1})$ and $\theta\in\Ch(S_{n_2})$, define
		\[ \phi\boxtimes\theta:= (\phi\times\theta)\up_{S_{n_1}\times S_{n_2}}^{S_{n_1+n_2}}. \]
		
		\item Let $S_\lambda$ denote the Young subgroup $S_{\lambda_1}\times\cdots\times S_{\lambda_{l(\lambda)}}$ of $S_n$ and let
		\[ \zeta^\lambda:=\triv_{S_\lambda}\up^{S_n} = \chi^{(\lambda_1)}\boxtimes\cdots\boxtimes \chi^{(\lambda_{l(\lambda)})} \]
		denote the character of the permutation module $M^\lambda$. (Note $\zeta^\mu=\zeta^\lambda$ if $\mu$ is a composition of $n$ with the same parts as $\lambda$ but in a different order.) %(We obtain the same character and an isomorphic module if $\lambda$ is a composition of $n$ with the same parts but permuted, rather than necessarily a partition.)
	\end{enumerate}
\end{notation}

Recall that the irreducible decomposition of such a permutation character $\zeta^\lambda$ is described by Young's Rule \cite[2.8.5]{JK}. Equivalently,
\begin{equation}\label{eqn:kostka}
	\zeta^\lambda = \sum_{\gamma\vdash n} K_{\gamma,\lambda}\cdot \chi^\gamma
\end{equation}
where the Kostka number $K_{\gamma,\lambda} = \langle \zeta^\lambda,\chi^\gamma\rangle=c^\gamma_{(\lambda_1),\dotsc,(\lambda_{l(\lambda)})}$ equals the number of semistandard Young tableaux of shape $\gamma$ and content $\lambda$.
In particular, we therefore have
\begin{equation}\label{eqn:15.2}
	\sum_{\gamma\vdash n} K_{\gamma,\lambda}\cdot \rho^\gamma_m = \cX(\triv_{S_m}; \zeta^\lambda)\up_{S_m\wr S_n}^{S_{mn}} = \rho^{(\lambda_1)}_m\boxtimes\cdots\boxtimes\rho^{(\lambda_{l(\lambda)})}_m.
\end{equation}

We now precisely identify the $i=k$ term of Theorem~\ref{thm:14.6i}, in Theorem~\ref{thm:14.6ii}. We then deduce Theorem~\ref{thm:14.6i} from Theorem~\ref{thm:14.6ii}. Following that, we prove Theorem~\ref{thm:14.6ii}, during the course of which we will also prove Theorem B, which has been numbered as Theorem~\ref{thm:16.7} in this section for ease of reference.

\begin{theorem}\label{thm:14.6ii}
	Fix $n\in\N$. Let $m\in\N$, $k\in\{0,1,\dotsc,n-1\}$ and $\lambda\vdash n$. Let $\nu\vdash mn+k$ with $l(\nu)=n$, and set $\hat{\nu}:=\nu-(1^n)\vdash (m-1)n+k$. Then
	\[ a^{\nu/(1^k)}_{\lambda',(m)} = \sum_{\substack{\alpha\vdash mk\\\beta\vdash k}} a^\alpha_{\beta',(m)}\cdot a^{\hat{\nu}/\alpha}_{\lambda/\beta,(m-1)}. \]
\end{theorem}

\begin{proof}[Proof of Theorem~\ref{thm:14.6i} from Theorem~\ref{thm:14.6ii}]
	We proceed by induction on $k$, observing that the $k=0$ case of Theorem~\ref{thm:14.6i} follows immediately from the $k=0$ case of Theorem~\ref{thm:14.6ii}. %with the base case $k=0$ given by Theorem~\ref{thm:4.4.3}. 
	Now assume $k>0$, and fix $\mu\vdash mn$ with $l(\mu)=n-k$. Let $\nu=\mu\sqcup(1^k)$, so $l(\nu)=n$ and $\hat{\nu}=\hat{\mu}$.
	
	We aim to evaluate $a^{\nu/(1^k)}_{\lambda',(m)}$, and compare it to Theorem~\ref{thm:14.6ii}. To do this, we will study the constituents in the skew character $\chi^{\nu/(1^k)}$. First, note that for any $\omega\vdash|\nu|-k$, by Theorem~\ref{thm:LR} we must have $c^\nu_{\omega,(1^k)}\in\{0,1\}$. Moreover, $c^\nu_{\omega,(1^k)}=1$ if and only if $\omega\subseteq\nu$ and all $k$ boxes of $[\nu/\omega]$ belong to different rows. We will denote by $\mathcal{A}$ the collection of $\omega$ with $c^\nu_{\omega,(1^k)}=1$. In particular,
	\[ \sum_{\omega\in\mathcal{A}}\chi^\omega = \chi^{{\nu}/(1^k)}. \]
	We partition $\mathcal{A}$ as a disjoint union, $\mathcal{A} = \coprod_{j=0}^{k}\mathcal{A}_j$, where $\mathcal{A}_j$ is the collection of $\omega\in\mathcal{A}$ for which $l(\omega)=n-k+j$. Notice that for each $j$, $\mathcal{A}_j$ bijects to $\mathcal{B}_j := \{\varpi\vdash|\hat{\mu}|-j \mid c^{\hat{\mu}}_{\varpi,(1^j)}\}$; the map is given by removal of the first column, $\omega\mapsto\hat{\omega}$, and this is seen to be a bijection by Theorem~\ref{thm:LR}. By a similar application of Theorem~\ref{thm:LR},
	
	\begin{equation}\label{eqn:B_j}
	\sum_{\varpi\in\mathcal{B}_j}\chi^\varpi = \chi^{\hat{\mu}/(1^j)}.
	\end{equation}
	Now observe that
	\[ a^{\nu/(1^k)}_{\lambda',(m)} = \sum_{\omega\in\mathcal{A}} a^{\omega}_{\lambda',(m)} = a^{\mu}_{\lambda',(m)} + \sum_{j=1}^k \sum_{\omega\in\mathcal{A}_j} a^{\omega}_{\lambda',(m)}. \]
	The idea here is that in our partition of $\mathcal{A}$, the $j=0$ term contributes precisely $a^{\mu}_{\lambda',(m)}$. Let us set $X := -\sum_{j=1}^k \sum_{\omega\in\mathcal{A}_j} a^{\omega}_{\lambda',(m)}$. Assuming Theorem~\ref{thm:14.6ii}, it suffices to show that $X$ equals the $\sum_{i=0}^{k-1}(\dotsc)$ part of the summation on the right hand side of \eqref{eqn:A}. For $0<j\leq k$, we have that
	\begin{align*}
	\sum_{\omega\in\mathcal{A}_j} a^{\omega}_{\lambda',(m)} &= \sum_{\omega\in\mathcal{A}_j} \sum_{i=0}^{k-j}(-1)^{k-j+i}\sum_{\substack{\gamma\vdash k-j + (m-1)i\\ \beta\vdash i}}a^{\gamma/(k-j-i)}_{\beta',(m)}\cdot a^{\hat{\omega}/\gamma}_{\lambda/\beta,(m-1)}
	\quad\text{by inductive hypothesis}\\
	&=\sum_{\varpi\in\mathcal{B}_j} \sum_{i=0}^{k-j}(-1)^{k-j+i}\sum_{\substack{\gamma\vdash k-j + (m-1)i\\ \beta\vdash i}}a^{\gamma/(k-j-i)}_{\beta',(m)}\cdot a^{\varpi/\gamma}_{\lambda/\beta,(m-1)}
	\quad\text{by bijection }\mathcal{A}_j\to\mathcal{B}_j\\
	&=\sum_{\varpi\in\mathcal{B}_j} \sum_{i=0}^{k-j}(-1)^{k-j+i}\sum_{\substack{\gamma\vdash k-j + (m-1)i\\ \beta\vdash i}}a^{\gamma/(k-j-i)}_{\beta',(m)}\cdot \left\langle\chi^{\varpi},\rho^{\lambda/\beta}_{m-1}\boxtimes\chi^\gamma\right\rangle \quad\text{by \eqref{eqn:a}}\\
	&=\sum_{i=0}^{k-j}(-1)^{k-j+i}\sum_{\substack{\gamma\vdash k-j + (m-1)i\\ \beta\vdash i}}a^{\gamma/(k-j-i)}_{\beta',(m)}\cdot \left\langle\chi^{\hat{\mu}/(1^j)},\rho^{\lambda/\beta}_{m-1}\boxtimes\chi^\gamma\right\rangle
	\quad\text{by \eqref{eqn:B_j}}\\
	&=\sum_{i=0}^{k-j}(-1)^{k-j+i}\sum_{\substack{\gamma\vdash k-j + (m-1)i\\ \beta\vdash i}} a^{\gamma/(k-j-i)}_{\beta',(m)}\cdot \left\langle\chi^{\hat{\mu}},\rho^{\lambda/\beta}_{m-1}\boxtimes\chi^\gamma\boxtimes\chi^{(1^j)}\right\rangle\\
	&=\sum_{i=0}^{k-j}(-1)^{k-j+i}\sum_{\substack{\gamma\vdash k-j + (m-1)i\\ \alpha\vdash k+(m-1)i\\ \beta\vdash i}} c^{\alpha}_{\gamma,(1^j)}\cdot a^{\gamma/(k-j-i)}_{\beta',(m)}\cdot \left\langle\chi^{\hat{\mu}},\rho^{\lambda/\beta}_{m-1}\boxtimes\chi^\alpha\right\rangle\\
	&=\sum_{i=0}^{k-j}(-1)^{k-j+i}\sum_{\substack{\gamma\vdash k-j + (m-1)i\\ \alpha\vdash k+(m-1)i\\ \beta\vdash i}} c^{\alpha}_{\gamma,(1^j)}\cdot a^{\gamma/(k-j-i)}_{\beta',(m)}\cdot a^{\hat{\mu}/\alpha}_{\lambda/\beta,(m-1)} \quad\text{by \eqref{eqn:a}}.
	\end{align*}
	It follows that
	\begin{align*}
	X &= -\sum_{j=1}^k \sum_{i=0}^{k-j}(-1)^{k-j+i}\sum_{\substack{\gamma\vdash k-j + (m-1)i\\ \alpha\vdash k+(m-1)i\\ \beta\vdash i}} c^{\alpha}_{\gamma,(1^j)}\cdot a^{\gamma/(k-j-i)}_{\beta',(m)}\cdot a^{\hat{\mu}/\alpha}_{\lambda/\beta,(m-1)}\\
	&= \sum_{i=0}^{k-1} \sum_{j=1}^{k-i} (-1)^{k+i} \cdot (-1)^{j+1} \sum_{\substack{\gamma\vdash k-j + (m-1)i\\ \alpha\vdash k+(m-1)i\\ \beta\vdash i}} c^{\alpha}_{\gamma,(1^j)}\cdot a^{\gamma/(k-j-i)}_{\beta',(m)}\cdot a^{\hat{\mu}/\alpha}_{\lambda/\beta,(m-1)}\\
	&= \sum_{i=0}^{k-1} (-1)^{k+i} \sum_{j=1}^{k-i} \sum_{\beta\vdash i} (-1)^{j+1}\cdot Y^\beta_{i,j},
	\end{align*}
	where we set
	\begin{align*}
	Y^\beta_{i,j}:=\sum_{\substack{\gamma\vdash k-j + (m-1)i\\ \alpha\vdash k+(m-1)i}} c^{\alpha}_{\gamma,(1^j)}\cdot a^{\gamma/(k-j-i)}_{\beta',(m)}\cdot a^{\hat{\mu}/\alpha}_{\lambda/\beta,(m-1)}.
	\end{align*}
	Next, we will simplify $Y^\beta_{i,j}$, for any fixed $\beta$, $i$ and $j$. To ease notation, in the rest of this proof we will abbreviate sums over all partitions of a given size. That is, we shorten $\sum_{\omega\vdash t}$ to $\sum_\omega$ (the size $t$ will always be clear from context). We use \eqref{eqn:skew-plethysm} to obtain
	\begin{align*}
	Y^\beta_{i,j}=\sum_{\gamma} \sum_\alpha c^\alpha_{\gamma,(1^j)}\cdot a^{\hat{\mu}/\alpha}_{\lambda/\beta,(m-1)} \cdot \sum_\varepsilon c^\gamma_{\varepsilon,(k-j-i)}\cdot a^\varepsilon_{\beta',(m)}.
	\end{align*}
	By \cite[(2.1)]{SLthesis}, we have that $\sum_\gamma c^\alpha_{\gamma,(1^j)}\cdot c^{\gamma}_{\varepsilon,(k-j-i)} = c^{\alpha}_{\varepsilon,(k-j-i),(1^j)} = \langle \chi^{\alpha/\varepsilon}, \chi^{k-j-i}\boxtimes \chi^{(1^j)}\rangle$. By Theorem~\ref{thm:LR}, $\chi^{k-j-i}\boxtimes \chi^{(1^j)} = \chi^{H(j)}+\chi^{H(j-1)}$ where $H(j):=(k-i-j,1^j)$, except if $j=k-i$ then $\chi^\emptyset\boxtimes \chi^{(1^{k-i})}=\chi^{H(k-i-1)}$ only, i.e.~we treat the $\chi^{H(k-i)}$ term as the zero character. Hence
	\[ Y^\beta_{i,j} = \sum_\alpha \sum_\varepsilon \langle \chi^{\alpha/\varepsilon}, \chi^{H(j)}+\chi^{H(j-1)}\rangle \cdot a^\varepsilon_{\beta',(m)}\cdot a^{\hat{\mu}/\alpha}_{\lambda/\beta,(m-1)} = \sum_\alpha \sum_\varepsilon (c^\alpha_{\varepsilon,H(j)} + c^\alpha_{\varepsilon,H(j-1)}) \cdot a^\varepsilon_{\beta',(m)}\cdot a^{\hat{\mu}/\alpha}_{\lambda/\beta,(m-1)} \]
	(where we omit the $c^\alpha_{\varepsilon,H(k-i)}$ term). Since $H(0)=(k-i)$, we obtain
	\[ \sum_{j=1}^{k-i}(-1)^{j+1} \cdot Y^\beta_{i,j} = \sum_\alpha \sum_\varepsilon c^\alpha_{\varepsilon,(k-i)}\cdot a^\varepsilon_{\beta',(m)}\cdot a^{\hat{\mu}/\alpha}_{\lambda/\beta,(m-1)} = \sum_\alpha a^{\alpha/(k-i)}_{\beta',(m)}\cdot a^{\hat{\mu}/\alpha}_{\lambda/\beta,(m-1)}.\]
	Thus, we finally obtain
	\[ X = \sum_{i=0}^{k-1}(-1)^{k+i} \cdot \sum_{\beta\vdash i} \sum_{j=1}^{k-i}(-1)^{j+1}\cdot Y^\beta_{i,j} = \sum_{i=0}^{k-1}(-1)^{k+i}\cdot \sum_{\beta\vdash i}\ \sum_{\alpha\vdash k+(m-1)i} a^{\alpha/(k-i)}_{\beta',(m)}\cdot a^{\hat{\mu}/\alpha}_{\lambda/\beta,(m-1)}, \]
	as desired.
\end{proof}

To prove Theorem~\ref{thm:14.6ii}, we first deal with the case of $m=1$. In this case, $a^{\nu/(1^k)}_{\lambda',(1)} = c^{\nu}_{\lambda',(1^k)}$, which takes value 1 if $\nu-\lambda'$ is a sequence of 0s and 1s containing exactly $k$ many 1s, and takes value 0 otherwise. On the other hand, $\sum_{\alpha,\beta\vdash k} \big(a^\alpha_{\beta',(1)}\cdot a^{\hat{\nu}/\alpha}_{\lambda/\beta,\emptyset}\big)=c^{\lambda'}_{\hat{\nu},(1^{n-k})}$, which takes value 1 if $\lambda'-\hat{\nu}$ is a sequence of 0s and 1s containing exactly $n-k$ many 1s, and takes value 0 otherwise. We see that these two quantities are equal since $\nu=\hat{\nu}+(1^n)$, and hence Theorem~\ref{thm:14.6ii} holds when $m=1$. %, and so for the remainder of this section we may assume $m\ge 2$.

Next, we introduce some lemmas in preparation for proving Theorem~\ref{thm:14.6ii} when $m\ge 2$.

\begin{lemma}\label{lem:16.2}
	Let $r\in\N$ and $n_1,\dotsc,n_r\in\N_0$. Let $\gamma$ be a partition and for each $i\in\{1,2,\dotsc,r\}$, let $\psi_i$ be a virtual character of $S_{n_i}$. Then
	\[ (\psi_1\boxtimes\cdots\boxtimes\psi_r)/\chi^\gamma = \sum_{\gamma_1,\dotsc,\gamma_r} c^\gamma_{\gamma_1,\dotsc,\gamma_r}\cdot (\psi_1/\chi^{\gamma_1})\boxtimes\cdots\boxtimes(\psi_r/\chi^{\gamma_r}), \]
	summed over all sequences of partitions $\gamma_1,\dotsc,\gamma_r$ such that $|\gamma_1|+\cdots+|\gamma_r|=|\gamma|$.
\end{lemma}

\begin{proof}
	The case $r=1$ is trivial. For ease of notation we prove the statement for $r=2$; the case of general $r$ follows by an analogous argument. Define $n:=n_1+n_2$, $k:=|\gamma|$ and let $\delta\vdash n-k$. Then
	\begin{align*}
		\langle (\psi_1\boxtimes\psi_2)/\chi^\gamma,\chi^\delta\rangle &= \langle (\psi_1\times\psi_2)\up_{S_{n_1}\times S_{n_2}}^{S_n}, (\chi^\delta\times\chi^\gamma)\up_{S_{n-k}\times S_k}^{S_n}\rangle\\
		&= \langle (\psi_1\times\psi_2)\up_{S_{n_1}\times S_{n_2}}^{S_n}\down_{S_{n-k}\times S_k}, \chi^\delta\times\chi^\gamma\rangle,
	\end{align*}
	and applying Mackey's Theorem,
	\begin{align*}
		&= \sum_{\substack{0\le t_1,t_2\le k\\ t_1+t_2=k}} \langle (\psi_1\times\psi_2)\down^{S_{n_1}\times S_{n_2}}_{S_{n_1-t_1}\times S_{t_1}\times S_{n_2-t_2}\times S_{t_2}}\up^{S_{n-k}\times S_{k}}, \chi^\delta\times\chi^\gamma\rangle\\
		&= \sum_{t_1+t_2=k} \sum_{\substack{i=1,2\\\gamma_i\vdash t_i\\\delta_i\vdash n_i-t_i}} c^\gamma_{\gamma_1,\gamma_2}\cdot c^\delta_{\delta_1,\delta_2}\cdot \langle \psi_1\down_{S_{n_1-t_1}\times S_{t_1}} \times \psi_2\down_{S_{n_2-t_2}\times S_{t_2}}, \chi^{\delta_1}\times\chi^{\gamma_1}\times\chi^{\delta_2}\times\chi^{\gamma_2}\rangle\\
		%&=\sum_{t_1+t_2=k} \sum_{\substack{i\\\gamma_i\vdash t_i\\\delta_i\vdash n_i-t_i}} c^\gamma_{\gamma_1,\gamma_2}\cdot c^\delta_{\delta_1,\delta_2}\cdot \langle \psi_1\down_{S_{n_1-t_1}\times S_{t_1}}, \chi^{\delta_1}\times\chi^{\gamma_1}\rangle\cdot\langle \psi_2\down_{S_{n_2-t_2}\times S_{t_2}},\chi^{\delta_2}\times\chi^{\gamma_2}\rangle\\
		&= \sum_{t_1+t_2=k} \sum_{\substack{i\\\gamma_i\vdash t_i\\\delta_i\vdash n_i-t_i}} c^\gamma_{\gamma_1,\gamma_2}\cdot c^\delta_{\delta_1,\delta_2}\cdot \langle \psi_1/\chi^{\gamma_1},\chi^{\delta_1}\rangle\cdot\langle \psi_2/\chi^{\gamma_2},\chi^{\delta_2}\rangle\\
		%&= \sum_{t_1+t_2=k} \sum_{\substack{i\\\gamma_i\vdash t_i}} c^\gamma_{\gamma_1,\gamma_2}\cdot \langle (\psi_1/\chi^{\gamma_1})\times(\psi_2/\chi^{\gamma_2}), \chi^\delta\down_{S_{n_1-t_1}\times S_{n_2-t_2}}\rangle\\
		&= \sum_{t_1+t_2=k} \sum_{\substack{i\\\gamma_i\vdash t_i}} c^\gamma_{\gamma_1,\gamma_2}\cdot \langle (\psi_1/\chi^{\gamma_1})\boxtimes(\psi_2/\chi^{\gamma_2}), \chi^\delta\rangle,
	\end{align*}
	as claimed.
\end{proof}

\begin{corollary}\label{cor:16.2}
	Let $m,n\in\N$ and $\lambda=(\lambda_1,\dotsc,\lambda_r)\vdash n$. Let $k\in\{0,1,\dotsc,n-1\}$ and $\beta\vdash k$. Then
	\begin{enumerate}[label=\textup{(\roman*)}]
		\item $\big(\operatorname*{\boxtimes}_{i=1}^r \rho^{(\lambda_i)}_m \big)/\chi^{(1^{n-k})} = \sum_{\underline{t}}  \operatorname*{\boxtimes}_{i=1}^r \big( \rho^{(\lambda_i)}_m/\chi^{(1^{\lambda_i-t_i})} \big)$, and
		
		\item $\big( \operatorname*{\boxtimes}_{i=1}^r \chi^{(1^{\lambda_i})} \big)/\chi^{\beta'} = \sum_{\underline{t}} c^\beta_{(t_1),\dotsc,(t_r)}\cdot \operatorname*{\boxtimes}_{i=1}^r \chi^{(1^{\lambda_i-t_i})}$,
	\end{enumerate}
	summed over all compositions $\underline{t}=(t_1,\dotsc,t_r)$ of $k$ into $r$ parts. That is, $t_i\in\N_0$ for all $i$ and $t_1+\cdots+t_r=k$ (and we may further assume $t_i\le\lambda_i$ for all $i$).
\end{corollary}

\begin{proof}
	\noindent\textbf{(i)} Applying Lemma~\ref{lem:16.2} with $\gamma=(1^{n-k})$, observe that $c^\gamma_{\gamma_1,\dotsc,\gamma_r}\in\{0,1\}$ and is non-zero only if each $\gamma_i=(1^{s_i})$ for some $s_i\in\N_0$. By Lemma~\ref{lem:tall-pleth}, $\rho^{(\lambda_i)}_m$ only has irreducible constituents $\chi^\mu$ where $l(\mu)\le \lambda_i$, so we may further assume that $s_i\le\lambda_i$. Writing $t_i=\lambda_i-s_i$ gives the result.
	
	\noindent\textbf{(ii)} Applying Lemma~\ref{lem:16.2} with $\psi_i=\chi^{(1^{\lambda_i})}$, observe that $\psi_i/\chi^{\gamma_i}\ne 0$ only if $\gamma_i=(1^{t_i})$ for some $0\le t_i\le \lambda_i$. Moreover, $c^{\beta'}_{(1^{t_1}),\dotsc,(1^{t_r})}=c^\beta_{(t_1),\dotsc,(t_r)}\in\{0,1\}$.
\end{proof}

\begin{lemma}\label{lem:16.3}
	Let $n\in\N$ and $k\in\{0,1,\dotsc,n\}$. Let $\nu$ be a partition with $l(\nu)=n$ and set $\hat{\nu}:=\nu-(1^n)$. Let $\delta\vdash|\nu|-k$ with $l(\delta)\le n$. Then
	\[ \langle\chi^\delta,\chi^{\nu/(1^k)} \rangle = \langle \chi^{\delta/(1^{n-k})},\chi^{\hat{\nu}}\rangle. \]
\end{lemma}

\begin{proof}
	The determinantal form of skew characters of symmetric groups (see e.g.~\cite[2.3.13]{JK}) gives $\chi^{\alpha/\beta} = \det\big(\chi^{(\alpha_i-i-\beta_j+j)}\big)$ whenever $\alpha$ and $\beta$ are partitions, where the multiplication of characters in expanding the determinant is given by the operation $\boxtimes$. Applying this to $\alpha=\nu'$ and $\beta=\emptyset$ and expanding the determinant with respect to the first row gives
	\[ \chi^{\nu'} = \sum_{j\ge 1} (-1)^{j-1}\cdot \chi^{(n+j-1)}\boxtimes \chi^{\hat{\nu}'/(1^{j-1})} = \sum_{j\ge 0} (-1)^j\cdot\chi^{(n+j)}\boxtimes \chi^{\hat{\nu}'/(1^j)}. \]
	Multiplying both sides by the sign representation then gives
	\[ \chi^\nu = \sum_{j\ge 0} (-1)^j\cdot \chi^{(1^{n+j})}\boxtimes \chi^{\hat{\nu}/(j)}. \]
	Then
	\begin{align*}
		\langle\chi^\delta,\chi^{\nu/(1^k)}\rangle &= \langle \chi^\delta\boxtimes\chi^{(1^k)},\chi^\nu\rangle \\
		&= \sum_{j\ge 0}(-1)^j\cdot\langle(\chi^\delta\boxtimes\chi^{(1^k)})/\chi^{(1^{n+j})}, \chi^{\hat{\nu}/(j)} \rangle\\
		&= \sum_{j=0}^k (-1)^j\cdot \left\langle \sum_{s=0}^{k-j}(\chi^\delta/\chi^{(1^{n-s})})\boxtimes(\chi^{(1^k)}/\chi^{(1^{s+j})}), \chi^{\hat{\nu}/(j)}\right\rangle\quad\text{by Lemma~\ref{lem:16.2}}\\
		&=\sum_{j=0}^k \sum_{s=0}^{k-j}(-1)^j \langle (\chi^\delta/\chi^{(1^{n-s})})\boxtimes \chi^{(1^{k-j-s})}, \chi^{\hat{\nu}/(j)}\rangle\\
		&= \sum_{s=0}^k \left\langle (\chi^\delta/\chi^{(1^{n-s})})\boxtimes \left(\sum_{j=0}^{k-s}(-1)^j\cdot\chi^{(1^{k-s-j})}\boxtimes\chi^{(j)}\right), \chi^{\hat{\nu}}\right\rangle\\
		&=\langle \chi^\delta/\chi^{(1^{n-k})}, \chi^{\hat{\nu}}\rangle,
	\end{align*}
	where the final equality follows since $\chi^{(1^{k-s-j})}\boxtimes\chi^{(j)}=\chi^{(j+1,1^{k-s-j-1})}+\chi^{(j,1^{k-s-j})}$%(omitting $\chi^{(j+1,1^{k-s-j-1})}$ if $j=k-s$ and $\chi^{(j,1^{k-s-j})}$ if $j=0$)
	, and so $\sum_{j=0}^{k-s}(-1)^j\cdot\chi^{(1^{k-s-j})}\boxtimes\chi^{(j)}$ equals zero if $k\ne s$, and equals $\chi^\emptyset$ if $k=s$.
\end{proof}

\begin{lemma}\label{lem:16.6}
	Let $m$, $u$ and $t$ be integers with $m\ge 2$ and $u\ge t\ge 0$. Then
	\[ \rho^{(u)}_m/\chi^{(1^{u-t})} = \rho^{(t)}_m\boxtimes \rho^{(1^{u-t})}_{m-1}. \]
\end{lemma}

\begin{proof}
	Let $\delta\vdash mu-(u-t)$ be arbitrary. We show that $\langle \rho^{(u)}_m/\chi^{(1^{u-t})}, \chi^\delta\rangle = \langle\rho^{(t)}_m\boxtimes \rho^{(1^{u-t})}_{m-1}, \chi^\delta\rangle$. 
	Letting $H:=S_m\wr S_u$ and $K:=S_{|\delta|}\times S_{u-t}$, and substituting in the definition $\rho^{(u)}_m$ from Notation~\ref{not:rho}, we have by Mackey's theorem that
	\begin{align}\label{eqn:16.6}
		\langle \rho^{(u)}_m/\chi^{(1^{u-t})}, \chi^\delta\rangle &= \langle \rho^{(u)}_m, \chi^\delta\boxtimes\chi^{(1^{u-t})}\rangle = \langle \triv_H\up^{S_{mu}}, (\chi^\delta\times\chi^{(1^{u-t})})\up_K^{S_{mu}}\rangle \nonumber\\
		&= \sum_{\sigma\in K\setminus S_{mu}/H}\langle\triv_{H^\sigma}\down_{K\cap H^\sigma}\up^K,\chi^\delta\times\chi^{(1^{u-t})}\rangle\nonumber\\
        &=\sum_{\sigma\in K\setminus S_{mu}/H}\langle\triv_{K\cap H^\sigma},(\chi^\delta\times\chi^{(1^{u-t})})\down^K_{K\cap H^\sigma}\rangle,
	\end{align}
	where the final equality follows from Frobenius reciprocity.
	Here $\sigma$ runs over a set of representatives of double $(K,H)$-cosets in $S_{mu}$. Since $K\cap H^\sigma$ are the point stabilisers of the action of $K$ on the set of partitions of $\{1,2,\dotsc,mu\}$ into $u$ subsets of size $m$, the representatives $\sigma$ are parametrised by partitions of $u-t$ into exactly $u$ parts, including parts of size zero. 
	Fix one such partition $\sigma$ of $u-t$ and suppose that $\gamma_i$ is the number of parts of size $i$, for each $i\in\N_0$. Then $K\cap H^\sigma \cong \prod_{i\in\N_0} (S_{m-i}\times S_i)\wr S_{\gamma_i}$, and
	\[ \langle (\chi^\delta\times\chi^{(1^{u-t})})\down^K_{K\cap H^\sigma}, \triv\rangle \le \langle (\chi^\delta\times\chi^{(1^{u-t})})\down_{\prod_i (S_{m-i}\times S_i)^{\times \gamma_i}}, \triv\rangle = \langle \chi^\delta\down_{\prod_i S_{m-i}^{\times\gamma_i}},\triv\rangle\cdot \langle \chi^{(1^{u-t})}\down_{\prod_i S_i^{\times\gamma_i}},\triv\rangle. \]
	However, $\chi^{(1^{u-t})}$ is the sign representation, so $\langle \chi^{(1^{u-t})}\down_{\prod_i S_i^{\times\gamma_i}},\triv\rangle\ne 0$ if and only if $\gamma_i=0$ for all $i\ge 2$. Hence there is at most one $\sigma$ giving a non-zero contribution to the sum in \eqref{eqn:16.6}, namely $\sigma=(1^{u-t},0^t)$, and in this case $K\cap H^\sigma\cong (S_m\wr S_t)\times \big((S_{m-1}\times S_1)\wr S_{u-t}\big)$. Substituting into \eqref{eqn:16.6},
	\begin{align*}
		\langle \rho^{(u)}_m/\chi^{(1^{u-t})}, \chi^\delta\rangle &= \langle (\chi^\delta\times\chi^{(1^{u-t})})\down^K_{S_m\wr S_t\times (S_{m-1}\times S_1)\wr S_{u-t}},\ \triv\rangle\\
		&= \langle (\chi^\delta\times\chi^{(1^{u-t})})\down^{S_{|\delta|}\times S_{u-t}}_{S_m\wr S_t\times S_{m-1}\wr S_{u-t}\times S_1\wr S_{u-t}},\ \triv\up_{S_m\wr S_t\times (S_{m-1}\times S_1)\wr S_{u-t}}^{{S_m\wr S_t\times S_{m-1}\wr S_{u-t}\times S_1\wr S_{u-t}}}\rangle
	\end{align*}
	where the second equality follows from Frobenius reciprocity. Noting that $|\delta|=mu-(u-t)=mt+(m-1)(u-t)$ and $\cX(\triv_{S_1};\chi^\omega)=\chi^\omega$, and using Lemma~\ref{lem:X-induced} in the second equality below, we have
	\begin{align*}
		\langle \rho^{(u)}_m/\chi^{(1^{u-t})}, \chi^\delta\rangle &= \langle \chi^\delta\down^{S_{|\delta|}}_{S_m\wr S_t\times S_{m-1}\wr S_{u-t}} \times \chi^{(1^{u-t})},\ \triv_{S_m\wr S_t} \times \triv\up_{(S_{m-1}\times S_1)\wr S_{u-t}}^{S_{m-1}\wr S_{u-t}\times S_1\wr S_{u-t}}\rangle\\
		&= \left\langle \chi^\delta\down^{S_{|\delta|}}_{S_m\wr S_t\times S_{m-1}\wr S_{u-t}} \times \chi^{(1^{u-t})},\ \triv_{S_m\wr S_t} \times \sum_{\omega\vdash u-t}\cX(\triv_{S_{m-1}}; \chi^\omega)\cdot\cX(\triv_{S_1};\chi^\omega) \right\rangle\\
		&= \sum_{\omega\vdash u-t} \langle \chi^\delta\down_{S_m\wr S_t\times S_{m-1}\wr S_{u-t}},\ \triv_{S_m\wr S_t}\times \cX(\triv_{S_{m-1}};\chi^\omega)\rangle\cdot\langle \chi^{(1^{u-t})},\chi^\omega\rangle.
	\end{align*}
	Now $\langle \chi^{(1^{u-t})},\chi^\omega\rangle=1$ precisely when $\omega=(1^{u-t})$ and is 0 otherwise, so
	\[ \langle \rho^{(u)}_m/\chi^{(1^{u-t})}, \chi^\delta\rangle = \langle \chi^\delta\down_{S_m\wr S_t\times S_{m-1}\wr S_{u-t}},\ \triv_{S_m\wr S_t}\times \cX(\triv_{S_{m-1}};\chi^{(1^{u-t})})\rangle = \langle \chi^\delta,\ \rho^{(t)}_m\boxtimes \rho^{(1^{u-t})}_m\rangle \]
	by Frobenius reciprocity, recalling Notation~\ref{not:rho}(ii) and (iii).
	Since $\delta$ was arbitrary, then $\rho^{(u)}_m/\chi^{(1^{u-t})} = \rho^{(t)}_m\boxtimes \rho^{(1^{u-t})}_{m-1}$ as desired.
\end{proof}

\begin{remark}
	When $m=2$, we can see from Proposition~\ref{prop:thrall} that $\rho^{(u)}_2/\chi^{(1^{u-t})} = \sum_\delta \chi^\delta = \rho^{(t)}_2\boxtimes\chi^{(1^{u-t})}$ where the sum is over all $\delta\vdash u+t$ with exactly $u-t$ many odd parts.\hfill$\lozenge$
\end{remark}

Next, we generalise Lemma~\ref{lem:16.6} from the trivial partition $(u)$ to arbitrary partitions, giving Theorem B, after which it will be straightforward to deduce Theorem~\ref{thm:14.6ii}.

\begin{theorem}[Theorem B]\label{thm:16.7}
	Let $m,n\in\N$ with $m\ge 2$. Let $\lambda\vdash n$ and $k\in\{0,1,\dotsc,n-1\}$. Then
	\[ \rho^\lambda_m/\chi^{(1^{n-k})} = \sum_{\beta\vdash k} \rho^\beta_m\boxtimes \rho^{\lambda'/\beta'}_{m-1}. \]
\end{theorem}

\begin{proof}
	Following the notation for $\underline{t}$ in Corollary~\ref{cor:16.2} and letting $r=l(\lambda)$, observe that
	\begin{align*}
		\sum_{\gamma\vdash n} K_{\gamma,\lambda}\cdot \big( \rho^\gamma_m/\chi^{(1^{n-k})} \big) &= \Big( \sum_{\gamma\vdash n} K_{\gamma,\lambda}\cdot \rho^\gamma_m \Big)/\chi^{(1^{n-k})} = \Big( \operatorname*{\boxtimes}_{i=1}^r \rho^{(\lambda_i)}_m\Big)/\chi^{(1^{n-k})} \quad\text{by \eqref{eqn:15.2}}\\
		&= \sum_{\underline{t}} \operatorname*{\boxtimes}_{i=1}^r \big( \rho^{(\lambda_i)}_m/\chi^{(1^{\lambda_i-t_i})} \big) \quad\text{by Corollary~\ref{cor:16.2}(i)}\\
		&= \sum_{\underline{t}} \operatorname*{\boxtimes}_{i=1}^r \big( \rho^{(t_i)}_m \boxtimes \rho^{(1^{\lambda_i-t_i})}_{m-1} \big) \quad\text{by Lemma~\ref{lem:16.6}}\\
		&= \sum_{\underline{t}} \sum_{\beta\vdash k} c^\beta_{(t_1),\dotsc,(t_r)}\cdot\rho^\beta_m \boxtimes \big(\operatorname*{\boxtimes}_{i=1}^r \rho^{(1^{\lambda_i-t_i})}_{m-1}\big) \quad\text{by \eqref{eqn:15.2}}\\
		&= \sum_{\beta\vdash k} \rho^\beta_m \boxtimes \cX\Big(\triv_{S_{m-1}}; \sum_{\underline{t}} c^\beta_{(t_1),\dotsc,(t_r)}\cdot \operatorname*{\boxtimes}_{i=1}^r \chi^{(1^{\lambda_i-t_i})} \Big)\up_{S_{m-1}\wr S_{n-k}}^{S_{(m-1)(n-k)}}\quad\text{by Lemma~\ref{lem:cX}}\\
		&= \sum_{\beta\vdash k} \rho^\beta_m \boxtimes \cX \Big( \triv_{S_{m-1}}; \big( \operatorname*{\boxtimes}_{i=1}^r \chi^{(1^{\lambda_i})}\big) /\chi^{\beta'} \Big) \up_{S_{m-1}\wr S_{n-k}}^{S_{(m-1)(n-k)}} \quad\text{by Corollary~\ref{cor:16.2}(ii)}\\
		&= \sum_{\beta\vdash k} \rho^\beta_m \boxtimes \cX \Big( \triv_{S_{m-1}}; \big( \zeta^\lambda\cdot\sgn_{S_n} \big) /\chi^{\beta'} \Big) \up_{S_{m-1}\wr S_{n-k}}^{S_{(m-1)(n-k)}} \quad\text{by \eqref{eqn:sign}}\\
		&= \sum_{\beta\vdash k} \rho^\beta_m \boxtimes \cX \Big( \triv_{S_{m-1}}; \big( \sum_{\gamma\vdash n} K_{\gamma,\lambda}\cdot \chi^{\gamma'} \big) /\chi^{\beta'} \Big) \up_{S_{m-1}\wr S_{n-k}}^{S_{(m-1)(n-k)}} \quad\text{by \eqref{eqn:sign} and \eqref{eqn:kostka}}\\
		&= \sum_{\gamma\vdash n} K_{\gamma,\lambda} \cdot \Big( \sum_{\beta\vdash k} \rho^\beta_m \boxtimes \rho^{\gamma'/\beta'}_{m-1} \Big)\quad\text{since $\chi^{\gamma'}/\chi^{\beta'}=\chi^{\gamma'/\beta'}$ by Notation~\ref{not:rho}(i)}.\\
	\end{align*}
	Since the matrix $(K_{\gamma,\lambda})_{\gamma,\lambda\vdash n}$ is invertible (in fact unitriangular if the partitions are ordered lexicographically, see e.g.~\cite[Chapter 2]{JK}), we deduce that $\rho^\gamma_m/\chi^{(1^{n-k})} = \sum_{\beta\vdash k} \rho^\beta_m \boxtimes \rho^{\gamma'/\beta'}_{m-1}$ for each $\gamma\vdash n$.
\end{proof}

\begin{lemma}\label{lem:cX}
	Let $m,r,a_1,\dotsc,a_r\in\N$, and let $n=\sum_{i=1}^r a_i$. For each $i\in\{1,\dotsc,r\}$, let $\nu_i\vdash a_i$. Then $\boxtimes_{i=1}^r \rho_m^{\nu_i} = \cX(\triv_{S_m};\boxtimes_{i=1}^r \chi^{\nu_i} )\up_{S_m\wr S_n}^{S_{mn}}$.
\end{lemma}

\begin{proof}
	The case $r=1$ follows from Notation~\ref{not:rho}(ii). For each of notation we prove the statement for $r=2$; the case of general $r$ follows by an analogous argument. In fact, we can prove more generally that if $a,b\in\N$ and $\phi_1\in\Ch(S_a), \phi_2\in\Ch(S_b)$, then $\cX_{left}=\cX_{right}$ where
	\[ \cX_{left}:= \left[\cX(\triv_{S_m}; \phi_1)\up_{S_m\wr S_a}^{S_{ma}} \times \cX(\triv_{S_m}; \phi_2)\up_{S_m\wr S_b}^{S_{mb}}\right]\up_{S_{ma}\times S_{mb}}^{S_{m(a+b)}} \]
	and
	\[ \cX_{right}:= \cX\left(\triv_{S_m}; (\phi_1\times \phi_2)\up_{S_a\times S_b}^{S_{a+b}}\right)\up_{S_m\wr S_{a+b}}^{S_{m(a+b)}}, \]
%	\begin{equation}\label{eqn:Xab}
%		\left[\cX(\triv_{S_m}; \phi_1)\up_{S_m\wr S_a}^{S_{ma}} \times \cX(\triv_{S_m}; \phi_2)\up_{S_m\wr S_b}^{S_{mb}}\right]\up_{S_{ma}\times S_{mb}}^{S_{m(a+b)}} = \cX\left(\triv_{S_m}; (\phi_1\times \phi_2)\up_{S_a\times S_b}^{S_{a+b}}\right)\up_{S_m\wr S_{a+b}}^{S_{m(a+b)}},
%	\end{equation}
	from which we recover the case of $r=2$ by setting $\phi_i=\chi^{\nu_i}$. %We denote by $\cX_{left}$ (resp.~$\cX_{right}$) the left (resp.~right) hand term of \eqref{eqn:Xab}.
	
	To prove that $\cX_{left}=\cX_{right}$, we first observe that $S_m\wr(S_a\times S_b) = S_m\wr S_a \times S_m\wr S_b$ (viewing $S_a\times S_b$ as a subgroup of $S_{a+b}$). Calling this group $U$, it is a subgroup of both $T_d:=S_{ma}\times S_{mb}$ and $T_w:=S_m\wr S_{a+b}$, and both $T_d$ and $T_w$ are subgroups of $S:=S_{m(a+b)}$. Now, by Lemma~\ref{lem:infl-ind} and the definition of $\cX(-;-)$,
	\begin{align*}
		\cX(\triv_{S_m}; (\phi_1\times \phi_2)\up_{S_a\times S_b}^{S_{a+b}}) &= \cX(\triv_{S_m}; \phi_1\times\phi_2)\up_U^{T_w} = \left(\Infl_{S_a\times S_b}^{T_w}(\phi_1\times\phi_2)\right)\up_U^{T_w}\\
		&= \left( \Infl_{S_a}^{S_m\wr S_a}(\phi_1)\times \Infl_{S_b}^{S_m\wr S_b}(\phi_2) \right)\up_U^{T_w} = \Big( \cX(\triv_{S_m};\phi_1)\times \cX(\triv_{S_m};\phi_2) \Big)\up_U^{T_w}.
	\end{align*}
	Therefore
	\[ \cX_{right} = \Big( \cX(\triv_{S_m};\phi_1)\times \cX(\triv_{S_m};\phi_2) \Big)\up_U^{T_w} \up_{T_w}^S = \Big( \cX(\triv_{S_m};\phi_1)\times \cX(\triv_{S_m};\phi_2) \Big)\up_U^{T_d} \up_{T_d}^S = \cX_{left}, \]
	where the second equality follows from the transitivity of induction.
\end{proof}
	
\begin{proof}[Proof of Theorem~\ref{thm:14.6ii} when $m\ge 2$]
	Take $\langle -, \chi^{\hat{\nu}}\rangle$ in Theorem~\ref{thm:16.7} to obtain
	\[ \langle \rho^\lambda_m/\chi^{(1^{n-k})}, \chi^{\hat{\nu}} \rangle = \left\langle \sum_{\beta\vdash k} \rho^\beta_m\boxtimes \rho^{\lambda'/\beta'}_{m-1}, \chi^{\hat{\nu}} \right\rangle. \]
	By Lemma~\ref{lem:16.3}, $\langle \rho^\lambda_m/\chi^{(1^{n-k})}, \chi^{\hat{\nu}} \rangle = \langle \rho^\lambda_m, \chi^{\nu/(1^k)}\rangle = a^{\nu/(1^k)}_{\lambda,(m)}$. On the other hand,
	\[ \left\langle \sum_{\beta\vdash k} \rho^\beta_m\boxtimes \rho^{\lambda'/\beta'}_{m-1}, \chi^{\hat{\nu}} \right\rangle = \sum_{\beta\vdash k} \sum_{\alpha\vdash mk} \langle \rho^\beta_m,\chi^\alpha\rangle \cdot \langle \rho^{\lambda'/\beta'}_{m-1}, \chi^{\hat{\nu}/\alpha}\rangle = \sum_{\substack{\alpha\vdash mk\\\beta\vdash k}} a^\alpha_{\beta,(m)}\cdot a^{\hat{\nu}/\alpha}_{\lambda'/\beta',(m-1)}, \]
	which concludes the proof.
\end{proof}

%Thus, the proof of Theorem~\ref{thm:14.6i} (Theorem A) is completed. 
We conclude this section with a conjecture based on computational data in small cases, and which is motivated by Foulkes' Conjecture as described below. %which appears to be true for $a,b\le 7$
\begin{conjecture}\label{conj:18}
	Let $1\le a\le b$ be integers. Then
	\begin{enumerate}[label=\textup{(\roman*)}]
		\item $\rho^{(a)}_{b-1}\boxtimes \chi^{(a-1)} - \rho^{(a-1)}_b\boxtimes \chi^{(b-1)}\in \Ch(S_{ab-1})$, and
		\item $(\rho^{(b)}_a-\rho^{(a)}_b)/\chi^{(1)}\in \Ch(S_{ab-1})$.
	\end{enumerate}
\end{conjecture}
In other words, we conjecture that the two virtual characters in (i) and (ii) are in fact genuine characters of $S_{ab-1}$, i.e.~the integer linear combinations of irreducible characters only have non-negative coefficients.

Conjecture~\ref{conj:18} is motivated by Foulkes' Conjecture, which in the present notation predicts that $\rho_a^{(b)}-\rho_b^{(a)}\in\Ch(S_{ab})$. We also write this as $\rho_a^{(b)}\ge\rho_b^{(a)}$, viewed in the representation ring of $S_{ab}$. %, i.e.~the Grothendieck group of the category of finite-dim reps of $S_{ab}$. 
Indeed, suppose $a<b$. %if $a=b$ then part (i) reduces to a special case of Foulkes' Conjecture while (ii) is trivially true
Then part (ii) follows from part (i) assuming only smaller cases of Foulkes' Conjecture: assuming $\rho^{(b-1)}_a\ge\rho^{(a)}_{b-1}$, substituting into $(\rho^{(b)}_a-\rho^{(a)}_b)/\chi^{(1)} = \rho^{(b-1)}_a\boxtimes\chi^{(a-1)} - \rho^{(a-1)}_b\boxtimes\chi^{(b-1)}$ (from Lemma~\ref{lem:16.6}) then gives $(\rho^{(b)}_a-\rho^{(a)}_b)/\chi^{(1)} \ge \rho^{(a)}_{b-1}\boxtimes\chi^{(a-1)} - \rho^{(a-1)}_b\boxtimes\chi^{(b-1)}$.

\bigskip
%----------------------------------------------
\section{Applications to Sylow branching coefficients}\label{sec:13}
For the remainder of this article, we fix $p=2$ and again consider Sylow branching coefficients $Z^\lambda$ for the prime 2.
In this section, we present several applications of the results on plethysms from Section~\ref{sec:pre-recursive-formula} as well as our main theorems in Section~\ref{sec:14} to the computation of Sylow branching coefficients.
In particular, we make use of the connection between plethysms and Sylow branching coefficients via various wreath product groups: plethysms can be used to describe character restrictions from $S_{mn}$ to $S_m\wr S_n$, while the Sylow 2-subgroup $P_{mn}$ of $S_{mn}$ is isomorphic to $P_m\wr P_n$ whenever $m$ is a power of 2. (Again, we recall Notation~\ref{not:convention} and Remark~\ref{rem:convention} regarding wreath products involving $P_n$.)

\medskip

We first record a simplification of Theorem A when $m=2$. By observing that $a^\phi_{\theta,(1)}=\delta_{\phi,\theta}$ when $\phi$ and $\theta$ are partitions, substituting $m=2$ into Theorem A gives %simplifies in the case $m=2$ to the following:
\begin{equation}\label{eqn:14.4}
	a^\mu_{\lambda',(2)} = \sum_{i=0}^k (-1)^{k+i} \sum_{\substack{\alpha\vdash k+i\\\beta\vdash i}} \left(\sum_{\sigma\vdash 2i}c^\alpha_{\sigma,(k-i)}\cdot a^\sigma_{\beta',(2)}\right)\cdot \left(\sum_{\tau\vdash n-i} c^{\hat{\mu}}_{\tau,\alpha}\cdot c^\lambda_{\tau,\beta}\right).
\end{equation}
In particular,
\begin{itemize}
	\item When $k=0$, \eqref{eqn:14.4} simplifies to %As remarked in Remark~\ref{rem:10}%(and recalling that Theorem~\ref{thm:4.4.3} describes the $k=0$ case of Theorem A)
	%, the case of $k=0$ is precisely Corollary~\ref{cor:10.1,2,4}(i): 
	$a^\mu_{\lambda',(2)} = c^{\hat{\mu}}_{\lambda,\emptyset} = \delta_{\hat{\mu},\lambda}$ (cf.~Corollary~\ref{cor:10.1,2,4}(i) below).
	\item When $k=1$, \eqref{eqn:14.4} simplifies to
	\begin{small}
		\[ a^\mu_{\lambda',(2)} = \sum_{\tau\vdash n-1}c^{\hat{\mu}}_{\tau,(2)}\cdot c^\lambda_{\tau,(1)} - c^{\hat{\mu}}_{\lambda,(1)}. \]
	\end{small}
	\item When $k=2$, \eqref{eqn:14.4} simplifies to
	\begin{footnotesize}
		\[ a^\mu_{\lambda',(2)} = \sum_{\tau\vdash n-2}\left( c^{\hat{\mu}}_{\tau,(4)}\cdot c^\lambda_{\tau,(1^2)} + c^{\hat{\mu}}_{\tau,(3,1)}\cdot c^\lambda_{\tau,(2)} + c^{\hat{\mu}}_{\tau,(2,2)}\cdot c^\lambda_{\tau,(1^2)} \right)
		- \sum_{\tau\vdash n-1}\left( c^{\hat{\mu}}_{\tau,(3)}\cdot c^\lambda_{\tau,(1)} + c^{\hat{\mu}}_{\tau,(2,1)}\cdot c^\lambda_{\tau,(1)} \right) + c^{\hat{\mu}}_{\lambda,(2)}.\]
	\end{footnotesize}
\end{itemize}

\medskip

\subsection{Isotypical deflations}
Understanding isotypical deflations allows us to directly express certain Sylow branching coefficients in terms of those corresponding to smaller partitions.
\begin{lemma}\label{lem:isotypic}
	Fix $n\in\N$ and let $\mu\vdash 2n$. Then
	\begin{enumerate}[label=\textup{(\roman*)}]
		\item $Z^\mu=\sum_{\gamma\vdash n} a^\mu_{\gamma,(2)}\cdot Z^\gamma$.
		\item Suppose $\delta^\mu$ is isotypical, i.e.~$\delta^\mu=a\cdot\chi^\lambda$ for some $a\in\N$ and $\lambda\vdash n$. Then $a=a^\mu_{\lambda,(2)}$ and $Z^\mu=aZ^\lambda$. In particular, if $Z^\lambda=0$ then $Z^\mu=0$.
	\end{enumerate}
\end{lemma}

\begin{proof}
	%That $a=a^\mu_{\lambda,(2)}$ follows from the definition of deflation. To see that $Z^\mu=aZ^\lambda$, let $H:=P_2=S_2$ and note $\triv_{P_{2n}}\down_{H^{\times n}} = (\triv_H)^n$. Also $\Irr(S_2\wr S_n\mid (\triv_H)^n)=\{\cX(\triv_H;\chi^\gamma) \mid \gamma\vdash n \}$, which together with $\delta^\mu$ being isotypical gives the second equality in the following:
	%\[ Z^\mu = \langle \chi^\mu\down^{S_{2n}}_{S_2\wr S_n}\down_{P_2\wr P_n}, \triv_{P_{2n}}\rangle = \langle a\cdot\cX(\triv_H;\chi^\lambda)\down_{P_2\wr P_n}, \triv_{P_{2n}}\rangle = a\langle \cX(\triv_H;\chi^\lambda\down_{P_n}), \cX(\triv_H;\triv_{P_n})\rangle = aZ^\lambda. \]
	\noindent\textbf{(i)} Let $H:=P_2=S_2$ and note that $\triv_{P_{2n}}\down_{H^{\times n}} = (\triv_H)^n$ and $\Irr(S_2\wr S_n\mid (\triv_H)^n)=\{\cX(\triv_H;\chi^\gamma) \mid \gamma\vdash n \}$. Hence
	\[ Z^\mu = \langle \chi^\mu\down^{S_{2n}}_{S_2\wr S_n}\down_{P_2\wr P_n}, \triv_{P_{2n}}\rangle = \sum_{\gamma\vdash n} a^\mu_{\gamma,(2)}\cdot \langle \cX(\triv_H;\chi^\gamma\down^{S_n}_{P_n}), \cX(\triv_H;\triv_{P_n})\rangle = \sum_{\gamma\vdash n} a^\mu_{\gamma,(2)}\cdot Z^\gamma. 
	\]
	\noindent\textbf{(ii)} If $\delta^\mu$ is isotypical then $a^\mu_{\gamma,(2)}=0$ whenever $\gamma\ne\lambda$, and the assertions follow immediately from (i).
\end{proof}

\begin{corollary}\label{cor:10.1,2,4}
	Fix $n\in\N$. For the following partitions $\mu\vdash 2n$, the deflation $\delta^\mu$ (with respect to $S_n$) is irreducible and given as follows:
	\[ \begin{array}{rlp{0.3cm}l}
		\textup{(i)} & l(\mu)=n: && \delta^\mu=\chi^\lambda\text{ where } \lambda = (\mu-(1^n))';\\
		\textup{(ii)} & \mu=(2n-\ell,1^{\ell}),\ 0\le\ell\le n-1: && \delta^\mu = \chi^{(n-\ell,1^{\ell})};\\
		\textup{(iii)} & \mu\subseteq(3^n): && \delta^\mu=\chi^{\lambda'}\text{ where }\lambda=\square_{3,n}(\mu).
	\end{array} \]
\end{corollary}

\begin{proof}
	\noindent\textbf{(i)} %Let $\lambda\vdash n$ be such that $\mu=(n,\lambda)'$. Choose $k\in\N$ such that $k\ge n-1$. % so $\lambda'\subseteq\big((1+k)^n\big)$. 
	%Letting $\nu:=\square_{1+k,n}(\lambda)$, then by Proposition~\ref{prop:12.4} we have that
	%\[ \chi^\lambda=\delta^\lambda=s_1\cdot\delta^{\nu},\quad \text{where }s_1:=\begin{cases} \triv_{S_n} & \text{if $k$ is odd},\\ \sgn_{S_n} & \text{if $k$ is even}.\end{cases} \]
	%Also by Proposition~\ref{prop:12.4}, %\square_{1+k,n}(\lambda'}) \subseteq \big((2+k)^n\big)
	%\[\delta^\nu = s_2\cdot\ \delta^{\square_{2+k,n}(\nu)}=s_2\cdot\ \delta^\mu,\quad \text{where }s_2:=\begin{cases} \sgn_{S_n} & \text{if $k$ is odd},\\ \triv_{S_n} & \text{if $k$ is even}.\end{cases}\]
	%Hence $\chi^\lambda=\sgn_{S_n}\cdot\ \delta^{\mu}$ and so $\delta^\mu=\chi^{\lambda'}$.
	This is precisely the case of $m=1$ in Theorem~\ref{thm:4.4.3}.
	
	\medskip
	
	\noindent\textbf{(ii)} Let $H(j):=(2n-j,1^j)$ for $0\le j\le 2n-1$ and let $h(j):=(n-j,1^j)$ for $0\le j\le n-1$.
	We use Theorem~\ref{thm:14.6ii} with $m=2$ and $\nu=(2n-\ell,1^{n-1})\vdash 2n+k$ where $k:=n-\ell-1$, giving $\hat{\nu}=(n+k)$. When $1\le k\le n-1$, this gives for all $\lambda\vdash n$ that
	\[ a^{\nu/(1^k)}_{\lambda',(2)} = \sum_{\sigma\vdash 2n} c^\nu_{\sigma,(1^k)}\cdot a^\sigma_{\lambda',(2)} = \langle \chi^\lambda, \sgn_{S_n}\cdot\ (\delta^{H(n-k)}+\delta^{H(n-k-1)})\rangle \]
	is equal to
	\[ \sum_{\substack{\alpha\vdash 2k\\\beta\vdash k}} a^{\alpha}_{\beta',(2)}\cdot a^{\hat{\nu}/\alpha}_{\lambda/\beta,(1)} = c^\lambda_{(n-k),(1^k)} = \langle \chi^\lambda, \chi^{h(k)}+\chi^{h(k-1)}\rangle. \]
	Using \eqref{eqn:sign}, we hence deduce
	\[ \delta^{H(n-k)}+\delta^{H(n-k-1)}=\chi^{h(n-k)}+\chi^{h(n-k-1)}. \]
	When $k=0$, we similarly obtain $\delta^{H(n-1)}=\chi^{h(n-1)}$, so inductively we deduce that $\delta^{H(\ell)}=\chi^{h(\ell)}$ for all $0\le \ell\le n-1$.
	
	\medskip
	
	\noindent\textbf{(iii)} By Proposition~\ref{prop:12.4} with $m_1=1$ and $m_2=2$, we have that $\delta^{\lambda}=\sgn_{S_n}\cdot\ \delta^\mu$ where $\lambda:=\square_{3,n}(\mu)$. Hence $\delta^\mu=\sgn_{S_n}\cdot\ \delta^\lambda= \sgn_{S_n}\cdot\ \chi^\lambda=\chi^{\lambda'}$.
\end{proof}

\begin{remark}\label{rem:10}
	\begin{itemize}
		\item Corollary~\ref{cor:10.1,2,4}(ii) describes a special case of plethysms for hook shapes, which were computed more generally in \cite{LR}. For hooks $\mu=(2n-\ell,1^\ell)$ where $\ell\ge n$, we have that $\delta^\mu=0$ by Lemma~\ref{lem:tall-pleth}.
		\item Lemma~\ref{lem:isotypic} and Corollary~\ref{cor:10.1,2,4} allow us to determine $Z^\mu$ for $\mu\vdash 2n$ such that $l(\mu)=n$ or $\mu\subseteq(3^n)$ via observing that $Z^\mu=Z^\lambda$ for some $\lambda\vdash n$. We also recover $Z^\mu$ when $\mu$ is a hook of $2n$, which agrees with Proposition~\ref{prop:hook}.
		\item Lemmas~\ref{lem:tall-pleth} and~\ref{lem:isotypic}(i) together also allow us to recover Lemma~\ref{lem:tall} in the even case. \hfill$\lozenge$
	\end{itemize}
\end{remark}

In addition to those described in Corollary~\ref{cor:10.1,2,4}, the deflation $\delta^{(5,5)}$ (with respect to $S_5$) is also irreducible. It would be interesting to classify all of the partitions $\mu\vdash 2n$ such that the deflation $\delta^\mu$ is irreducible, and more generally to investigate whether isotypical deflations are always irreducible (as is the case for all $|\mu|\le 32$).

\bigskip
%----------------------------------------------
\subsection{Inside partitions}\label{sec:inside-partition}
In this section, we consider statistics $\sn_i(\mu)$ of partitions $\mu$ involving the removal of its rows and columns, and give sufficient conditions for $Z^\mu$ to be zero in terms of these statistics. First, we describe the special cases of $\sn_1(\mu)$ (which will turn out to equal $l(\mu)$) and $\sn_2(\mu)$, before introducing $\sn_i(\mu)$ in full in Definition~\ref{def:stats}.

\begin{definition}\label{def:inside}
	Let $\mu=(\mu_1,\mu_2,\dotsc,\mu_{l(\mu)})$ be a partition.
	\begin{enumerate}[label=\textup{(\roman*)}]
		\item Define $I(\mu):=(\mu_2-1,\mu_3-1,\dotsc,\mu_{l(\mu)}-1)$, where we remove any trailing zeros. In other words, $I(\mu)$ is obtained from $\mu$ by removing its first row and column, leaving only the `inside partition'.
		
		\item Define $\tilde{\mu}:= (\mu-(1^{l(\mu)}))'$. In other words, $\tilde{\mu}$ is obtained from $\mu$ by removing its first column and then taking its conjugate.
	\end{enumerate}
\end{definition}

\begin{remark}
	\begin{enumerate}[label=\textup{(\roman*)}]
		\item The partition $I(\mu)$ equals $\mu/H_{1,1}(\mu)$, where $H_{1,1}(\mu)$ denotes the largest hook of $\mu$. See \cite{JK,Olsson} for further background on hooks and the combinatorics of partitions.
		\item Using Definition~\ref{def:inside}, the deflation in Corollary~\ref{cor:10.1,2,4}(i) may be written as $\delta^\mu = \chi^{\tilde{\mu}}$. \hfill$\lozenge$
	\end{enumerate}
\end{remark}

Let $n\in\N$ and $\mu\vdash 2n$. In Lemma~\ref{lem:tall}, we showed that if the statistic $l(\mu)$ was sufficiently large (namely $l(\mu)>n$) then $Z^\mu=0$. The next statistic we consider is $l(\mu)-|I(\mu)|$: if this is sufficiently large, meaning $l(\mu)-|I(\mu)|>\frac{n}{2}$, then we again show that $Z^\mu=0$ (Corollary~\ref{cor:13}). If $n$ is even and $l(\mu)-|I(\mu)|=\frac{n}{2}$, then we use Theorem A to compute $Z^\mu$ recursively (Corollary~\ref{cor:13.3,4,5}).

\begin{corollary}\label{cor:13}
	Fix $n\in\N$ and let $\mu\vdash 2n$.% Let $I(\mu):=\mu/ H_{1,1}(\mu)$. 
	\begin{enumerate}[label=\textup{(\roman*)}]
		\item For each $\lambda\vdash n$ such that $a^\mu_{\lambda,(2)}>0$, we have that $l(\lambda)\ge l(\mu)-|I(\mu)|$.
		\item Let $\varepsilon\in\{0,1\}$ such that $\varepsilon\equiv n$ (mod 2). If $l(\mu)-|I(\mu)|>\frac{n+\varepsilon}{2}$, then $Z^\mu=0$.
	\end{enumerate}
\end{corollary}

\begin{proof}
	\noindent\textbf{(i)} By Lemma~\ref{lem:tall-pleth}, if $l(\mu)>n$ then $a^\mu_{\lambda,(2)}=0$ for all $\lambda\vdash n$, so we may assume that $l(\mu)=n-k$ for some $k\ge 0$. Let $\hat{\mu}:=\mu-(1^{n-k})\vdash n+k$. We show that if $\lambda\vdash n$ satisfies $a^\mu_{\lambda',(2)}>0$, then $\lambda_1\ge l(\mu)-|I(\mu)|$.
	
	First note $\hat{\mu}_1=|\hat{\mu}|-|I(\mu)|=n+k-|I(\mu)|=l(\mu)-|I(\mu)|+2k$. From \eqref{eqn:14.4}, $a^\mu_{\lambda',(2)}>0$ implies that there exist $i\in\{0,\dotsc,k\}$ and $\tau\vdash n-i$ such that $c^{\hat{\mu}}_{\tau,\alpha}\cdot c^\lambda_{\tau,\beta}>0$ for some $\alpha\vdash k+i$ and $\beta\vdash i$. That is, $[\lambda]$ can be obtained by removing from $[\hat{\mu}]$ a skew shape with a Littlewood--Richardson filling of type $\alpha$ (to produce $[\tau]$), then adding on a skew shape with a Littlewood--Richardson filling of type $\beta$. In particular, $\lambda_1\ge \hat{\mu}_1-|\alpha|=l(\mu)-|I(\mu)|+2k-k-i\ge l(\mu)-|I(\mu)|$.
	
	\noindent\textbf{(ii)} This follows from part (i) of the present corollary, Lemma~\ref{lem:tall} and Lemma~\ref{lem:isotypic}(i).
\end{proof}

\begin{proposition}\label{prop:13.2}
	Let $n\in\N$ be even and $\mu\vdash 2n$. % and $\gamma:=\mu/ H_{1,1}(\mu)$. 
	Suppose $l(\mu)-|I(\mu)|=\frac{n}{2}$ and $l(\mu)=n-k$ for some $k\in\N_0$. Then for each $\lambda\vdash n$ such that $l(\lambda)=\frac{n}{2}$, 
	\[ a^\mu_{\lambda,(2)} = c^{\lambda-(1^{n/2})}_{I(\mu)',(k)}. \]
\end{proposition}

\begin{proof}
	Let $\hat{\mu}:=\mu-(1^{n-k})=(\frac{n}{2}+2k,I(\mu))$. By Theorem A,
	\[ a^\mu_{\lambda,(2)} = \sum_{i=0}^k (-1)^{k+i}\cdot \sum_{\substack{\alpha\vdash k+i\\\beta\vdash i}} a^{\alpha/(k-i)}_{\beta',(2)}\cdot a^{\hat{\mu}/\alpha}_{\lambda'/\beta,(1)} = \sum_{i=0}^k (-1)^{k+i}\cdot \sum_{\substack{\alpha\vdash k+i\\\beta\vdash i}} a^{\alpha/(k-i)}_{\beta',(2)}\cdot\sum_{\varepsilon\vdash n-i}c^{\lambda'}_{\varepsilon,\beta}\cdot c^{\hat{\mu}}_{\varepsilon,\alpha}. \]
	Now if $c^{\lambda'}_{\varepsilon,\beta}>0$, then $\varepsilon_1\le\lambda'_1=\frac{n}{2}$. On the other hand, $c^{\hat{\mu}}_{\varepsilon,\alpha}>0$ implies $\frac{n}{2}+2k=\hat{\mu}_1\le\varepsilon_1+\alpha_1$. Since $\alpha\vdash k+i\le 2k$, then $c^{\lambda'}_{\varepsilon,\beta}\cdot c^{\hat{\mu}}_{\varepsilon,\alpha}>0$ only if $\varepsilon_1=\frac{n}{2}$ and $\alpha_1=2k$, i.e.~$i=k$ and $\alpha=(2k)$. Thus
	\[ a^\mu_{\lambda,(2)} = \sum_{\beta\vdash k}a^{(2k)}_{\beta',(2)}\cdot a^{\hat{\mu}/(2k)}_{\lambda'/\beta,(1)} = a^{\hat{\mu}/(2k)}_{\lambda'/(1^k),(1)} = \sum_{\varepsilon\vdash n-k} c^{\lambda'}_{\varepsilon,(1^k)}\cdot c^{\hat{\mu}}_{\varepsilon,(2k)} = c^{\lambda'}_{(\frac{n}{2},I(\mu)),(1^k)}, \]
	where the final equality holds since we must have $\varepsilon_1=\frac{n}{2}$ and $\hat{\mu}-(2k)=(\frac{n}{2},I(\mu))$. Finally, since $\lambda'_1=\frac{n}{2}$,
	\[ a^\mu_{\lambda,(2)} = c^{\lambda'}_{(\frac{n}{2},I(\mu)),(1^k)} = c^{(\lambda'_2,\lambda'_3,\dotsc)}_{I(\mu),(1^k)} = c^{\lambda-(1^{n/2})}_{I(\mu)',(k)} \]
	as desired.
\end{proof}

\begin{corollary}\label{cor:13.3,4,5}
	Let $n$, $\mu$, %$\gamma$ 
	and $k$ be as defined in Proposition~\ref{prop:13.2}. Then
	\[ Z^\mu=\sum_{\nu\vdash \frac{n}{2}} c^\nu_{I(\mu),(1^k)}\cdot Z^\nu. \]
	In particular, if $k>\lceil\frac{n}{4}\rceil$, then $Z^\mu=0$. Moreover,
	\begin{enumerate}[label=\textup{(\roman*)}]
		\item Suppose $I(\mu)=(1^{\frac{n}{2}-k})$. Then
		\[ Z^\mu = \begin{cases}
			1 & \text{if }k\in\{\lceil\frac{n}{4}\rceil,\lfloor\frac{n}{4}\rfloor\},\\ 0 & \text{otherwise}.
		\end{cases}\]
		
		\item Suppose $I(\mu)=(\frac{n}{2}-k)$. Then $Z^\mu=\binom{b}{k}$, where $b$ is the number of digits in the binary expansion of $\frac{n}{2}$ (i.e.~$\frac{n}{2}=2^{n_1}+\cdots+2^{n_b}$ for some $n_1>\cdots>n_b\ge 0$).
	\end{enumerate}
\end{corollary}

\begin{proof}
	By Lemma~\ref{lem:isotypic}(i), %$Z^\mu=\sum_{\lambda\vdash n}a^\mu_{\lambda,(2)}\cdot Z^\lambda$
	Lemma~\ref{lem:tall} % $Z^\lambda=0$ if $l(\lambda)>n/2$
	and Corollary~\ref{cor:13}, %$a^\mu_{\lambda,(2)}>0\implies l(\lambda)\ge l(\mu)-|I(\mu)|$, and by assumption $l(\mu)-|I(\mu)|=\frac{n}{2}$
	\[ Z^\mu = \sum_{\substack{\lambda\vdash n\\l(\lambda)=\frac{n}{2}}} a^\mu_{\lambda,(2)}\cdot Z^\lambda. \]
	Writing $\lambda\vdash n$ with $l(\lambda)=\tfrac{n}{2}$ as $\lambda=(1^{\tfrac{n}{2}})+\nu'$ for some $\nu\vdash\tfrac{n}{2}$, we find by Proposition~\ref{prop:13.2} that 
	\[ Z^\mu = \sum_{\nu\vdash\tfrac{n}{2}} c^{\nu'}_{I(\mu)',(k)}\cdot Z^{(1^{\frac{n}{2}})+\nu'}. \]
	But $\delta^\lambda=\chi^\nu$ by Corollary~\ref{cor:10.1,2,4}(i), and so $Z^\lambda=Z^\nu$ by Lemma~\ref{lem:isotypic}(ii). Combining with the well known property $c^{\gamma'}_{\alpha',\beta'} = c^\gamma_{\alpha,\beta}$ of Littlewood--Richardson coefficients, we therefore obtain
	\[ Z^\mu = \sum_{\nu\vdash\tfrac{n}{2}} c^{\nu}_{I(\mu),(1^k)}\cdot Z^{\nu}. \]
	We note that $Z^\nu=0$ if $l(\nu)>\lceil\frac{n}{4}\rceil$ by Lemma~\ref{lem:tall}, while $c^\nu_{I(\mu),(1^k)}>0$ implies $l(\nu)>k$. It follows that $Z^\mu=0$ if $k>\lceil\frac{n}{4}\rceil$.
	
	\noindent\textbf{(i)} If $I(\mu)_1=1$ then $c^\nu_{I(\mu),(1^k)}>0$ only if $\nu_1\le 2$. The assertion then follows from Lemma~\ref{lem:2col}.
	
	\noindent\textbf{(ii)} If $I(\mu)_2=0$ then $c^\nu_{I(\mu),(1^k)}>0$ only if $\nu$ is a hook. The assertion then follows from Proposition~\ref{prop:hook}.
\end{proof}

In fact, we can generalise from $l(\mu)$ and $l(\mu)-|I(\mu)|$ to a collection of statistics $\sn_i(\mu)$ as follows.

\begin{definition}\label{def:stats}
	\begin{enumerate}[label=\textup{(\roman*)}]
		\item For each $i\in\N$, define $\mathsf{m}_i:=\frac{4^i+8}{6}$. %(Notice $\sm_1=2$ and $\sm_i=4(\sm_{i-1}-1)$ for all $i\ge 2$.)
		
		\item Let $\mu$ be an arbitrary partition. 
		\begin{itemize}
			\item[$\circ$] Define $\sk(\mu):=\frac{|\mu|}{2}-l(\mu)$.
			\item[$\circ$] Let $\big(\sn_i(\mu)\big)_{i\in\N}$ be recursively defined by $\sn_0(\mu)=\frac{|\mu|}{2}$ and $\sn_i(\mu)=2\sn_{i-1}(\tilde{\mu}) - \sm_i\sk(\mu)$ for all $i\in\N$, where $\tilde{\mu}$ is as in Definition~\ref{def:inside}.
		\end{itemize}
	\end{enumerate}
\end{definition}

For example, we note that $\sn_1(\mu)=l(\mu)$ since $|\tilde{\mu}|=|\mu|-l(\mu)$, and $\sn_2(\mu)=2\big(l(\mu)-|I(\mu)|\big)$ since $|\mu|=l(\mu)+l(\tilde{\mu})+|I(\mu)|$. The statistics $\sn_i(\mu)$ can be calculated as a weighted sum of the sizes of successive columns, rows and inside partitions as illustrated in Figure~\ref{fig:stats}.

\begin{figure}
	\centering
	\begin{tikzpicture}[scale=0.6, every node/.style={scale=0.7}]
		\begin{scope}[xshift=0cm, yshift=4.5cm]
			% mu = (5,4,4,2,1)
			\draw (-1,-1.25) node[] {$\mu:$};
			\draw (0,0) -- (2.5,0) -- (2.5,-0.5) -- (2,-0.5) -- (2,-1.5) -- (1,-1.5) -- (1,-2) -- (0.5,-2) -- (0.5,-2.5) -- (0,-2.5) -- (0,0);
		\end{scope}
		
		\begin{scope}[xshift=6cm, yshift=4.5cm]
			\draw (-1,-1.25) node[] {$2\sk(\mu):$};
			\draw (0,0) -- (2.5,0) -- (2.5,-0.5) -- (2,-0.5) -- (2,-1.5) -- (1,-1.5) -- (1,-2) -- (0.5,-2) -- (0.5,-2.5) -- (0,-2.5) -- (0,0);
			\draw (0.5,0) -- (0.5,-2);
			\draw (0.25,-1.25) node[] {$-1$};
			\draw (1.25,-0.75) node[] {1};
		\end{scope}
	
		\begin{scope}[yshift=1cm]
			\draw (-2.5,0) node[] {$\sn_i(\mu)\ = $};
			\draw (1.25,0) node[] {$\sn_{i-1}(\tilde{\mu})$};
			\draw (4.3,0) node[] {$-\ \tfrac{\sm_i}{2}\cdot 2\sk(\mu)$};
			\draw [dotted] (-4,-0.5) -- (12,-0.5);
			\draw [dotted] (-4,0.5) -- (12,0.5);
		\end{scope}
		
		\begin{scope}[xshift=0cm, yshift=0cm]
			\draw (-2.5,-1.25) node[] {$\sn_1(\mu)\ =$};
			\draw (0,0) -- (2.5,0) -- (2.5,-0.5) -- (2,-0.5) -- (2,-1.5) -- (1,-1.5) -- (1,-2) -- (0.5,-2) -- (0.5,-2.5) -- (0,-2.5) -- (0,0);
			\draw (0.5,0) -- (0.5,-2);
			\draw (0.25,-1.25) node[] {0};
			\draw (1.25,-0.75) node[] {1};
		\end{scope}
		
		\begin{scope}[xshift=4cm, yshift=0cm]
			\draw (-0.5,-1.25) node[] {$-1\ \cdot$};
			\draw (0,0) -- (2.5,0) -- (2.5,-0.5) -- (2,-0.5) -- (2,-1.5) -- (1,-1.5) -- (1,-2) -- (0.5,-2) -- (0.5,-2.5) -- (0,-2.5) -- (0,0);
			%\draw (1,-3) node[] {$\sn_1(\mu)$};
			\draw (0.5,0) -- (0.5,-2);
			\draw (0.25,-1.25) node[] {$-1$};
			\draw (1.25,-0.75) node[] {1};
		\end{scope}
		
		\begin{scope}[xshift=8cm, yshift=0cm]
			\draw (-0.5,-1.25) node[] {$=$};
			\draw (0,0) -- (2.5,0) -- (2.5,-0.5) -- (2,-0.5) -- (2,-1.5) -- (1,-1.5) -- (1,-2) -- (0.5,-2) -- (0.5,-2.5) -- (0,-2.5) -- (0,0);
			\draw (0.5,0) -- (0.5,-2);
			\draw (0.25,-1.25) node[] {1};
			\draw (1.25,-0.75) node[] {0};
		\end{scope}
		
		\begin{scope}[xshift=0cm, yshift=-3.5cm]
			\draw (-2.5,-1.25) node[] {$\sn_2(\mu)\ =$};
			\draw (-0.5,-1.25) node[] {$2\ \cdot$};
			\draw (0,0) -- (2.5,0) -- (2.5,-0.5) -- (2,-0.5) -- (2,-1.5) -- (1,-1.5) -- (1,-2) -- (0.5,-2) -- (0.5,-2.5) -- (0,-2.5) -- (0,0);
			\draw (1,-3) node[] {};
			\draw (0.5,0) -- (0.5,-2);
			\draw (0.25,-1.25) node[] {0};
			\draw (0.5,-0.5) -- (2.5,-0.5);
			\draw (1.5,-0.25) node[] {1};
			\draw (1.25,-1) node[] {0};
		\end{scope}
		
		\begin{scope}[xshift=4cm, yshift=-3.5cm]
			\draw (-0.5,-1.25) node[] {$-2\ \cdot$};
			\draw (0,0) -- (2.5,0) -- (2.5,-0.5) -- (2,-0.5) -- (2,-1.5) -- (1,-1.5) -- (1,-2) -- (0.5,-2) -- (0.5,-2.5) -- (0,-2.5) -- (0,0);
			\draw (0.5,0) -- (0.5,-2);
			\draw (0.25,-1.25) node[] {$-1$};
			\draw (1.25,-0.75) node[] {1};
		\end{scope}
	
		\begin{scope}[xshift=8cm, yshift=-3.5cm]
			\draw (-0.5,-1.25) node[] {$=$};
			\draw (0,0) -- (2.5,0) -- (2.5,-0.5) -- (2,-0.5) -- (2,-1.5) -- (1,-1.5) -- (1,-2) -- (0.5,-2) -- (0.5,-2.5) -- (0,-2.5) -- (0,0);
			\draw (0.5,0) -- (0.5,-2);
			\draw (0.25,-1.25) node[] {2};
			\draw (0.5,-0.5) -- (2.5,-0.5);
			\draw (1.5,-0.25) node[] {0};
			\draw (1.25,-1) node[] {$-2$};
		\end{scope}
		
		\begin{scope}[xshift=0cm, yshift=-7cm]
			\draw (-2.5,-1.25) node[] {$\sn_3(\mu)\ =$};
			\draw (-0.5,-1.25) node[] {$2\ \cdot$};
			\draw (0,0) -- (2.5,0) -- (2.5,-0.5) -- (2,-0.5) -- (2,-1.5) -- (1,-1.5) -- (1,-2) -- (0.5,-2) -- (0.5,-2.5) -- (0,-2.5) -- (0,0);
			\draw (0.5,0) -- (0.5,-2);
			\draw (0.25,-1.25) node[] {0};
			\draw (0.5,-0.5) -- (2.5,-0.5);
			\draw (1.5,-0.25) node[] {$2$};
			\draw (1,-0.5) -- (1,-2);
			\draw (0.75,-1.25) node[] {$0$};
			\draw (1.5,-1.1) node[] {$-2$};
		\end{scope}
	
		\begin{scope}[xshift=4cm, yshift=-7cm]
			\draw (-0.5,-1.25) node[] {$-6\ \cdot$};
			\draw (0,0) -- (2.5,0) -- (2.5,-0.5) -- (2,-0.5) -- (2,-1.5) -- (1,-1.5) -- (1,-2) -- (0.5,-2) -- (0.5,-2.5) -- (0,-2.5) -- (0,0);
			\draw (0.5,0) -- (0.5,-2);
			\draw (0.25,-1.25) node[] {$-1$};
			\draw (1.25,-0.75) node[] {1};
		\end{scope}
	
		\begin{scope}[xshift=8cm, yshift=-7cm]
			\draw (-0.5,-1.25) node[] {$=$};
			\draw (0,0) -- (2.5,0) -- (2.5,-0.5) -- (2,-0.5) -- (2,-1.5) -- (1,-1.5) -- (1,-2) -- (0.5,-2) -- (0.5,-2.5) -- (0,-2.5) -- (0,0);
			\draw (0.5,0) -- (0.5,-2);
			\draw (0.25,-1.25) node[] {6};
			\draw (0.5,-0.5) -- (2.5,-0.5);
			\draw (1.5,-0.25) node[] {$-2$};
			\draw (1,-0.5) -- (1,-2);
			\draw (0.75,-1.25) node[] {$-6$};
			\draw (1.5,-1.1) node[] {$-10$};
		\end{scope}
	
		\begin{scope}[xshift=4cm, yshift=-10cm]
			\draw (0,0) node[] {\LARGE$\vdots$};
		\end{scope}
	\end{tikzpicture}
	\caption{Visualising $\sn_i(\mu)$ as a weighted sum of sizes of various portions of the partition $\mu$; in the diagrams, the weights are illustrated inside the corresponding portions. For example, $2\sk(\mu):=|\mu|-2l(\mu)$ is illustrated in the top right diagram with $-1$ in the first column, and 1 in the remaining part of the partition, corresponding to $-1\cdot l(\mu) + (|\mu|-l(\mu))$.}
	\label{fig:stats}
\end{figure}

\begin{proposition}\label{prop:19.2}
	Let $\mu$ be a partition. Suppose $i\in\N$ is such that $2^i\mid |\mu|$. If $\sn_i(\mu)>\frac{|\mu|}{2}$, then $Z^\mu=0$.
\end{proposition}

\begin{remark}
	Since $\sn_1(\mu)=l(\mu)$, the $i=1$ case of Proposition~\ref{prop:19.2} recovers Lemma~\ref{lem:tall} when the partition has even size. Since $\sn_2(\mu) = 2\big(l(\mu)-|I(\mu)|\big)$, the $i=2$ case of Proposition~\ref{prop:19.2} recovers Corollary~\ref{cor:13} when $4\mid|\mu|$.\hfill$\lozenge$
\end{remark}
% Lemma~\ref{lem:tall} and Cor~\ref{cor:13} are slightly stronger than this as stated, since we also talk about when $2\nmid|\mu|$ or $4\nmid|\mu|$ resp.

To prove Proposition~\ref{prop:19.2}, we first describe the weighting of columns and rows illustrated in Figure~\ref{fig:stats}. 

\begin{definition}\label{def:weight}
	We define a collection of sequences $(a_i^{(1)}, a_i^{(2)}, a_i^{(3)}, \dotsc)$ indexed by $i\in\N$ as follows:
	\[ (a_1^{(j)})_{j} := (1,0,0,\dotsc);\qquad a_i^{(1)}:=\tfrac{\sm_i}{2}\ \ \forall\ i\in\N;\quad\ \text{and}\quad\ a_i^{(j)}:=2a_{i-1}^{(j-1)} - \tfrac{\sm_i}{2}\ \ \forall\ i\in\N,\ j\in\N_{\ge 2}. \]
	For each $i\in\N$, since $\sm_i\in2\Z$ then clearly $(a_i^{(j)})_j$ is an integer sequence. We also define $w_i:\N^2\to\Z$ by
	\[ w_i(x,y) = \left\{
		\begin{array}{ll}
			a_i^{(2j-1)} & \forall\ (x,j)\ \text{with}\ x\ge j,\\
			a_i^{(2j)} & \forall\ (j,y)\ \text{with}\ y\ge j+1,
		\end{array}
		\text{ for all}\ j\in\N. \right. \]
	We may view $w_i(x,y)$ as a weight on the box $(x,y)$ of a Young diagram, that is, the box in row $x$ and column $y$. As illustrated in Figure~\ref{fig:weight}, $a^{(2j-1)}_i$ is the weight of a box in column $j$ which is in a sufficiently low row, while $a^{(2j)}_i$ is the weight of a box in row $j$ in a column sufficiently far to the right.
	\begin{figure}
		\centering
		\begin{tikzpicture}[scale=0.5, every node/.style={scale=0.7}]
			%\draw [step=1,dotted] (0,0) grid (7,-5);
			
			\draw (0,-5) -- (0,0) -- (1,0) -- (1,-5);
			\draw (0.5,-1.5) node[] {$a_i^{(1)}$};
			
			\draw (7,0) -- (1,0) -- (1,-1) -- (7,-1);
			\draw (3.5,-0.5) node[] {$a_i^{(2)}$};
			
			\draw (2,-1) -- (2,-5);
			\draw (1.5,-2.5) node[] {$a_i^{(3)}$};
			
			\draw (2,-2) -- (7,-2);
			\draw (4.5,-1.5) node[] {$a_i^{(4)}$};
			
			\draw (3,-2) -- (3,-5);
			\draw (2.5,-3.5) node[] {$a_i^{(5)}$};
			
			\draw (3,-3) -- (7,-3);
			\draw (5.5,-2.5) node[] {$a_i^{(6)}$};
			
			\draw (5,-4) node[] {\LARGE$\ddots$};
		\end{tikzpicture}
		\caption{The value $w_i(x,y)$ is filled into $(x,y)\in\N^2$, viewed as the box in row $x$ and column $y$ of a Young diagram. Each vertical (resp.~horizontal) rectangular strip depicted is one box wide (resp.~tall).}
		\label{fig:weight}
	\end{figure}
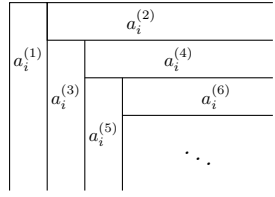
\end{definition}

The following lemma shows that we may compute $\sn_i(\mu)$ using the weights $w_i(x,y)$, whose values are independent of $\mu$ (see Figure~\ref{fig:stats} for examples when $i\in\{1,2,3\}$).
\begin{lemma}\label{lem:weights}
	For all $i\in\N$ and partitions $\mu$, we have $\sn_i(\mu) = \sum_{(x,y)\in[\mu]} w_i(x,y)$.
\end{lemma}

\begin{proof}
	We proceed by induction on $i$. Fix an arbitrary partition $\mu$.
	First, we have that $\sn_1(\mu)=l(\mu) = \sum_{(x,1)\in[\mu]}1 + \sum_{y\ge 2}\sum_{(x,y)\in[\mu]}0$. The assertion for $i=1$ then follows since $w_1(x,y)=\delta_{y,1}$.
	Next, suppose $i\ge 2$ and $\sn_{i-1}(\nu) = \sum_{(x,y)\in[\nu]} w_{i-1}(x,y)$ for all partitions $\nu$. Recalling Definitions~\ref{def:inside} and~\ref{def:stats}, then
	\begin{align*}
		\sn_i(\mu) &= 2\sn_{i-1}(\tilde{\mu}) - \sm_i\sk(\mu) = 2\sum_{(x,y)\in[\tilde{\mu}]} w_{i-1}(x,y) - \tfrac{\sm_i}{2}\cdot 2\sk(\mu)\\
		&=2\left[ \sum_{(x,1)\in[\mu]}0 + \sum_{y\ge 2}\sum_{(x,y)\in[\mu]}w_{i-1}(y-1,x)\right] - \tfrac{\sm_i}{2}\left[ \sum_{(x,1)\in[\mu]}(-1) + \sum_{y\ge 2}\sum_{(x,y)\in[\mu]}1 \right]\\
		&= \sum_{(x,1)\in[\mu]}\tfrac{\sm_i}{2} + \sum_{y\ge 2}\sum_{(x,y)\in[\mu]} \Big(2w_{i-1}(y-1,x)-\tfrac{m_i}{2}\Big).
	\end{align*}
	\begin{itemize}
		\item For $(x,1)\in\N^2$: notice $\tfrac{\sm_i}{2}=a_i^{(1)}=w_i(x,1)$.
		\item For $j\in\N$ and $(j,y)\in\N^2$ where $y\ge j+1$, we have $2w_{i-1}(y-1,j)-\tfrac{m_i}{2}=2a_{i-1}^{(2j-1)}-\tfrac{\sm_i}{2}=a_i^{(2j)}$. 
		\item For $j\in\N_{\ge 2}$ and $(x,j)\in\N^2$ where $x\ge j$, we have $2w_{i-1}(j-1,x)-\tfrac{\sm_i}{2}=2a_{i-1}^{(2j-2)}-\tfrac{\sm_i}{2}=a_i^{(2j-1)}$.
	\end{itemize}
	Hence we conclude $\sn_i(\mu) = \sum_{(x,y)\in[\mu]} w_i(x,y)$, as desired.
\end{proof}

\begin{lemma}\label{lem:19.3}
	For all $i\in\N$, the integer sequence $(a_i^{(j)})_{j\in\N}$ is weakly decreasing and eventually constant, with limit $a_i^{(\infty)} = -\sm_i+2$.
\end{lemma}

\begin{proof}
	It is clear from Definition~\ref{def:weight} and induction on $i$ that $(a_i^{(j)})_j$ is eventually constant. 
	To see that $(a_i^{(j)})_j$ is weakly decreasing, it suffices to show that $a_h^{(1)}\ge a_h^{(2)}$ for all $h\in\N$, since $(a_1^{(j)})_j$ is already weakly decreasing by definition.
	But this follows since $\frac{\sm_h}{2}\ge 2\cdot\frac{\sm_{h-1}}{2}-\frac{\sm_h}{2}$.
	Finally, we observe that $a_1^{(\infty)}=0$ and $a_i^{(\infty)} = 2a_{i-1}^{(\infty)}-\frac{\sm_i}{2}$ for all $i\ge 2$, which gives $a_i^{(\infty)}=-\sm_i+2$ by induction on $i$.
\end{proof}

We are now ready to prove Proposition~\ref{prop:19.2}: the ideas used in the proof extend those in the proof of Corollary~\ref{cor:13} (which can be viewed as the case of $i=2$).

\begin{proof}[Proof of Proposition~\ref{prop:19.2}]
	We proceed by induction on $i$, with base case $i=1$ given by Lemma~\ref{lem:tall}. Now suppose $i\ge 2$ and $2^i\mid|\mu|$. If $l(\mu)>\frac{|\mu|}{2}$ then $Z^\mu=0$ by Lemma~\ref{lem:tall}, so we may assume that $k:=\sk(\mu)\ge 0$.
	
	Suppose $\lambda\vdash\frac{|\mu|}{2}$ is such that $a^\mu_{\lambda,(2)}>0$. By Theorem A (see also \eqref{eqn:14.4}), there exist $j\in\{0,1,\dotsc,k\}$ and $\tau\vdash\frac{|\mu|}{2}-j$ such that $c^{\hat{\mu}}_{\tau,\alpha}\cdot c^{\lambda'}_{\tau,\beta}>0$ for some $\alpha\vdash k+j$ and $\beta\vdash i$. In other words, $[\lambda]$ can be obtained by removing from $[\tilde{\mu}]$ (recalling that $\tilde{\mu}$ is the conjugate of $\hat{\mu}$) a skew shape with a Littlewood--Richardson filling of type $\alpha$ (to produce $[\tau]$), then adding on a skew shape with a Littlewood--Richardson filling of type $\beta$. Hence
	\[ |\sn_{i-1}(\lambda) - \sn_{i-1}(\tilde{\mu})| \le (k+j)\cdot a_{i-1}^{(1)} - j\cdot a_{i-1}^{(\infty)} \]
	since $(a_{i-1}^{(t)})_t$ is weakly decreasing by Lemma~\ref{lem:19.3}. Moreover, since $a_{i-1}^{(1)}\ge 0$ and $a_{i-1}^{(\infty)}\le 0$, we obtain
	\[ \sn_{i-1}(\lambda) \ge \sn_{i-1}(\tilde{\mu}) -k(2a_{i-1}^{(1)} - a_{i-1}^{(\infty)}) = \sn_{i-1}(\tilde{\mu}) - \tfrac{\sm_i}{2}\cdot k =\tfrac{\sn_i(\mu)}{2}. \]
	From the assumption that $\sn_i(\mu)>\frac{|\mu|}{2}$, we obtain $\sn_{i-1}(\lambda)>\frac{|\lambda|}{2}$. Since $2^{i-1}\mid|\lambda|$, by the inductive hypothesis we deduce that $Z^\lambda=0$.
		
	Finally, using Lemma~\ref{lem:isotypic}(i) we conclude that $Z^\mu=\sum_{\lambda\vdash n}a^\mu_{\lambda,(2)}\cdot Z^\lambda=0$.
\end{proof}

\bigskip
%----------------------------------------------
\subsection{Near hook deflations}
In Example~\ref{ex:14.8} below, we use Theorem~\ref{thm:14.6ii} to compute deflations of partitions of the form $(a,2,1^b)$. First we introduce a useful piece of notation.

\begin{definition}\label{def:a-2-1b}
	For $n,l\in\N$ such that $n\ge 4$ and $2\le l\le n-2$, define $\lambda_{n,l}:=(n-l,2,1^{l-2})\vdash n$.
\end{definition}

\begin{example}\label{ex:14.8}
	Fix $n\in\N_{\ge 2}$ and suppose $\mu:=\lambda_{2n,l}$ where $2\le l\le 2n-2$. We wish to compute the deflation $\delta^\mu$; we may assume $n\ge 5$ since $\delta^\mu$ may be calculated directly for small $n$. %e.g.~$\delta^{(2,2)}=\chi^{(2)}$ by Corollary~\ref{cor:10.1,2,4}(i). 
	By Lemma~\ref{lem:tall-pleth}, $\delta^\mu=0$ if $l(\mu)>n$, so we may further assume that $l\le n$. Let $\nu:=\mu\sqcup(1^{n-l})=(2n-l,2,1^{n-2})$ and $\hat{\nu}=(2n-l-1,1)$. Applying Theorem~\ref{thm:14.6ii} with $k=n-l$ and $m=2$,
	\begin{equation}\label{14.8a}
		a^{\nu/(1^k)}_{\lambda',(2)} = \sum_{\substack{\alpha\vdash 2k\\\beta\vdash k}} a^{\alpha}_{\beta',(2)}\cdot a^{\hat{\nu}/\alpha}_{\lambda/\beta,(1)},
	\end{equation}
	for all $\lambda\vdash n$. Recall the relationship between (skew) plethysm coefficients and deflations from \eqref{eqn:plethysm-deflation-skew}. First, we deduce from \eqref{14.8a} that
	\begin{itemize}
		\item if $k=0$ then $\sgn_{S_n}\cdot\ \delta^{\lambda_{2n,n}} = \chi^{\hat{\nu}}=\chi^{(n-1,1)}$%. That is, $\delta^{\lambda_{2n,n}}=\chi^{(2,1^{n-2})}$
		; and
		\item if $k=1$ then  $\sgn_{S_n}\cdot(\delta^{\lambda_{2n,n-1}}+\delta^{\lambda_{2n,n}}+\delta^{(n+1,1^{n-1})}) =\chi^{(n)}+2\chi^{(n-1,1)}+\chi^{(n-2,1^2)}+\chi^{(n-2,2)}$, which using Corollary~\ref{cor:10.1,2,4}(ii) simplifies to $\sgn_{S_n}\cdot\ \delta^{\lambda_{2n,n-1}}=\chi^{(n-1,1)}+\chi^{(n-2,1^2)}+\chi^{(n-2,2)}$.% Equivalently, $\delta^{\lambda_{2n,n-1}} = \chi^{(2,1^{n-2})}+\chi^{(3,1^{n-3})}+\chi^{(2,2,1^{n-4})}$.
	\end{itemize}
	Now assume $k\ge 2$. We have that $a^{\nu/(1^k)}_{\lambda',(2)} = \langle \delta^{\lambda_{2n,l}} + \delta^{\lambda_{2n,l+1}} + \delta^{(2n-l,1^l)} + \delta^{(2n-l-1,1^{l+1})}, \chi^{\lambda'}\rangle$ and
	\begin{align*}
		\sum_{\substack{\alpha\vdash 2k\\\beta\vdash k}} a^{\alpha}_{\beta',(2)}\cdot a^{\hat{\nu}/\alpha}_{\lambda/\beta,(1)} &= \sum_{\beta\vdash k} \big(a^{(2k)}_{\beta',(2)}\cdot\langle \chi^{\lambda/\beta},\chi^{(l)}+\chi^{(l-1,1)}\rangle + a^{(2k-1,1)}_{\beta',(2)}\cdot\langle\chi^{\lambda/\beta},\chi^{(l)}\rangle \big)\\
		&= \sum_{\beta\vdash k} \big(a^{(2k)}_{\beta',(2)}\cdot\langle \chi^{\lambda},(\chi^{(l)}+\chi^{(l-1,1)})\boxtimes\chi^\beta\rangle + a^{(2k-1,1)}_{\beta',(2)}\cdot\langle\chi^{\lambda},\chi^{(l)}\boxtimes\chi^\beta\rangle \big)
	\end{align*}
	since $\alpha\subseteq\hat{\nu}$ only if $\alpha=(2k)$ or $(2k-1,1)$. Furthermore, $\delta^{(2k)}=\chi^{(k)}$ and $\delta^{(2k-1,1)}=\chi^{(k-1,1)}$ by Corollary~\ref{cor:10.1,2,4}(ii) (the latter is what requires $k\ge 2$). Thus we obtain
	\begin{align}\label{eqn:14.8b}
		\sgn_{S_n}\cdot(&\delta^{\lambda_{2n,l}} + \delta^{\lambda_{2n,l+1}} + \delta^{(2n-l,1^l)} + \delta^{(2n-l-1,1^{l+1})}) = (\chi^{(l)}+\chi^{(l-1,1)})\boxtimes\chi^{(1^k)} + \chi^{(l)}\boxtimes\chi^{(2,1^{k-2})}\nonumber\\
		&= \chi^{(l+2,1^{k-2})} + 2\chi^{(l+1,1^{k-1})} + 2\chi^{(l,1^k)} + \chi^{(l-1,1^{k+1})}+\chi^{(l+1,2,1^{k-3})} + 2\chi^{(l,2,1^{k-2})}+\chi^{(l-1,2,1^{k-1})}
	\end{align}
	(we omit the $\chi^{(l+1,2,1^{k-3})}$ term if $k=2$, and the $\chi^{(l-1,2,1^{k-1})}$ term if $l=2$), giving
	\[ \sgn_{S_n}\cdot\ \delta^{\lambda_{2n,l}} = \chi^{(l,1^k)}+\chi^{(l-1,1^{k+1})} + \chi^{(l,2,1^{k-2})} + \chi^{(l-1,2,1^{k-1})}\qquad\forall\ 3\le l\le n-2, \]
	and $\sgn_{S_n}\cdot\ \delta^{\lambda_{2n,2}} = \chi^{(1^n)}+\chi^{(2,1^{n-2})}+\chi^{(2,2,1^{n-4})}$. Hence	
	%Finally, in all cases we obtain $\delta^{\lambda_{2n,l}} = \sgn_{S_n}\cdot(\sgn_{S_n}\cdot\ \delta^{\lambda_{2n,l}})$ by applying $\chi^{\lambda'}=\sgn_{S_n}\cdot\chi^\lambda$.
	\[ \delta^{\lambda_{2n,l}} = \begin{cases}
		\chi^{(2,1^{n-2})} & \text{ if } l=n,\\
		\chi^{(2,1^{n-2})} + \chi^{(3,1^{n-3})} + \chi^{(2,2,1^{n-4})} & \text{ if } l=n-1,\\
		\chi^{(n-l+1,1^{l-1})} + \chi^{(n-l+2,1^{l-2})} + \chi^{(n-l,2,1^{l-2})} + \chi^{(n-l+1,2,1^{l-3})} & \text{ if }3\le l\le n-2,\\
		\chi^{(n)} + \chi^{(n-1,1)} + \chi^{(n-2,2)} & \text{ if }l=2.
	\end{cases} \]
	\hfill$\lozenge$
\end{example}

Using Theorem~\ref{thm:14.6ii} and Example~\ref{ex:14.8} we are able to give an alternative method for calculating Sylow branching coefficients for the partitions $\lambda_{2^r,l}=(2^r-l,2,1^{l-2})$ (cf.~Lemma~\ref{lem:a-2-1b}).

\begin{corollary}\label{cor:14.9}
	Let $r,l\ge 2$ be natural numbers with $l\le 2^r-2$. Then $Z^{(2^r-l,2,1^{l-2})} = \binom{r-1}{l-1}$.
\end{corollary}

\begin{proof}
	The assertion holds for small $r$ by direct computation, so now assume $r\ge 3$ and consider $\mu=\lambda_{2^{r+1},l}$ for some $2\le l\le 2^{r+1}-2$. By Lemma~\ref{lem:isotypic}(i), $Z^\mu = \sum_{\gamma\vdash n}\langle\delta^\mu,\chi^\gamma\rangle\cdot Z^\gamma$. Therefore from Example~\ref{ex:14.8} and Proposition~\ref{prop:hook},
	\[ Z^{\mu} = \begin{cases}
		0 & \text{ if }l\ge 2^r,\\
		Z^{\lambda_{2^r,2^r-2}} & \text{ if }l=2^r-1,\\
		Z^{\lambda_{2^r,l}}+Z^{\lambda_{2^r,l-1}} & \text{ if }3\le l\le 2^r-2,\\
		1+Z^{\lambda_{2^r,2}} & \text{ if }l=2.
	\end{cases} \]
	By the inductive hypothesis, we obtain $Z^{\lambda_{2^{r+1},l}} = \binom{r}{l-1}$ in all cases (noting that $\binom{r}{l-1}=0$ if $l-1>r$).
\end{proof}

\begin{remark}\label{rem:general-a-2-1b}
	We generalise some of the ideas from the case of $n=2^r$ in Corollary~\ref{cor:14.9} to arbitrary $n\in\N$.
	
	Let $n\in\N$ with $n\ge 4$, and suppose $n$ has $t$ digits in its binary expansion (i.e.~$n=2^{n_1}+\cdots+2^{n_t}$ for some $n_1>\cdots>n_t\ge 0$). Let $2\le l\le 2n-2$ and set $\mu=\lambda_{2n,l}$. If $l>n$ then $Z^\mu=0$ from Lemma~\ref{lem:tall}. By Example~\ref{ex:14.8}, Proposition~\ref{prop:hook} and Lemma~\ref{lem:isotypic}(i),
	\[ Z^\mu = \begin{cases}
		\binom{t-1}{n-2} & \text{ if }l=n,\\
		\binom{t-1}{n-2}+\binom{t-1}{n-3} + Z^{\lambda_{n,n-2}} & \text{ if }l=n-1,\\
		\binom{t-1}{l-1}+\binom{t-1}{l-2} + Z^{\lambda_{n,l}} + Z^{\lambda_{n,l-1}} & \text{ if }3\le l\le n-2,\\
		\binom{t-1}{1}+\binom{t-1}{0} + Z^{\lambda_{n,2}} & \text{ if }l=2.
	\end{cases} \]
	Notice $\binom{t-1}{n-2}=\binom{t}{n-1}=0$ since $n\ge 2^0+2^1+\cdots+2^{t-1}=2^t-1$, and $t\le 2^t-2$ for all $t\ge 2$ with equality only at $t=2$ (but $n\ge 4$ by assumption). Hence
	\[ Z^{\lambda_{2n,l}} = \binom{t}{l-1} + Z^{\lambda_{n,l}} + Z^{\lambda_{n,l-1}} \]
	for all $2\le l\le 2n-2$, where we set $Z^{\lambda_{n,l}}:=0$ if $l>n-2$. \hfill$\lozenge$ 
\end{remark}

\begin{example}\label{ex:32-2}
	Using the results in Section~\ref{sec:13}, we extend the description of those partitions $\mu\vdash 32$ such that $Z^\mu=0$ begun in Example~\ref{ex:32-1}. Recall %the notation $I(\mu)$ from Definition~\ref{def:inside}, and 
	that $|\{\mu\vdash 32\mid Z^\mu=0\}|=879$. %In the table below, $\gamma:=\mu/H_{1,1}(\mu)$, as defined in Section~\ref{sec:inside-partition}.
	
	\bigskip
	
	\begin{small}
		\begin{tabular}{p{4.8cm}|p{0.5cm}p{5.2cm}|p{3.1cm}}
			\textit{Property} & \multicolumn{2}{l|}{\textit{\# of such $\mu\vdash 32$}} & \textit{$Z^\mu=0$} \\
			\hline
			%$l(\mu)>16$ & 684 & & from Lemma~\ref{lem:tall}\\
			%$\mu$ has $\le 2$ columns & 16 & i.e.~$\mu=(2^a,1^{32-2a})$ for $0\le a\le 15$ & from Lemma~\ref{lem:2col}\\
			%$\mu$ is a non-trivial hook & 31 & i.e.~$\mu=(32-l,1^l)$ for $1\le l\le 31$ & from Proposition~\ref{prop:hook}\\
			%$\mu$ is of form $(a,2,1^b)$ & 25 & for $2\le a\le 26$ & from Lemma~\ref{lem:a-2-1b}\\[5pt]
			%\hdashline
			$l(\mu)=16$ & 77 & $\mu=(16,\lambda)'$ s.t.~$Z^\lambda=0$ & from Corollary~\ref{cor:10.1,2,4}(i)\\
			$\mu\subseteq(3^{16})$ & 2 & $\mu=(3,3,2^{13})$, $(3,2^{14},1)$ & from Corollary~\ref{cor:10.1,2,4}(iii)\\
			%$l(\mu)-|I(\mu)|>8$ & 640 & & from Corollary~\ref{cor:13}\\
			%$l(\mu)-|I(\mu)|=8$, $k=16-l(\mu)\ge 0$: & & &\\
			$\sn_2(\mu)=16$ and $\sk(\mu)\ge 0$: & & &\\
			$\quad k>4$ & 7 & & from Corollary~\ref{cor:13.3,4,5}\\
			$\quad I(\mu)=(1^{8-k})$ & 8 & $\mu=(25-2a,2^a,1^7)$, $0\le a\le 8$, $a\ne 4$ & from Corollary~\ref{cor:13.3,4,5}(i) \\
			$\quad I(\mu)=(8-k)$ & 6 & $\mu=(27-2a,a,1^{a+5})$, $2\le a\le 7$ & from Corollary~\ref{cor:13.3,4,5}(ii) \\
			$\sn_i(\mu)>16$:$\quad i=1, 2, 3, 4, 5$ & \multicolumn{2}{l|}{$684,\ 640,\ 702,\ 724,\ 734\ $} & from Proposition~\ref{prop:19.2}\\
		\end{tabular}
	\end{small}
	
	\bigskip
	
	\noindent In Example~\ref{ex:32-1}, we had identified 710 out of the 879 partitions %$\mu\vdash 32$ such that $Z^\mu=0$ 
	using results from Section~\ref{sec:sbc}. %; these correspond to those rows of the table lying above the dashed line. 
	Together with the results from Section~\ref{sec:13} listed in the above table, in total we are able to identify 868 out of the 879 partitions $\mu\vdash 32$ such that $Z^\mu$ equals zero\footnote{The eleven remaining partitions are $(23,2,2,1^5)$, $(22,3,1^7)$, $(22,2,2,1^6)$, $(20,4,1^8)$, $(17,4,2,1^9)$, $(17,3,2,2,1^8)$, $(13,4,2^3,1^9)$, $(13,3^3,1^{10})$, $(11,2^8,1^5)$, $(10,9,1^{13})$ and $(8,2^{10},1^4)$.}. \hfill$\lozenge$
\end{example}

\begin{example}
	The proportion of Sylow branching coefficients of $S_{2^k}$ for the prime 2 which have value zero is tabulated for small $k$ below.
	\[ \begin{array}{c|c|c|c}
		n & |\cP(n)| & |\{\mu\vdash n\mid Z^\mu=0\}| & \frac{|\{\mu\vdash n\mid Z^\mu=0\}|}{|\cP(n)|}\\
		\hline
		4 & 5 & 3 & 0.6\\
		8 & 22 & 15 & 0.682\\
		16 & 231 & 77 & 0.333\\
		32 & 8349 & 879 & 0.105\\
		64 & 1741630 & 38531 & 0.022\\
	\end{array} \]
	
	For comparison, we also investigate $\mu\vdash 64$ such that $Z^\mu=0$. In particular, $|\cP(64)|=1741630$ but $|\{\mu\vdash 64\mid Z^\mu=0\}|=38531$, and we are able to explain 38386 of these (leaving 145) using our results as follows:
	
	\bigskip
	
	\begin{small}
		\begin{tabular}{p{4.8cm}|p{6cm}|p{3.1cm}}
			\textit{Property} & \textit{\# of such $\mu\vdash 64$} & \textit{$Z^\mu=0$} \\
			\hline
			$l(\mu)=32$ & 879 & from Corollary~\ref{cor:10.1,2,4}(i)\\
			$\mu\subseteq(3^{32})$ & 2 & from Corollary~\ref{cor:10.1,2,4}(iii)\\
			$\mu$ is a non-trivial hook & 63 & from Proposition~\ref{prop:hook}\\
			$\mu=(64-i,2,1^{i-2})$, $7\le i\le 62$ & 56 & \\
			$\sn_2(\mu)=32$ and $\sk(\mu)\ge 0$: & & from Corollary~\ref{cor:13.3,4,5}\\
			$\quad \sk(\mu)>8$ & 45 & \\
			$\quad I(\mu)=(1^{|I(\mu)|})$ & 16 & \\
			$\quad I(\mu)=(|I(\mu)|)$ & 14 & \\
			$\sn_i(\mu)>32$:$\quad i=1, 2, 3, 4, 5, 6$ & $35471,\ 21751,\ 22216,\ 22937,\ 23513,\ 23722$ & from Proposition~\ref{prop:19.2}\\
		\end{tabular}
	\end{small}

	\bigskip

	\hfill$\lozenge$
\end{example}

Finally, we conclude with a conjecture.
\begin{conjecture}\label{conj:2c}
	Let $k\in\N$ and suppose $\lambda\vdash 2^k$. If $\lambda_{l(\lambda)}\ge 2$, then $Z^\lambda>0$ unless $\lambda=(5,3)$, or $k\ge 3$ and $\lambda=(3,3,2^{2^{k-1}-3})$.
\end{conjecture}
Indeed, we saw in Example~\ref{ex:32-2} that when $k=5$ then $Z^{(3,3,2^{13})}=0$. More generally, suppose $k\ge 3$ and let $n=2^{k-1}$ and $\mu=(3,3,2^{2^{k-1}-3})\vdash 2n$. By Corollary~\ref{cor:10.1,2,4}(iii), $\delta^\mu=\chi^\lambda$ where $\lambda=(3,1^{n-3})$ and the deflation of $\mu$ is with respect to $S_n$. Hence $Z^\mu=0$ by Lemma~\ref{lem:isotypic} and Proposition~\ref{prop:hook}, explaining the exceptions in the statement of Conjecture~\ref{conj:2c}.

\bigskip
%%%%%%%%%%%%%%%%%%%%%%%%%%%%%%%%%%%%%%%%%%%%%%%%%%%%%%%%%%%%%%%%%%%%%%%%%%%%%%%%

\end{document}